\documentclass[english,10pt]{amsart}
\usepackage{amsmath}
\usepackage{amsthm}
\usepackage{amssymb}
\usepackage{amsfonts}
\usepackage{latexsym}
\usepackage{array}
\usepackage{verbatim}
\usepackage{mathptmx} %% main font style
\usepackage{graphicx} %% for pictures
\usepackage{color}%% for color text
\usepackage{mysects}  %% for section headings
\usepackage{fancyhdr} %% for top of pages
\usepackage[labelformat=simple]{caption}
\usepackage{setspace}
\usepackage{microtype}
\usepackage{bm}
\usepackage{babel} %% for cyrillic characters
\newtheorem{thm}{\bf Theorem\bf}[section]
\newtheorem{prop}[thm]{\bf Proposition\bf}
\newtheorem{cor}[thm]{\bf Corollary\bf}
\newtheorem{lem}[thm]{\bf Lemma\bf}

\newtheorem{rem}[thm]{\bf Remark\bf}

\newtheorem{Conj}[thm]{\bf Conjecture\bf}
\def\proof{{\noindent\bf Proof. }}

\def\ve {\varepsilon}
\def\d {{\rm d}}

\def\half {{\frac 12}}

\newcommand{\captionfonts}{\scriptsize}

\makeatletter  % Allow the use of @ in command names
\long\def\@makecaption#1{%
  \vskip\abovecaptionskip
  \sbox\@tempboxa{{\captionfonts {\bf  #1}}}%
  \ifdim \wd\@tempboxa >\hsize
    {\captionfonts #1\par}
  \else
    \hbox to\hsize{\hfil\box\@tempboxa\hfil}%
  \fi
  \vskip\belowcaptionskip}
\makeatother   % Cancel the effect of \makeatletter

\newcommand{\GL}{\operatorname{GL}}
\newcommand{\SL}{\operatorname{SL}}
\newcommand{\PSL}{\operatorname{PSL}}

\begin{document}
\title{ Nodal domains of Maass forms I}
\author{Amit Ghosh, Andre Reznikov and Peter Sarnak}
\date{}
\setcounter{section}{0}
\setcounter{subsection}{0}
\setcounter{footnote}{0}
\setcounter{page}{1}
%%%%%%%%%%%%%%%%%%%%%%%%%% SECTION 1 %%%%%%%%%%%%%%%%%%%%%%%%%%%%%%% 

\dedicatory{Dedicated belatedly to Dennis Hejhal on the occasion of his sixtieth birthday.}

\begin{abstract}
This paper deals with some questions that have received a lot of attention since they were raised by Hejhal and Rackner in their 1992 numerical computations of Maass forms. We establish sharp upper and lower bounds for the $L^2$-restrictions of these forms to certain curves on the modular surface. These results, together with the Lindelof Hypothesis and known subconvex $L^\infty$-bounds are applied to prove that locally the number of nodal domains of such a form goes to infinity with its eigenvalue.
\end{abstract}

\maketitle
\thispagestyle{empty}

\pagestyle{fancy}
\fancyhead{}
\fancyhead[LE]{\thepage}
\fancyhead[RO]{\thepage}
\fancyhead[CO]{\small Nodal domains of Maass forms I}
\fancyhead[CE]{\small  Amit Ghosh, Andre Reznikov and Peter Sarnak}
\renewcommand{\headrulewidth}{0pt}
\headsep = 1.0cm
\fancyfoot[C]{}
\ \vskip 1cm
\renewcommand{\thefootnote}{ } %% no footnote number
\renewcommand{\footnoterule}{{\hrule}\vspace{3.5pt}} %% long footnote line


%%%%%%%%%%%%%%%%%%%%%%%%%%%%%%%%%%%%%%%%%%%%%%%%%%%%%%%%%%%%%%%%%%%%%%%%%%%
%%%%%%%%%%%%%%%%%%%%%%%%%%%%%%%%%%%%%%%%%%%%%%%%%%%%%%%%%%%%%%%%%%%%%%%%%%%
\renewcommand{\thefootnote}{\arabic{footnote}\quad } %% restore the footnote number
\setcounter{footnote}{0} %% start of counting of footnotes

\setcounter{section}{0}\setcounter{equation}{0}
%%%%%%%%%%%%%%%%%%%%%%%%%%%%%%%%%%%%%%%%%%%%%%%%%%%%%%%%%%%%%%%%%%%%%%%%%%%
%%%%%%%%%%%%%%%%%%%%%%%%%%%%%%%%%%%%%%%%%%%%%%%%%%%%%%%%%%%%%%%%%%%%%%%%%%%
%%%%%%%%%%%%%%%%%%%%%%%%%%%%%%%%%%%%%%%%%%%%%%%%%%%%%%%%%%%%%%%%%%%%%%%%%%%%%%%%
%%%%%%%%%%%%%%%%%%%%%%%%%%%%%%%%%%%%%%%%%%%%%%%%%%%%%%%%%%%%%%%%%%%%%%%%%%%%%%%%
%pg1%
\section{Introduction.}
Cusp forms are the building blocks of the modern theory of automorphic forms. They are elusive objects and even the existence of the everywhere unramified forms on $\GL_{2}(\mathbb{Q})\backslash$ $\GL_{2}(\mathbb{A})$ that we study in this paper, is known only by indirect means (Selberg 1956 using the trace formula). The first reliable numerical computations of such forms was carried out by Hejhal and Rackner \cite{HR92}. They computed some high frequency Maass cusp forms (both even and odd) for $\Gamma = \SL_{2}(\mathbb{Z})$. Let $\mathbb{X} = \Gamma\backslash\mathbb{H}$ be the modular surface. It has an orientation reversing isometry $\sigma$ induced from the reflection of $\mathbb{H}$, given by $\sigma(x+iy) = -x + iy$. 

The even Maass cusp forms (which will be the central objects in this paper) are functions $\phi : \mathbb{H} \rightarrow \mathbb{R}$ satisfying
%\begin{equation}\label{eq:Oh1}
\begin{itemize}
\item[(i)]\quad $\Delta \phi + \lambda \phi = 0$ ,\quad $\lambda = \lambda_{\phi} > 0$ ,
\item[(ii)]\quad $\phi(\gamma z) = \phi(z) ,\quad \gamma \in \Gamma$,
\item[(iii)]\quad $\phi(\sigma z) = \phi(z)$, 
\item[(iv)]\quad $\phi \in L^{2}(\mathbb{X})$, normalized so that $\int_{\mathbb{X}}\phi^{2}(z)\ \d\text{A}(z)=1$, 
\end{itemize}
where $\d\text{A}(z) = \frac{\d x\d y}{y^{2}}$ is the hyperbolic area measure.
%\end{equation}

Such a $\phi$ is automatically a cusp form \cite{Hej76} and we will assume that it is also an eigenfunction of all the Hecke operators $T_{n}$, $n\geq 1$. The corresponding automorphic cusp forms on $\GL_{2}(\mathbb{Q})\backslash \GL_{2}(\mathbb{A})$ are the everywhere unramified forms and their $L$-functions $L(s,\phi)$ are the $\GL_{2}$ analogues of the Riemann zeta function. We will use the parameter $t_{\phi}$ to parametrize the eigenvalue by $\lambda_{\phi} = t_{\phi}^{2} + \frac{1}{4}$ and we let $n_{\phi}$ be the numbering of such $\phi$'s when ordered by $\lambda_{\phi}$. Selberg \cite{Se89} proved a Weyl law for these $\phi$'s (thus demonstrating their existence);
\begin{equation}\label{eq:Oh2}
\lambda_{\phi} \sim 24n_{\phi} ,\quad n\rightarrow \infty.
\end{equation}

%pg2%

Let $Z_{\phi} \subset \mathbb{X}$ denote the nodal line of $\phi$, that is the set $\{ z\in \mathbb{X}: \phi(z)=0\}$. It consists of a finite union of real analytic curves. The connected components of $\mathbb{X}\setminus Z_{\phi}$ are the nodal domains of $\phi$ and we denote their number by ${\rm N}(\phi)$. More generally, for any subset $\mathbb{K} \subset \mathbb{X}$ we denote by ${\rm N}^{\mathbb{K}}(\phi)$ the number of such domains that meet $\mathbb{K}$. From \cite{HR92} we have in Figure 1, a picture of the nodal lines for four such $\phi$'s with $t_{\phi} = 47.926..., 125.313..., 125.347...$ and $125.523...$ (note that $t_{\phi}$ is denoted by $R$ in \cite{HR92}).

\begin{figure}[ht]
  \begin{center}
    \includegraphics[width=4.99in]{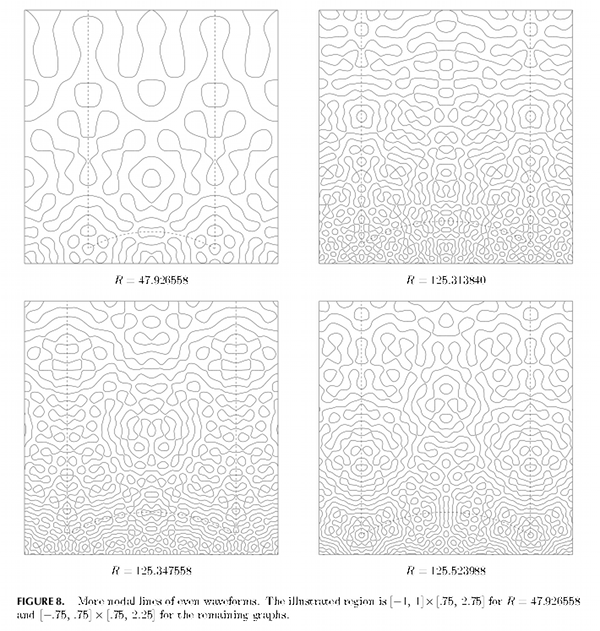}
  \end{center}
  \caption{}
\end{figure}
This paper is concerned with answering some questions that have been raised in connection with such pictures. The most basic is to determine the number of nodal domains, either in totality or in subsets of a fundamental domain. Another question, which appears particularly hard to address analytically, is whether $Z_{\phi}$ is nonsingular (and hence is self-avoiding). Using heuristic arguments relating the zero sets of random waves to an exactly solvable bond percolation model, as well as further heuristics relating eigenfunctions of quantizations of chaotic Hamiltonians to random waves, Bogomolny and Schmit \cite{BS02} suggest that in such settings ${\rm N}(\phi)$ has an asymptotic law! In our setting of even Maass cusp forms for $\mathbb{X}$, it reads \\
%pg3%
\noindent{\bf Bogomolny-Schmit Conjecture}
\[
{\rm N}(\phi)\sim \frac{2}{\pi}(3\sqrt{3} -5)n_{\phi},\quad n_{\phi}\rightarrow \infty.
\]
\par
There is also a localized version for ${\rm N}^{\mathbb{K}}(\phi)$ for $\mathbb{K}$ a subset of $\mathbb{X}$ like the windows in Figure 1 and the prediction is remarkably accurate for these four forms. Based on this and other experimentation in different settings (\cite{BS07}, \cite{NS09},\cite{Nas11},\cite{Ko12}) it seems that this prediction is probably correct, with either the above constant or one very close to it.

According to Courant's general nodal domain theorem \cite{CH53} and \eqref{eq:Oh2}, one has asymptotically
\begin{equation}\label{eq:Oh3}
{\rm N}(\phi) \leq n_{\phi} = \frac{1}{24}t_{\phi}^{2}\big{(}1 +o(1)\big{)}.
\end{equation}
However, even showing that ${\rm N}(\phi)$ (or more generally ${\rm N}^{\mathbb{K}}(\phi)$ with $\mathbb{K}$ a compact piece of $\mathbb{X}$) goes to infinity with $n_{\phi}$ is problematic. At issue is the fact that ${\rm N}(\phi)$ is a global topological invariant and hence is difficult to study from local considerations. In particular, it is not true for a general Riemannian surface that ${\rm N}(\phi)$ must increase with $n_{\phi}$, so that the problem is truly a global one (see  \cite{St25}, \cite{CH53} and \cite{Le77} where it is shown that there are infinite sequences of eigenfunctions each with few nodal domains). Our main result is a conditional proof (assuming something much weaker then a generalized Riemann Hypothesis) that ${\rm N}^{\mathbb{K}}$ goes to infinity with $n_{\phi}$.

We turn to a precise formulation of our results and methods. The isometry $\sigma$ of $\mathbb{X}$ maps the nodal domains of $\phi$ bijectively to themselves. In this way, the nodal domains are partitioned into those fixed by $\sigma$, which we call {\it inert}, and those paired off by $\sigma$ which we call {\it split} (see Figure 2 from \cite{HR92}, where the black and red ones are inert and the green ones are split).
Correspondingly we write
\begin{equation}\label{eq:Oh4}
{\rm N}(\phi) = {\rm N}_{{\rm in}}(\phi) + {\rm N}_{{\rm sp}}(\phi),
\end{equation}
where ${\rm N}_{{\rm in}}(\phi)$ is the number of inert domains and ${\rm N}_{{\rm sp}}(\phi)$, which is even, is the number of split domains. The nodal domains for $\phi$ that we will produce are all inert. According to the estimates given below, if the Bogomolny-Schmit Conjecture is true, then most of the nodal domains for $\phi$ are split (when $n_{\phi} \rightarrow \infty$); however we have no proof that there are any split nodal domains! Our analysis is based on intersecting $Z_{\phi}$ with various analytic arcs in $\mathbb{X}$. If $\beta$ is such a simple arc then clearly
\begin{equation}\label{eq:Oh5}
{\rm N}^{\beta}(\phi) \leq |Z_{\phi}\cap \beta|+1 .
\end{equation}
%pg4%
\begin{figure}[!ht]
  \begin{center}
    \includegraphics[width=3.5in]{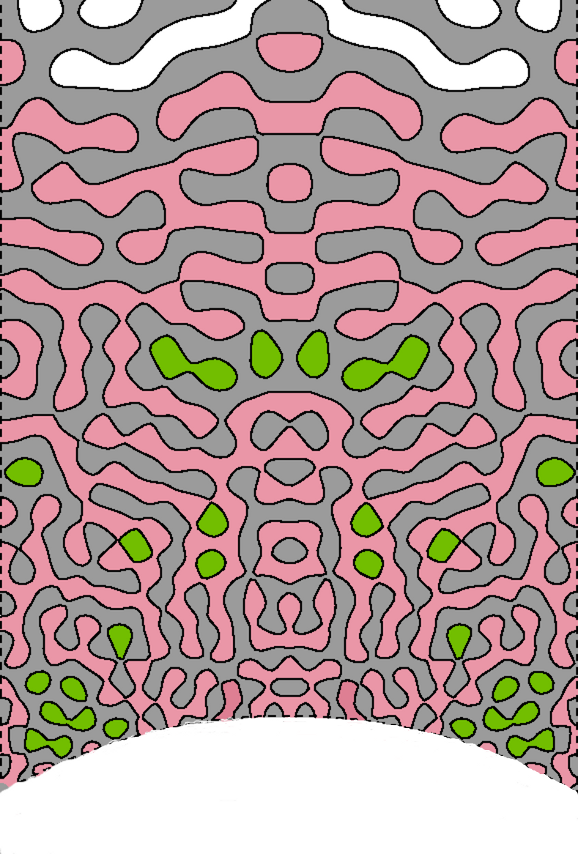}
  \end{center}
  \caption{}
\end{figure}

Denote by ${\rm K}^{\beta}(\phi)$ the number of sign changes of $\phi$ when traversing $\beta$. A key role is played by the arc $\delta = \{ z\in \mathbb{X}: \sigma(z)=z \}$ which is fixed by $\sigma$. This arc is composed of three piecewise analytic geodesics, denoted by $\delta_{1}$, $\delta_{2}$ and $\delta_{3}$ in Figure 3. A simple topological analysis (see Section 2.1) shows that
\begin{equation}\label{eq:Oh6}
1 + \frac{1}{2}{\rm K}^{\delta}(\phi) \leq {\rm N}_{{\rm in}}(\phi) \leq |Z_{\phi}\cap \delta| ,
\end{equation}
and more generally for $\beta$ a nonempty segment of $\delta$
\begin{equation}\label{eq:Oh7}
1 + \frac{1}{2}{\rm K}^{\beta}(\phi) \leq {\rm N}_{{\rm in}}^{\beta}(\phi) \leq |Z_{\phi}\cap \beta| + 1,
\end{equation}
The complexification technique of Toth and Zelditch (using a form of Jensen's Lemma) may be applied to our $\phi$'s whose normal derivative vanishes along $\delta$, to yield
\begin{equation}\label{eq:Oh8}
|Z_{\phi}\cap \delta| \ll t_{\phi}\log t_{\phi}.
\end{equation}
Thus, the number of inert nodal domains is much smaller than what is predicted by the Bogomolny-Schmit Conjecture. To give a lower bound for ${\rm N}_{{\rm in}}(\phi)$ we will use \eqref{eq:Oh6} and \eqref{eq:Oh7} and hence we seek sign changes in $\phi $ along a given curve. To this end, we examine the behavior of $\phi $ when restricted to certain curves. These restriction questions were first examined in \cite{R} and \cite{BGT07}, and sharp versions are discussed in \cite{Sa08}. Let $\beta $ be a nonempty simple analytic segment. We seek sharp upper and lower bounds for the $L^{2}$-restrictions of $\phi $. It is plausible that for any such fixed $\beta$ and large $t_{\phi}$ (recall that $\phi$ will always be $L^{2}$ normalized over $\mathbb{X}$)
\begin{equation}\label{eq:Oh9}
t_{\phi}^{-\ve} \ll_{_{\ve}} \int_{\beta}\phi^{2}(z)\ \d^{\times} z \ll_{_{\ve}} t_{\phi}^{\ve},
\end{equation}
where throughout the paper $\d^{\times} z$ denotes the hyperbolic arclength measure $\frac{\sqrt{\d x^{2} + \d y^{2}}}{y}$ for $z=x+iy$.
%pg5%
Clearly, this implies that $\beta$ is not part of the nodal line $Z_{\phi}$ for $t_{\phi}$ large enough, which itself is a very strong statement. If $\beta$ is a geodesic or horocycle segment, then one can show without difficulty that $\beta$ is not a part of $Z_{\phi}$ (see Section 2.2).

If $\beta$ is a closed horocycle or a long enough (compact) subsegment of $\delta$, we establish the sharp bounds in \eqref{eq:Oh9}\ . One of the key ingredients in the proof of the lower bound is the Quantum Unique Ergodicity (QUE) Theorem of Lindenstrauss \cite{Lin06} and Soundararajan \cite{So10}. It asserts that the restriction of $\phi^2$ to a nice $\Omega \subset \mathbb{X}$ (say a geodesic ball or more generally micro local lifts to $\Gamma\backslash\SL_{2}(\mathbb{R})$)  satisfies
\begin{equation}\label{i101}
\int_{\Omega}\phi^{2}(z)\ \d\text{A}(z) \rightarrow \frac{\text{Area}(\Omega)}{\text{Area}(\mathbb{X})} \quad \text{as}\quad t_{\phi}\rightarrow\infty\ .
\end{equation}

%Theorem 1.1
\begin{thm}\label{TheoremI1}
Let ${\mathcal C}$ be a closed horocycle in $\mathbb{X}$. Then, for $\ve >0$
\[
t_{\phi}^{-\ve} \ll_{_{\ve}} \int_{{\mathcal C}}\phi^{2}(z)\ \d^{\times} z \ll_{_{\ve}} t_{\phi}^{\ve}.
\]
\end{thm}
\begin{rem}
Using probabilistic arguments, Hejhal and Rackner (\cite{HR92}  (6.12)) conjecture that the restricted mean square of $\phi$ to a closed horocycle is asymptotic to $\frac{1}{\text{Area}(\mathbb{X})} = \frac{3}{\pi}$. This follows from the conjectured equation \eqref{aa2} in Appendix A.
\end{rem}
%Theorem 1.2
\begin{thm}\label{TheoremI2}   If $\beta$ is a long enough but fixed compact subsegment in $\delta_{1}$ or $\delta_{2}$, then for $\ve > 0$
\[
1 \ll_{_{\beta}} \int_{\beta}\phi^{2}(z)\ \d^{\times} z \leq \int_{\delta}\phi^{2}(z)\ \d^{\times} z  \ll_{_{\ve}} t_{\phi}^{\ve} .
\]
\end{thm}

In Section 6 (see Theorem \ref{Theorem5.1}) we first prove the lower bound in Theorem \ref{TheoremI2} in non-localised form, namely the inequality between the first and third terms in this Theorem. For this, little of the arithmetic features of $\mathbb{X}$ (other than QUE) are used. The localized version is technically much more involved, requiring a delicate asymptotic analysis of the off-diagonal terms, with both arithmetic and analytic input being crucial. It is in this analysis that we are forced to make $\beta$ long (see Section 6.2 and Appendix A).

The  sharp lower bounds in the restriction theorems above can be combined with strong upper bounds for \ $\int_{\mathcal{I}} \phi (z)\ \d^{\times} z$ \ for ${\mathcal I}$ an arbitrary segment of ${\mathcal C}$ or $\beta$,  and  subconvexity bounds  for $\parallel\phi\parallel _{\infty}$ on $\mathbb{X}$ (see \cite{IS95}) to give lower bounds for $|Z_{\phi}\cap {\mathcal C}|$ and $|Z_{\phi}\cap \beta|$ (see the end of Section 5 and Section 6.3).

%Theorem 1.4
\begin{thm}\label{TheoremI4} Let $\mathcal{C}$ be a fixed closed horocycle in $\mathbb{X}$. Then for $\ve >0$
\[
t_{\phi}^{\frac{1}{12} - \ve} \ll_{\ve} |Z_{\phi}\cap\mathcal{C}| \ll t_{\phi}.
\]
\end{thm}

%Remark 1.5
\begin{rem} The upper bound above can be established without recourse to arithmetic, that is it holds for any hyperbolic surface $\mathbb{Y}$ and any closed horocycle on $\mathbb{Y}$, as was shown recently by Jung \cite{Ju11}. The lower bound can be improved to $t_{\phi}^{\frac{1}{2} - \ve}$ if one were to assume that the $L^{\infty}$-norm of $\phi$ restricted to compact sets is bounded by $t_{\phi}^{\ve}$.
\end{rem}

\noindent The random wave model for the $\phi$'s predicts that \ $|Z_{\phi}\cap\mathcal{C}| \sim \frac{{\rm length}(\mathcal{C})}{\pi}t_{\phi}$ \ as \ $t_{\phi}\rightarrow \infty$ \ (see Appendix B).

Our primary interest is in sign changes of $\phi$ along $\delta$. The lower bound in the next Theorem is conditional on the Lindelof Hypothesis for the $L$-functions $L(s,\phi)$ associated to the form $\phi$. This conjecture asserts that $L(\half + it,\phi) \ll_{\ve} \big{(}(1 + |t|)t_{\phi}\big{)}^{\ve}$ and its truth follows from the corresponding Riemann Hypothesis. For our purpose, what we need is a bound of the form $L(\half + it,\phi) \ll_{\ve} (1 + |t|)^{A}t_{\phi}^{B}$ for some finite $A$ but with $B<\frac{1}{12}$. The best subconvex bounds for $L(s,\phi)$ known give $B=\frac{1}{3}$ (see \cite{Iv01},\cite{JM05}) which is far short of $\frac{1}{12}$. For this reason we simply assume the full Lindelof Hypothesis.
%pg6%
%Theorem 1.6
\begin{thm}\label{TheoremI5} Fix $\beta \subset \delta$ a sufficiently long  compact geodesic segment on $\delta_{1}$ or $\delta_{2}$ and assume the Lindelof Hypothesis for the $L$-functions $L(s,\phi)$. Then
\[
 |Z_{\phi}\cap\beta| \gg_{_{\ve}}\  t_{\phi}^{\frac{1}{12}-\ve}. 
\]
\end{thm}

%Remark 1.7
\begin{rem} As noted in \eqref{eq:Oh8}\ , general complexification techniques \cite{TZ09} yield $|Z_{\phi}\cap\beta|\ll t_{\phi}$.
\end{rem}
Theorem \ref{TheoremI5} together with \eqref{eq:Oh6} leads to our main result.

%Theorem 1.8
\begin{thm}\label{TheoremI6} With the same assumptions as in Theorem \ref{TheoremI5}
\[
{\rm N}_{{\rm in}}^{\beta}(\phi) \gg_{_{\ve}} t_{\phi}^{\frac{1}{12}-\ve} ,
\]
and in particular that ${\rm N}^{\beta}(\phi)$ goes to infinity as $n_{\phi}$ goes to infinity.
\end{thm}

Finally, we examine the behavior of $\phi$ in the shrinking regions near the cusp of $\mathbb{X}$. For $z$ in the standard fundamental domain for $\Gamma$, and when $\Im(z) > \frac{1}{2\pi}t_{\phi}$, \ $\phi$ is approximated very well by one term in its Fourier expansion. From this, one can execute a much finer analysis of $Z_{\phi}$ even into a slightly larger region away from the cusp. It appears that all the nodal domains in these regions are inert domains, see Section 4 (this being the analogue of the suggestion in \cite{GS12} that all the zeros of a holomorphic Hecke eigenform of weight $k$ in the region of the standard fundamental domain with $y \geq \sqrt{k}$ are real, that is they lie on the boundary $\delta$). In any case, we show that there are at least $\frac{1}{1000}t_{\phi}$ inert domains in the region $y > \frac{1}{100}t_{\phi}$. In particular,
\begin{equation}\label{eq:Oh10}
{\rm N}({\phi}) \geq {\rm N}_{{\rm in}}(\phi) \geq \frac{1}{1000}t_{\phi} .
\end{equation}
Other finer aspects of the geometry and topology of $Z_{\phi}$ near the cusp (here $y> \ve_{0}t_{\phi}$ with $\ve_{0}$ any positive constant) are discussed in Section 4.

Our results concerning these even Maass cusp forms can be extended to the Eisenstein series $E(z,\half +it)$ and to odd cusp forms on $\mathbb{X}$. The Eisenstein series are even with respect to the symmetry $\sigma$ and other than renormalizing the $L^{2}$-norms by restricting to a compact subset of $\mathbb{X}$ and invoking (\cite{LS95}, \cite{Ja94}) the proofs go through in the same way. The odd forms vanish identically on $\delta$ and the analysis of the nodal domains proceeds by working with the normal derivative of $\phi$ restricted to $\delta$. The analogue $L^2$-restriction lower bounds and $L^{\infty}$ upper bounds can be established with small modifications as indicated in Appendix C. In particular, one concludes that Theorem \ref{TheoremI6} is valid for the odd cusp forms as well. 

One can also extend these results on nodal domains to congruence subgroups as long as they carry reflection symmetries. For example if $\mathbb{Y}=\Gamma_{0}(q)\backslash\mathbb{H}$ is a surface of genus $g$, where $\Gamma_{0}(q)$ is the usual `Hecke' congruence group, then our reflection $\sigma(z)=-\overline{z}$ induces an isometry $\tilde{\sigma}$ of $\mathbb{Y}$. The fixed set of $\tilde{\sigma}$ consists of a finite union of geodesics running between cusps. We can diagonalize the space of newforms (the oldforms being taken care of at a lower level) with respect to $\tilde{\sigma}$, yielding $\tilde{\sigma}$-even and $\tilde{\sigma}$-odd forms $\phi$. Correspondingly, $\phi$ has $\tilde{\sigma}$ inert and split nodal domains. The main result, Theorem \ref{TheoremI6}, extends to this setting, namely if $\mathbb{Y}$ is fixed (as is $g$) and $\phi$ is as above, then the lower bound in Theorem \ref{TheoremI6} for the number of inert nodal domains for $\phi$ holds as $t_{\phi}\rightarrow \infty$ (see the remark after Theorem \ref{TheoremNodal}). In particular, the number of nodal domains for $\phi$ goes to infinity with $t_{\phi}$.

The techniques used to prove our main results rely heavily on the Fourier expansion of $\phi$ and especially the arithmetic interpretation of its Fourier coefficients, so that it is not clear if the results extend to compact arithmetic surfaces $\mathbb{Y}$. In a follow-up paper \cite{GRS} we show that indeed similar results hold for nodal domains of eigenfunctions on certain such $\mathbb{Y}$'s. Again, the key is that $\mathbb{Y}$ carry a reflection symmetry. To produce such arithmetic surfaces with a symmetry we take $\mathbb{Y}$ to be $\Gamma\backslash\mathbb{H}$, where $\Gamma$ corresponds to the unit group of an anisotropic rational quadratic form $f(x_{1},x_{2},x_{3})$ of signature $(2,1)$. To ensure that $\mathbb{Y}$ has a reflective symmetry $\tilde{\sigma}$, one requires that there be an integral root vector of $\underline{v}=(x_{1},x_{2},x_{3})$ of length $2$ (the sign is such that this gives a hyperbolic reflection). In this compact setting the proof of the critical Theorem \ref{TheoremI2} for $\tilde{\sigma}$ even eigenfunctions as well as various of the other ingredients, are deduced using more representation theoretical tools. We leave the precise statements and details to Part II.
%pg7%
\vskip 0.2in
\noindent{\small {\bf Acknowledgments.}} \\
AR and PS thank the BSF grant. All three authors have been partially supported by grant NSFDMS-0554345.\\
AR thanks the ISF Center of Excellency grant and the Veblen Fund at IAS. He was also supported by the ERC grant 291612.\\
AG thanks the Institute for Advanced Study and Princeton University for making possible visits during part of the years 2010-2012  when much of this work took place. He also acknowledges support from the Ellentuck Fund at IAS and the Vaughn Fund at OSU.\\
The software Mathematica$^{\copyright}$ was used on a dual-core PC to generate some of the images, using data from \cite{Th05}.
\vskip 0.2in

%%%%%%%%%%%%%%%%%%%%%%%%%%%%%%%%%%%%%%%%%%%%%%%%%%%%%%%%%%%%%%%%%%%%%%%%%%%%%%%
%%%%%%%%%%%%%%%%%%%%%%%%% SECTION 2                 %%%%%%%%%%%%%%%%%%%%%%%%%%%
%%%%%%%%%%%%%%%%%%%%%%%%%%%%%%%%%%%%%%%%%%%%%%%%%%%%%%%%%%%%%%%%%%%%%%%%%%%%%%%
%%%%%%%%%%%%%%%%%%%%%%%%%%%%%%%%%%%%%%%%%%%%%%%%%%%%%%%%%%%%%%%%%%%%%%%%%%%%%%%
\section{Nodal domains and crossings.} We give in this Section a discussion related to equations \eqref{eq:Oh4} to \eqref{eq:Oh7} in the introduction and a proof that segments of horocycles and geodesics cannot be contained in nodal lines. 
%%%%%%%%%%%%%%%%%%%%%%%%%%%%%%%%%%%%%%%%%%%%%%%%%%%%%%%%%%%%%%%%%%%%%%%%%%%%%%%
%%%%%%%%%%%%%%%%%%%%%%%%%%%%%%%%%%%%%%%%%%%%%%%%%%%%%%%%%%%%%%%%%%%%%%%%%%%%%%%
\subsection{Inert nodal domains and zeros.\\}\quad Let $\mathbb{X} = \Gamma\backslash\mathbb{H}$ be the modular surface and $\sigma : \mathbb{X} \rightarrow \mathbb{X}$ be the isometry of $\mathbb{X}$ (as a hyperbolic surface) induced by the reflection of $\mathbb{H}$ given by 
\[
\sigma(x+iy) = -x + iy.
\]
In what follows, $\phi(z)$ will be an even Hecke-Maass cusp form, so that 
\begin{equation}\label{eq:a1}
\phi(\sigma z)= \phi(z).
\end{equation}
Let $\delta$ be the set of points on $\mathbb{X}$ fixed by $\sigma$ so that $\delta = \delta_{1}\cup \delta_{2} \cup \delta_{3}$ as shown below (here we take the standard fundamental domain $\mathcal{F} = \mathcal{F}^{-} \cup \mathcal{F}^{+}$).
\begin{figure}[ht]
  \begin{center}
    \includegraphics[width=3.5in]{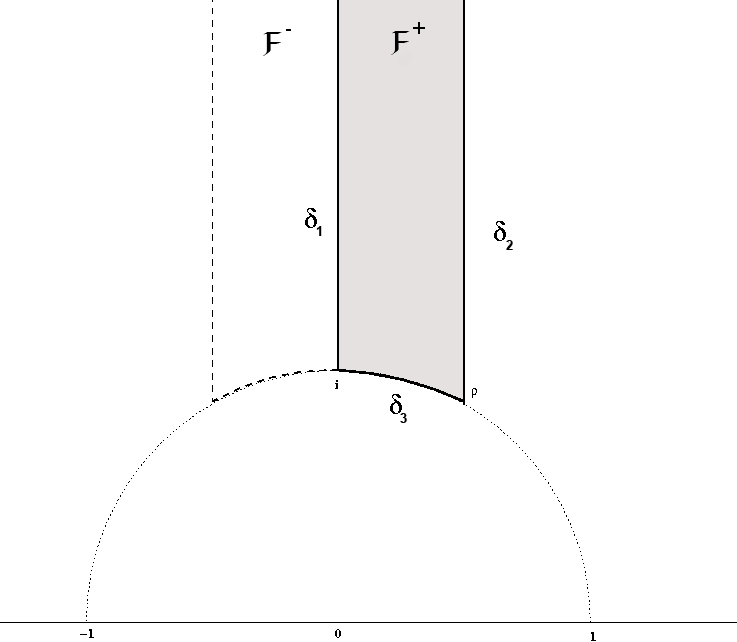}
  \end{center}
  \caption{}
\end{figure}
The nodal line of $\phi$ is the set
\[
Z_{\phi} = \{z\in \mathbb{X}: \phi(z)=0\},
\]
while the nodal domains of $\phi$ are the (finitely many) connected components of $\mathbb{X}\setminus Z_{\phi}$. We denote them by $\Omega_{1}, \Omega_{2}, ..., \Omega_{N}$ where $N=N(\phi)$ is their number. From \eqref{eq:a1}, it is clear that if $\Omega$ is a nodal domain (of $\phi$) then so is $\sigma(\Omega)$. Moreover, since the $\Omega_{j}$'s are disjoint, they fall into two types:
\begin{itemize}
\item[(1)] Inert domains $\mathcal{V}_1, ..., \mathcal{V}_R$ where $\mathcal{V}_j = \Omega_{j'}$ for some $j'$ and $\sigma(\mathcal{V}_j) = \mathcal{V}_j$;
\item[(2)] Split domains $\mathcal{W}_1, ..., \mathcal{W}_S$, $\sigma(\mathcal{W}_1), ..., \sigma(\mathcal{W}_S)$ where $\mathcal{W}_j = \Omega_{j'}$ for some $j'$ \\and $\sigma(\mathcal{W}_j)\cap \mathcal{W}_j = \emptyset$.
\end{itemize}

Thus the list of $\Omega$'s is $\mathcal{V}_1, ..., \mathcal{V}_R, \mathcal{W}_1, ..., \mathcal{W}_S, \sigma(\mathcal{W}_1), ..., \sigma(\mathcal{W}_S)$ so that
\begin{equation}\label{eq:a2}
N(\phi) = R_{\phi} + 2S_{\phi}.
\end{equation}

\noindent Geometrically, we can recognize the two types of nodal domains according as $\Omega_{j}\cap\delta$ is empty or not. If $\Omega_{j}\cap\delta$ is nonempty, then it is an open nonempty subset of $\delta$, so that the invariance of $\delta$ under $\sigma$ ensures that $\Omega_{j}\cap\sigma(\Omega_{j})$ is nonempty and so $\Omega_{j}=\sigma(\Omega_{j})$. If on the other hand $\Omega_{j}\cap\delta$ is empty, then being open and connected, $\Omega_{j}$ lies entirely in $\mathcal{F}^{-}$ or $\mathcal{F}^{+}$ as indicated in Figure 1, so that $\Omega_{j}$ does not meet $\sigma(\Omega_{j})$ since $\sigma(\mathcal{F}^{+}) = \mathcal{F}^{-}$. Thus, inert domains are exactly the domains meeting $\delta$ and can be thought of as boundary nodal domains for $\mathcal{F}^{+}$. They are the analogues of what are called real zeros for holomorphic forms in \cite{GS12}.
%pg8%
Topologically, $\delta$ is a circle when compactified at the cusp $i\infty$. The vanishing of $\phi$ at the cusp and the normalization of the first Fourier coefficient (see Section 3) ensures that $\phi$ (restricted to $\delta$) is positive in a neighborhood of $i\infty$. Let 
\[
\tilde{\mathcal{S}} = \big{\{} i\infty = \tilde{K}_{1}, \tilde{K}_{2}, ..., \tilde{K}_{m} \big{\}}
\]
with $m = m_{\phi} \geq 1$, denote the zeros of $\phi$ arranged cyclically on $\delta$. Denote by $\mathcal{S}$ the subset $\{K_{1}, ..., K_{n}\}$ of the $\tilde{K}_{j}$'s at which $\phi$ changes sign. Clearly $n = n_{\phi}$ is even  and $n<m$ (since $\tilde{K}_{1} \notin \mathcal{S}$). We call $m_{\phi}$ the number of zeros of $\phi$ on $\delta$ and $n_{\phi}$ the number of sign changes.

Each inert nodal domain $\mathcal{V}_{j}$ meets $\delta \setminus \tilde{\mathcal{S}}$ and the intersection consists of a disjoint set of segments of the form $(\tilde{K}_{j},\tilde{K}_{j+1})$. Moreover, different inert domains give rise to different disjoint sets of such segments. It follows that the number of inert $\mathcal{V}_{j}$'s is at most the number of intervals $(\tilde{K}_{j},\tilde{K}_{j+1})$ so that we have the upper bound $R_{\phi} \leq m_{\phi}$.

To obtain a lower bound for $R_{\phi}$, we first assume that the nodal line $Z_{\phi}$ is nonsingular (that is $\phi$ and $\triangledown\phi$ never vanish together) so that $m_{\phi} = n_{\phi} + 1$. In this case, $Z_{\phi}$ has no self-crossings and consists of say $\nu$ simple closed curves $\ell_{1}$, ..., $\ell_{\nu}$ (each an embedded circle in $\mathbb{X})$. If $K\in \mathcal{S}$, then $\ell_{i}$ passes through $K$ for some $i$ and being an embedded circle must cross $\delta$ at some other point $K' \in \mathcal{S}\cup\{i\infty\}$. The resulting reflected curve (by $\sigma$) yields a closed curve in $\mathbb{X}$ which is one of the components of $Z_{\phi}$ and which we denote by $\ell_{K,K'}$. These curves meet $\delta$ at exactly two distinct points and every point of $\mathcal{S}$ appears once. Thus we get exactly $\frac{1}{2}n$ simple closed curves in $\mathbb{X}$ corresponding to the sign changes of $\phi$ on $\delta$. The remaining $\nu - \frac{1}{2}n$ closed curves come in pairs and do not meet $\delta$. It follows from this, by induction on $\nu$ and $n$, that there are exactly $\frac{1}{2}n +1$ nodal domains whose boundary contains one of the $\ell_{K,K'}$'s. In other words, under the nonsingular assumption, there are $\frac{1}{2}n +1$ inert domains and $\nu - (\frac{1}{2}n +1)$ split domains.
 
%pg9%
In general, $Z_{\phi}$ consists of piecewise analytic arcs with singularities formed by crossings (cusps) and self-intersections. However these may be removed locally one at a time at the cost of possibly decreasing the number of inert nodal domains but not changing $\mathcal{S}$ or at worst not increasing the quantity $R_{\phi} - (\frac{1}{2}n_{\phi} +1)$. In more detail, suppose that locally (away from $\delta$), $Z_{\phi}$ looks like the figure on the left below with $\pm$ denoting the sign of $\phi$ near the singularity and $\Omega_{i}$ the corresponding pieces of the complement. 

\begin{figure}[ht]
  \begin{center}
    \includegraphics[width=3.5in]{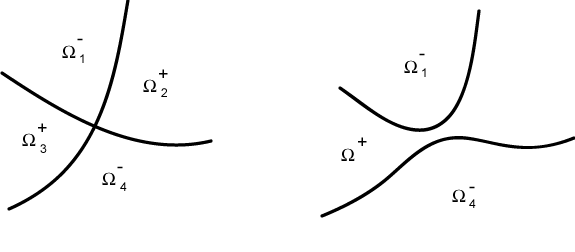}
  \end{center}
  \caption{}
\end{figure}

Deforming the curves near the singularity (as shown in the figure on the right above) would possibly decrease the number of connected components by one and keeps everything away from the singularity the same.
If the singularity looks like that below, with $p$ and $q$ joined as shown, then we remove the arc joining them which may again decrease the number of domains without changing anything else.

\begin{figure}[ht]
  \begin{center}
    \includegraphics[width=3.5in]{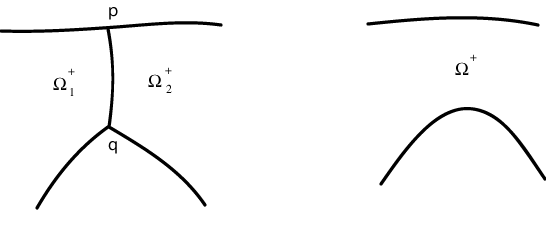}
  \end{center}
  \caption{}
\end{figure}

Next, consider the figure on the left below where $Z_{\phi}$ meets the boundary $\delta$. Here, $p$ and $q$ are sign changes of $\phi$ on $\delta$. To get the figure on the right, we deform the arcs containing $p$ and $q$ into each other so that $p$ and $q$ are replaced by the point $t$ which is not a sign change. Here we reduce the value of $n_{\phi}$ by two and lose exactly one inert nodal domain, so that $R_{\phi} - (\frac{1}{2}n +1)$ is unchanged and the singularity $s$ is removed.

\begin{figure}[ht]
  \begin{center}
    \includegraphics[width=3.5in]{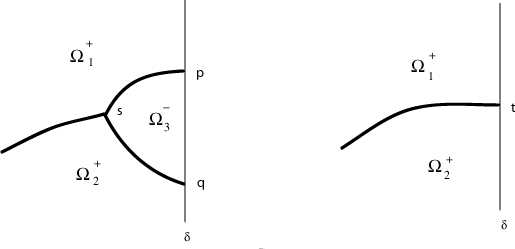}
  \end{center}
  \caption{}
\end{figure}

Last, we consider the case when one encounters a possible cusp singularity on the boundary as shown below. Here, there are two possible deformations. Suppose that $\phi$ is positive on $\Omega_{1}$ and $\Omega_{3}$, and negative on $\Omega_{2}$. Then, we may remove the singularity $s$ which is not a sign change, as shown in the middle figure. In this way, the number of inert nodal domains may have possibly decreased while the number of sign changes remains unchanged. The same conclusion holds if  $\phi$ is positive on $\Omega_{1}$ and $\Omega_{2}$, and negative on $\Omega_{3}$. Now we may remove the arc between $\Omega_{1}$ and $\Omega_{2}$ as shown on the right.
%pg10%
\begin{figure}[!ht]
  \begin{center}
    \includegraphics[width=3.5in]{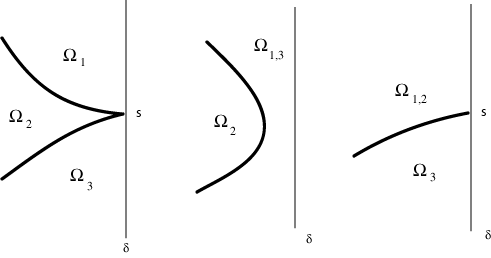}
  \end{center}
  \caption{}
\end{figure}

In this way, we remove all crossings and singularities so that when completed, $Z_{\phi}$ will be a nonsingular curve and the quantity $R_{\phi} - (\frac{1}{2}n_{\phi} +1)$ will not have increased. We then have the following

\begin{thm}\label{TheoremNodal} Let $R_{\phi}$ denote the number of inert nodal domains of $\phi$. Suppose $m_{\phi}$ denotes the number of zeros of $\phi$ on $\delta$, and $n_{\phi}$ the number of sign changes. Then
\[
\frac{1}{2}n_{\phi} + 1 \leq R_{\phi} \leq m_{\phi}.
\]
If $Z_{\phi}$ is nonsingular, then there are exactly $\frac{1}{2}n_{\phi} + 1$ inert nodal domains of $\phi$.
\end{thm}
\begin{rem} An analog of this lower bound also holds for other arithmetic surfaces as mentioned in the Introduction. However, one must take into account the genus $g$ of the surface. Then, the lower bound takes the shape
\[
R_{\phi} \geq \frac{1}{2}n_{\phi} + 1 -g,
\]
so that for a fixed surface, $R_{\phi} \gg n_{\phi}$.
\end{rem}

%%%%%%%%%%%%%%%%%%%%%%%%%%%%%%%%%%%%%%%%%%%%%%%%%%%%%%%%%%%%%%%%%%%%%%%%%%%%%%%
% pg 11  %%%%%%%%%%%%%%%%%%%%%%%%%%%%%%%%%%%%%%%%%%%%%%%%%%%%%%%%%%%%%%%%%%%%%
\subsection{Segments not in $Z_{\phi}$.\\}\quad
We end this Section by showing that the restriction of $\phi$ to a geodesic segment or a horocycle segment cannot be identically zero. That is to say that $Z_{\phi}$ cannot contain a geodesic segment or a horocycle segment.

We begin with a horocycle segment $\mathcal{C}\subset \mathbb{X}$. If $\phi \big{|}_{\mathcal{C}}\equiv 0$  then by analytic continuation $\phi \big{|}_{\mathcal{C}^{*}} \equiv 0$, where $\mathcal{C}^{*} \subset \mathbb{H}$ is the complete horocycle extending $\mathcal{C}$. It is well known \cite{Hed36} from the theory of the horocycle flow, that the projection of $\mathcal{C}^{*}$ in $\mathbb{X}$ is either a closed horocycle or it is dense in $\mathbb{X}$. In case of the latter, clearly $\phi \equiv 0$ in $\mathbb{X}$, which is a contradiction. On the other hand, if the projection of $\mathcal{C}^{*}$ in $\mathbb{X}$ is a closed horocycle $\bar{\mathcal{C}}$, then Theorem \ref{TheoremI1} is applicable (see Section 5 and Theorem 5.1) so that we have an essentially sharp, and in particular a non-zero, lower bound for the $L^{2}$-restriction of $\phi$ to $\bar{\mathcal{C}}$, which contradicts that $\phi \big{|}_{\mathcal{C}^{*}} \equiv 0$. One can also use the (exponentially small) lower bound of \cite{Ju11} which has the advantage of not making any arithmetic assumptions. However, we give an simpler argument as follows: suppose $\bar{\mathcal{C}}$ is the set $\{x+iY: -\half< x \leq \half\}$ with a fixed $Y>0$. Then for any integer $n \geq 1$,  
\[
\lVert \phi\Big{|}_{\bar{\mathcal{C}}}\rVert ^{2}  = \int_{-\half}^{\half}|\phi(x+iY)|^{2}\ \d x \gg |\lambda_{\phi}(n)K_{it_{\phi}}(2\pi nY)|^{2},
\]
where $\lambda_{\phi}(n)$ are the Hecke Fourier coefficients and $K$ is the Bessel-MacDonald function (see Sections 3 and 5 for details and notation). It is known that the Bessel function above does not vanish if we choose $n > \frac{t_{\phi}}{Y}$. Moreover, the Hecke relations ensures that $|\lambda_{\phi}(n)|>\half $ for all $n$ chosen to be either a prime or its square. Thus the $L^{2}$-norm above does not vanish. This completes the proof of the claim for horocyclic segments.

For geodesic segments $\beta$, assume that $\phi \big{|}_{\beta} \equiv 0$. First, if $ \beta \subset \delta$ then along $\beta$ we have both $\phi \big{|}_{\beta} \equiv 0$ and $\partial_{n}\phi \big{|}_{\beta} \equiv 0$ so that by the reflection principle (or Carleman estimates)  $\phi$ vanishes identically on $\mathbb{X}$. Hence we may assume that $\beta$ is disjoint from $\delta$. Let $\tilde{\beta}$ be the lifted complete geodesic in $\mathbb{H}$, so that again from analytic continuation $\phi \big{|}_{\tilde{\beta}} \equiv 0$. Hence by the reflection principle, as an eigenfunction on $\mathbb{H}$, $\phi$ satisfies $\phi(R_ {\tilde{\beta}}z) = -\phi(z)$, where $R_ {\tilde{\beta}}$ is the reflection of $\mathbb{H}$ in $\tilde{\beta}$. Moreover, we also have $\phi(R_ {j}z) = \phi(z)$ where $R_{j}$ is the reflection in $\tilde{\delta_{j}}$, for $j=1,2$ and $3$. Thus, $\phi^{2}(z)$ is invariant under the subgroup $\Lambda$ of the isometries of $\mathbb{H}$ generated by $R_{1},R_{2}, R_{3}$ and $R_{\tilde{\beta}}$. The reflection group $<R_{1},R_{2}, R_{3}>$ is known to be a maximal discrete subgroup of the isometry group of $\mathbb{H}$. Since $R_{\tilde{\beta}} \notin <R_{1},R_{2}, R_{3}>$ (since we are assuming here that $\tilde{\beta}({\rm mod} \Gamma)$ is not contained in $\delta$), we conclude that $\Lambda$ is not a discrete subgroup  of the isometry group of $\mathbb{H}$. Hence the same is true of the subgroup of words of even length in $R_{1},R_{2}, R_{3}$ and $R_{\tilde{\beta}}$, as a subgroup of $\SL_{2}(\mathbb{R})$. In particular, the closure of this last group contains a 1-parameter connected subgroup of $\SL_{2}(\mathbb{R})$ which we denote by $B$. So $\phi^{2}$ is invariant under the action of $B$ on $\mathbb{H}$. By connectedness, $\phi$ is also invariant under $B$ so that $\phi$ is an eigenfunction that lives on $B\backslash\mathbb{H}$. Such an eigenfunction is a function of one variable (satisfying an ODE) and unless it is identically zero cannot be invariant under $\Gamma$. Thus $\phi \equiv 0$ on $\mathbb{H}$, a contradiction, proving that $\phi \big{|}_{\beta} \equiv 0$ is impossible.

%%%%%%%%%%%%%%%%%%%%%%%%%%%%%%%%%%%%%%%%%%%%%%%%%%%%%%%%%%%%%%%%%%%%%%%%%%%%%%%%%%%%%%%%%%%%%%%%%%%%%%%%%%%%%%%%%%%%%%%%%%%%%%%%%%%%%%%%%%          SECTION 3        %%%%%%%%%%%%%%%%%%%%%%%%%%%%%%%%%%%%%%%%%%%%%%%%%%%%%%%%%%%%%%%%%%%%%%%%%%%%%%%%%%%%%%%%%%%%%%%%%%%% pg 12 %%%%%%%%%%%%%%%%%%%%%%%%%%%%%%%%%%%%%%%%%%%
\section{Preliminaries, approximation theorems and QUE.}\label{Preliminaries} %S1

If $\phi(z)$ is a Maass eigenform associated with the eigenvalue $\lambda = \frac{1}{4} + t_{\phi}^{2}$ for the Laplace operator on the modular surface $\mathbb{X} = \SL_{2}(\mathbb{Z})\backslash\mathbb{H}$,  then recall that it has a Fourier expansion of the type
\begin{equation}\label{eq:b1}
\phi(z) = \sum_{n \neq 0} \rho_{\phi}(n) y^{\frac{1}{2}}{\rm K}_{it_{\phi}}(2\pi |n|y)e(nx),
\end{equation}
where $z = x + iy$, ${\rm K}_{\nu}(y)$ the MacDonald-Bessel function and $e(x)= \exp(2\pi ix)$. In what follows, we will always assume that our eigenform is $L^{2}$-normalized, that it is a Hecke-eigenform and that $|t_{\phi}| \rightarrow \infty$. It is then known that (see Iwaniec \cite{Iw90} and Hoffstein-Lockhart \cite{HL94}) that for $\ve >0$
\begin{equation}\label{eq:b2}
t_{\phi}^{-\ve} \ll |\rho_{\phi}(1)|^{2}e^{-\pi t} \ll t_{\phi}^{\ve}.
\end{equation}
It will be convenient to use the notation 
\begin{equation}\label{eq:b300}
\tilde{\rho}_{\phi}(n) = \rho_{\phi}(n)e^{-\frac{\pi}{2} t_{\phi}}.
\end{equation}
One has
\begin{equation}\label{eq:b3}
\rho_{\phi}(n)= \lambda_{\phi}(n)\rho_{\phi}(1)\quad  \text{and} \quad \tilde{\rho}_{\phi}(n)= \lambda_{\phi}(n)\tilde{\rho}_{\phi}(1)
\end{equation}
where $\lambda_{\phi}(n)$, with $\lambda_{\phi}(1)=1$, are the Hecke eigenvalues which by \cite{KS03} are known to satisfy
\begin{equation}\label{eq:b4}
\lambda_{\phi}(n) \ll_{\ve} n^{\theta +\ve},
\end{equation}
with $\theta=\frac{7}{64}$.
and are expected to satisfy the Ramanujan-Selberg Conjecture
\begin{equation}\label{eq:b5}
 \lambda_{\phi}(n) \ll_{\ve} n^{\ve}.
\end{equation}\\

\subsection{Approximation theorems}\quad 

We recall here the asymptotic behavior of ${\rm K}_{\nu}(u)$ for $\nu=ir$ with large $r>0$ and all $u>0$. 

%%%%%%%%%%%%%%%%%%%%%%%%%%%%%%%%%%%%%% Lemma 1 %%%%%%%%%%%%%%%%%%%%%%%%%%%%%%%%%%%%%%%%%%%%%%
%%%%%%%%%%%%%%%%%%%%%%%%%%%%%%%%%%%%%%%%%%%%%%%%%%%%%%%%%%%%%%%%%%%%%%%%%%%%%%%%%%%%%%%%%%%%%
\begin{lem}\label{LemmaOne} Let $C$ be a sufficiently large positive constant and $r>0$.
\begin{itemize}
\item[(1)]  If $0 < u < r - Cr^{\frac{1}{3}}$, then
\[
e^{\frac{\pi}{2}r}{\rm K}_{ir}(u) = \frac{\sqrt{2\pi}}{(r^{2} - u^{2})^{\frac{1}{4}}}\sin \left(\frac{\pi}{4} +r {\rm H}(\frac{u}{r})\right) \left(1 + O(\frac{1}{r {\rm H}(\frac{u}{r})})\right),
\]
\item[(2)]  If $u > r + Cr^{\frac{1}{3}}$, then
\[
e^{\frac{\pi}{2}r}{\rm K}_{ir}(u) = \sqrt{\frac{\pi}{2}}\frac{1}{(u^{2} - r^{2})^{\frac{1}{4}}}e^{-r{\rm H}(\frac{u}{r})} \left(1 + O(\frac{1}{r {\rm H}(\frac{u}{r})})\right),
\]
\item[(3)] If $|u - r| \leq Cr^{\frac{1}{3}}$, then
\[
e^{\frac{\pi}{2}r}{\rm K}_{ir}(u) = \pi \left(\frac{2}{r}\right)^{\frac{1}{3}}{\rm Ai}\left((u-r)(\frac{u}{2})^{-\frac{1}{3}}\right) + O(r^{-\frac{2}{3}}),
\]
\end{itemize}
where ${\rm Ai}(x)$ is the Airy function. The implied constants are absolute.
\end{lem}
Here the non-negative function $H(\xi)$ (shown in Fig.8) is given by
\begin{equation}\label{eq:b6}
{\rm H}(\xi) = 
\begin{cases}
\hfill \text{arccosh}(\frac{1}{\xi}) - \sqrt{1-\xi^{2}} & :\ \  0<\xi \leq 1,\\
\hfill \sqrt{\xi^{2}-1} - \text{arcsec}(\xi) & :\ \  \xi > 1 .\\
\end{cases}
\end{equation}

\begin{figure}[!ht]

  \begin{center}

    \includegraphics[width=3.0in]{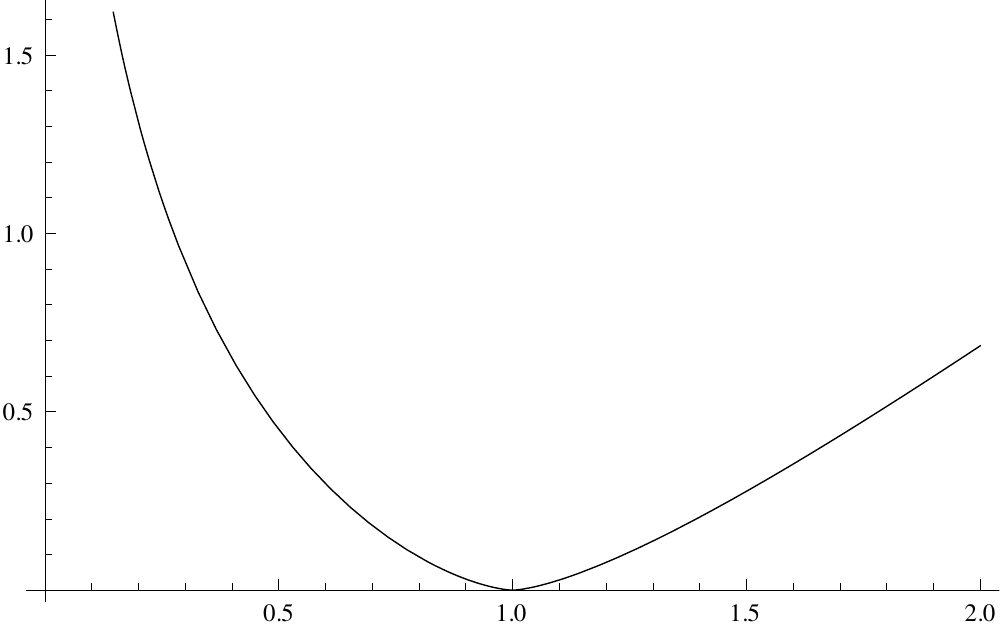}

  \end{center}
 \caption{} 
\end{figure}
%pg13%
For $0<\xi\leq 1$, it is a decreasing function, while for $\xi > 1$ it is increasing. Near $\xi =1$, it has order of magnitude $ |\xi^{2} - 1|^{\frac{3}{2}}$ and near $\xi =0$ it has a logarithmic singularity.

The lemma follows readily from asymptotic expansions for the Airy function, and from uniform asymptotic expansions for the K-Bessel function by Balogh \cite{Ba66}, \cite{Ba67} (based on earlier work of Olver).
%%%%%%%%%%%%%%%%%%%%%%%%%%%%%%%%%%%%% corollary %%%%%%%%%%%%%%%%%%%%%%%%%%%%%%%%%%%%%%%%%%%%%%%%%%%%%%
\begin{cor}\label{CorlOne} Uniformly in $u>0$ and $r$ sufficiently large, one has
\[
{\rm K}_{ir}(u) \ll
\begin{cases}
\hfill (r^{2} - u^{2})^{-\frac{1}{4}}e^{-\frac{\pi}{2}r},\\
 u^{-\frac{1}{2}}e^{-u},\\
 r^{-\frac{1}{3}}e^{-\frac{\pi}{2}t}.\\
\end{cases}
\]
\end{cor}

Define 
\begin{equation}\label{eq:b7}
\Phi(z) = \frac{\phi(z)}{\rho_{\phi}(1)e^{-\frac{\pi}{2}t_{\phi}}\sqrt{y}} = \sum_{n\neq 0}\lambda_{\phi}(n)e^{\frac{\pi}{2}t_{\phi}}{\rm K}_{it_{\phi}}(2\pi |n|y)e(nx).
\end{equation}

%%%%%%%%%%%%%%%%%%%%%%%%%%%%%%%%%%%%%%%%%%% Proposition 1 %%%%%%%%%%%%%%%%%%%%%%%%%%%%%%%%%%%%%%%%%%%
%%%%%%%%%%%%%%%%%%%%%%%%%%%%%%%%%%%%%%%%%%%%%%%%%%%%%%%%%%%%%%%%%%%%%%%%%%%%%%%%%%%%%%%%%%%%%%%%%%%%
\begin{prop}\label{PropOne} Let $\Delta = t_{\phi}^{\frac{1}{3}}\log t_{\phi}$. Then for any $c < y \leq \frac{t_{\phi}}{2\pi}$ with any fixed $c>0$, we have 
\begin{itemize}
\item[(1)]
\begin{multline*}
\quad\quad\Phi(z) =  \sqrt{2\pi}\sum_{|n|\leq \frac{t_{\phi}-\Delta}{2\pi y}} \lambda_{\phi}(n) \frac{e(nx)}{(t_{\phi}^{2}-(2\pi ny)^{2})^{\frac{1}{4}}} \sin \left(\frac{\pi}{4} +t_{\phi} {\rm H}(\frac{2\pi |n|y}{t_{\phi}})\right) \cr 
+ O\big{(}t_{\phi}^{\theta -\frac{1}{3} +\ve}(\frac{\Delta}{y}+1)\big{)},\quad\quad
\end{multline*}
\item[(2)] For any $-\half < \alpha \leq \beta \leq \half$
\begin{multline*}
\quad\quad\int_{\alpha}^{\beta} \Phi(x+iy) \ \d x =\cr
\sqrt{2\pi}\sum_{0<|n|\leq \frac{t_{\phi}-\Delta}{2\pi y}} \frac{\lambda_{\phi}(n)}{2\pi in} \frac{e(n\beta)-e(n\alpha)}{(t_{\phi}^{2}-(2\pi ny)^{2})^{\frac{1}{4}}} \sin \left(\frac{\pi}{4} +t_{\phi} {\rm H}(\frac{2\pi |n|y}{t_{\phi}})\right)\cr
+ O\big{(}t_{\phi}^{\theta -\frac{4}{3}+\ve}(\frac{\Delta}{y}+1)\big{)}.\quad\quad
\end{multline*}
\item[(3)] $\int_{-\half}^{\half}|\Phi(x+iy)|^{2} \ \d x =$
\begin{align*}
2\pi\sum_{|n|\leq \frac{t_{\phi}-\Delta}{2\pi y}} |\lambda_{\phi}(n)|^{2} \frac{1}{\sqrt{t_{\phi}^{2} - (2\pi ny)^{2}}} \sin ^{2} \left(\frac{\pi}{4} +t_{\phi} {\rm H}(\frac{2\pi |n|y}{t_{\phi}})\right) \cr 
 + O\big{(}t_{\phi}^{2\theta - \frac{2}{3}+\ve}(\frac{\Delta}{y}+1)\big{)},
\end{align*}
\noindent where the implied constants depend at most on $\theta$ and $\ve$ (and not on $\alpha$, $\beta$ and $y$).
\end{itemize}
\end{prop}
%pg14%
\begin{proof} This follows by decomposing the sum in \eqref{eq:b5} into three sums with $n$'s in the range $ 2\pi |n|y < t_{\phi} - \Delta$, $ \big{|}2\pi |n|y -t_{\phi}\big{|}\leq \Delta$  and $2\pi |n|y > t_{\phi} + \Delta$ so that we may apply Lemma \ref{LemmaOne}. The analysis is straightforward using \eqref{eq:b2}, \eqref{eq:b3}, \eqref{eq:b4} together with the bound ${\rm Ai}(x) \ll |x|^{-\frac{1}{4}}$  as $|x| \rightarrow \infty$ for the Airy function. The middle {\em transitional range} is responsible for the error terms indicated together with the error terms given by Lemma \ref{LemmaOne}, while the infinite range gives a term with exponential decay. 
\end{proof}

We next provide a good approximation to $\Phi(x+iy)$ for a sequence of large values of $y$ (this being the analog of Theorem 3.1 in \cite{GS12}, where a similar result was shown for holomorphic Hecke cusp forms). Define, for integers $l$ 
\begin{equation}\label{eq:b8}
y_{l} = \frac{t_{\phi}}{2\pi l}, \quad 1 \leq l \ll t_{\phi}.
\end{equation}

%%%%%%%%%%%%%%%%%%%%%%%%%%%%%%%%%%%%%%%%%% Proposition 2 %%%%%%%%%%%%%%%%%%%%%%%%%%%%%%%%%%%%%%%%%%%
%%%%%%%%%%%%%%%%%%%%%%%%%%%%%%%%%%%%%%%%%%%%%%%%%%%%%%%%%%%%%%%%%%%%%%%%%%%%%%%%%%%%%%%%%%%%%%%%%%%%
\begin{prop}\label{PropTwo} For any small $\ve > 0$ let $l\in \mathbb{Z}$ with $1 \leq l \leq \ve \frac{t_{\phi}}{\Delta}$. Assume that $\phi$ is an even Maass cusp form. Then, for  any $-\half < x\leq \half$, writing $z_{l} = x +iy_{l} $ we have
\begin{multline*}
\Phi(z_{l}) =  \pi \left(\frac{16}{t_{\phi}}\right)^{\frac{1}{3}}\lambda_{\phi}(l){\rm Ai}(0)\cos(2\pi lx) \cr
+ \sqrt{\frac{2\pi l}{t_{\phi}} }\sum_{|n|\leq l-1} \lambda_{\phi}(n) \frac{e(nx)}{(l^{2}- n^{2})^{\frac{1}{4}}} \sin \left(\frac{\pi}{4} +t_{\phi} {\rm H}(\frac{|n|}{l})\right)\left(1+O(\frac{1}{t_{\phi}{\rm H}(\frac{|n|}{l})})\right) \cr
+ O(t_{\phi}^{\theta - \frac{2}{3}}).\quad\quad
\end{multline*}
Moreover, if $t_{\phi}^{\frac{1}{3} + \ve} \ll y_{l} \ll \ve t_{\phi}$, then for any $y$ satisfying $|y - y_{l}| \leq \ve y_{l} t_{\phi}^{-\frac{2}{3}}$
\begin{multline*}
\Phi(x+iy)=  \pi \left(\frac{16}{t_{\phi}}\right)^{\frac{1}{3}}\lambda_{\phi}(l){\rm Ai}\left((2\pi l(y-y_{l}))(\pi ly)^{-\frac{1}{3}}\right)\cos(2\pi lx) \cr
+\sqrt{2\pi}\sum_{|n|\leq l-1} \lambda_{\phi}(n) \frac{e(nx)}{(t_{\phi}^{2}-(2\pi ny)^{2})^{\frac{1}{4}}} \sin \left(\frac{\pi}{4} +t_{\phi} {\rm H}(\frac{2\pi |n|y}{t_{\phi}})\right)\times \cr
\left(1+O(\frac{1}{t_{\phi}{\rm H}(\frac{2\pi |n|y}{t_{\phi}})})\right) + O(t_{\phi}^{\theta - \frac{2}{3}}),
\end{multline*}
where the implied constants depend at most on $\ve$.
\end{prop}
\begin{proof} The main terms above come from the {\em transitional range}, where the $n$'s are restricted to satisfy 
\[
 \big{|}|n| -\frac{t_{\phi}}{2\pi y_{l}}\big{|} = \big{|}|n|-l\big{|}\leq \frac{\Delta}{2\pi y_{l}} \leq \ve,
\]
 so that the contribution here comes only from $n = \pm l$ in the first case, with a similar analysis for the second. One uses Lemma \ref{LemmaOne}, \eqref{eq:b7} and the fact that the cusp form is an even function. The error term comes from the rapid exponential decay of the tail end together with the error from the main term. 
\end{proof}
%pg15%
We end this subsection with the following elementary lemma that will be used often.

%%%%%%%%%%%%%%%%%%%%%%%%%%%%%%%%%%%%%  Lemma 0 %%%%%%%%%%%%%%%%%%%%%%%%%%%%%%%%%%%%%%%%%%%%%%%%%%%
%%%%%%%%%%%%%%%%%%%%%%%%%%%%%%%%%%%%%%%%%%%%%%%%%%%%%%%%%%%%%%%%%%%%%%%%%%%%%%%%%%%%%%%%%%%%%%%%%%
\begin{lem}\label{LemmaZero} Let $\ve > 0$ be fixed. Then for all  $0 < u < U-\ve$ and $0 < \alpha \leq 1$
\[
\sum _{n=1}^{u} \frac{1}{(U^{2} - n^{2})^{\alpha}} \ll_{\ve}  
\begin{cases}
\hfill \frac{1}{U}\log U & :\ \  {\rm if} \quad \alpha=1\\
\hfill U^{1-2\alpha} & :\ \  {\rm if} \quad \alpha < 1 .\\

%%%%%
\end{cases}
\]
\end{lem}
\vspace{12pt}
%%%%%%%%%%%%%%%%%%%%%%%%%%%%%%%%%%%%%%%%%%% Section 3.2 %%%%%%%%%%%%%%%%%%%%%%%%%%%%%%%%%%%%%%%%%%%%%%%
%%%%%%%%%%%%%%%%%%%%%%%%%%%%%%%%%%%%%%%%%%%%%%%%%%%%%%%%%%%%%%%%%%%%%%%%%%%%%%%%%%%%%%%%%%%%%%%%%%%

\subsection{Mean-value theorems and QUE}\quad

In this subsection all the statements are valid for Maass forms $\phi$ without the assumption that it is even.

%%%%%%%%%%%%%%%%%%%%%%%%%%%%%%%%%%%%%%%%%%% Lemma 2 %%%%%%%%%%%%%%%%%%%%%%%%%%%%%%%%%%%%%%%%%%%%%%%
%%%%%%%%%%%%%%%%%%%%%%%%%%%%%%%%%%%%%%%%%%%%%%%%%%%%%%%%%%%%%%%%%%%%%%%%%%%%%%%%%%%%%%%%%%%%%%%%%%%
\begin{lem}\label{LemmaTwo} For $X\geq 1$, and $k = 1, 2$
\[
\sum_{|n|\leq X} |\lambda_{\phi}(n)|^{2k} \ll_{\ve} Xt_{\phi}^{\ve}.
\]
\end{lem}
\begin{proof} The case $k=1$ is due to Iwaniec \cite{Iw90}. For $k=2$, one follows Kim's proof \cite{Kim} that $sym^{4}(\phi)$ corresponds to an automorphic cusp form on $GL_{5}$, and then applies Molteni's extension \cite{Mo02} to $GL_{n}$ of Iwaniec's method to bound sums of Fourier coefficients of cusp forms (see also \cite{Li10}, and \cite{LY11} where a similar result is proved for $k=4$).
\end{proof}

For the proofs of some of the theorems in later sections, we will need stronger versions of the above Lemma for the case $k=1$, with sharp cutoff functions in suitable intervals. It is clear that one need only prove the case when $\phi$ is either even or odd.

%%%%%%%%%%%%%%%%%%%%%%%%%%%%%%%%%%%%%%%%%%% Lemma 3 %%%%%%%%%%%%%%%%%%%%%%%%%%%%%%%%%%%%%%%%%%%%%%%
%%%%%%%%%%%%%%%%%%%%%%%%%%%%%%%%%%%%%%%%%%%%%%%%%%%%%%%%%%%%%%%%%%%%%%%%%%%%%%%%%%%%%%%%%%%%%%%%%%%
\begin{prop}\label{shortsum}\  For any fixed $\omega > 0$
\[
 \frac{1}{100}\omega t_{\phi} \leq \sum_{|n|\leq \omega t_{\phi}}|\tilde{\rho}_{\phi}(n)|^{2} \leq 100\omega t_{\phi}
\]
and
\[
\sum_{10^{-5}\omega t_{\phi} \leq |n|\leq \omega t_{\phi}}|\tilde{\rho}_{\phi}(n)|^{2} \geq \frac{1}{1000}\omega t_{\phi}\ .
\]
\end{prop}

\begin{proof}\  As we are assuming $\phi$ is even or odd, there is no loss in assuming $n\geq 1$ when convenient. It is also clear that the second part follows easily from the first. 

We apply QUE with an incomplete Eisenstein series, that is with 
\[
E_{g}(z)=\sum_{\gamma\in\Gamma_{\infty}\backslash\Gamma}g\big{(}y(\gamma z)\big{)}.
\]
with a function $g\in C_{0}^{\infty}(\mathbb{R}^{+})$ to be chosen, applied to the integral
\[
\int_{\mathbb{X}}|\phi(z)|^{2}E_{g}(z)\ \d {\rm A}(z).
\]
Unfolding, this tells us 
\begin{equation}\label{eq:ea8}
\sum_{|n|\neq 0}|\rho_{\phi}(n)|^{2}\int_{0}^{\infty} g(y)K_{it_{\phi}}(2\pi |n|y)^{2}\ \d^{\times} y = \frac{3}{\pi}\int_{0}^{\infty} g(y) \frac{\d y}{y^2} + o(1).
\end{equation}
Let $\alpha > 0$ and set
\[
g(y) = \begin{cases} 1 & \mbox{if} \ \alpha<y\leq 2\alpha, \\ 0 & \mbox{otherwise,}
\end{cases}
\]
%pg16%
(while $g$ is not smooth, it has compact support and is a legitimate choice) so that
\begin{equation}\label{eq:d40}
\sum_{|n|\neq 0}|\rho_{\phi}(n)|^{2}\int_{0}^{\infty} g(y)K_{it_{\phi}}(2\pi |n|y)^{2}\ \d^{\times} y = \frac{3}{2\pi \alpha}  + o(1).
\end{equation}

\vspace{20pt}
\noindent(i). \underline{The lower bound}.

 By Prop. \ref{PropOne} and \eqref{eq:d40} we have
\[
\int_{\alpha}^{2\alpha} \sum_{0<|n|\leq \frac{t_{\phi}-\Delta}{2\pi y}}|\tilde{\rho}_{\phi}(n)|^{2}\frac{1}{\sqrt{t_{\phi}^{2} - (2\pi ny)^{2}}} \frac{\ \d y}{y} \geq \frac{1}{10\alpha}\ ,
\]
the error terms being negligible because $\theta$ is sufficiently small. Interchanging the order of summation and integration and simplifying gives us
\[
\frac{1}{\sqrt{t_{\phi}}}\sum_{0< |n|\leq \frac{t_{\phi}-\Delta}{2\pi \alpha}}|\tilde{\rho}_{\phi}(n)|^{2}\int_{2\pi |n|\alpha}^{ \min{\{ 4\pi |n|\alpha,t-\Delta\} }} \frac{1}{\sqrt{t_{\phi} - u}} \frac{\ \d u}{u} \geq \frac{1}{10\alpha}\ .
\]
The integral does not exceed  $ \frac{10}{\sqrt{t_{\phi}}}$, so that the statement in the proposition follows on putting $\alpha=\frac{1}{2\pi \omega}$.

\vspace{20pt}
\noindent(ii) \underline{The upper bound}.\\
The integral in \eqref{eq:d40} can be investigated asymptotically to sufficient accuracy. Let 
\[
\tilde{g}(s) = \int_{0}^{\infty} g(y)y^{s}\d^{\times} y\ ,
\]
and recall that
\[
\int_{0}^{\infty} K_{it_{\phi}}^{2}(2\pi |n|y)y^{s}\d^{\times} y = \frac{1}{8}(\pi |n|)^{-s}\frac{\Gamma^{2}(\half s)}{\Gamma(s)}\Gamma(\half s +it_{\phi})\Gamma(\half s -it_{\phi})\ ,
\]
so that
\[
\int_{0}^{\infty} g(y)K_{it_{\phi}}^{2}(2\pi |n|y)\ \d^{\times} y = \frac{1}{2\pi i}\int_{\Re(s)=1} \frac{1}{8}(\pi |n|)^{-s}\frac{\Gamma^{2}(\half s)}{\Gamma(s)}\Gamma(\half s +it_{\phi})\Gamma(\half s -it_{\phi})\tilde{g}(-s)\d s\ .
\]
Applying Stirling's formula as $t_{\phi} \rightarrow \infty$ (see Section 6.1 for a similar treatment), we have to good accuracy that
\begin{align}\label{eq:ea7}
\int_{0}^{\infty} g(y)K_{it_{\phi}}^{2}(2\pi |n|y)\ \d^{\times} y &\sim  2\pi \frac{e^{-\pi t_{\phi}}}{2\pi i t_{\phi}}\int_{\Re(s)=1} \frac{1}{8}\big{(}\frac{\pi |n|}{t_{\phi}}\big{)}^{-s}\frac{\Gamma^{2}(\half s)}{\Gamma(s)}\tilde{g}(-s)\d s\ ,\cr
&\sim 2\pi \frac{e^{-\pi t_{\phi}}}{8t_{\phi}}W_{g}\big{(}\frac{\pi |n|}{t_{\phi}}\big{)},
\end{align}
where 
\begin{align*}
W_{g}(x) &= \frac{1}{2\pi i}\int_{\Re(s)=1} \frac{\Gamma^{2}(\half s)}{\Gamma(s)}\tilde{g}(-s)x^{-s}\d s \ ,\cr
&=\frac{1}{2\pi i}\int_{\Re(s)=1} \frac{\Gamma^{2}(\half s)}{\Gamma(s)}\tilde{g_{0}}(s)x^{-s}\d s \ ,
\end{align*}
where $g_{0}(y)=g(y^{-1})$. Then for $x\geq 0$ we have $W_{g}(x) = 4A_{g}(2x)$, where $A_{g}(x)$ is (essentially) the Abel transform
%pg17%
\begin{equation}\label{ea8b}
A_{g}(x) = \int_{x}^{\infty}\frac{g(\frac{1}{y})}{\sqrt{y^{2} - x^{2}}}\d y\ .
\end{equation}
In particular, if $g\geq 0$, then so is $W_{g}$. Moreover, we see that
\[
\int_{0}^{\infty} g(y)\frac{\d y}{y^{2}} = \frac{1}{\pi}\int_{0}^{\infty} W_{g}(y)\d y \ ,
\]
so that \eqref{eq:ea8} now reads
\begin{equation}\label{eq:ea9}
\frac{2\pi}{t_{\phi}}\sum_{n\neq 0}|\tilde{\rho}_{\phi}(n)|^{2}W_{g}(\frac{\pi |n|}{t_{\phi}}) = \frac{24}{\pi^2}\int_{0}^{\infty} W_{g}(y) \d y + o(1)\ .
\end{equation}
Then, with the choice of $g(y)$ as before
\begin{equation}\label{eq:ea10}
\int_{0}^{\infty}g(y)\frac{\d y}{y^2} = \frac{1}{2\alpha}.
\end{equation}
Also, if $0\leq x \leq(8\alpha)^{-1}$, then $W_{g}(x) \geq 4 \log{2} > 1$, say. Hence, using positivity and the above, as $t_{\phi} \rightarrow \infty$
\[
\frac{1}{t_{\phi}}\sum_{n=1}^{\infty}|\tilde{\rho}_{\phi}(n)|^{2}W_{g}(\frac{\pi n}{t_{\phi}}) > \sum_{1\leq n\leq \frac{t_{\phi}}{8\pi \alpha}}|\tilde{\rho}_{\phi}(n)|^{2}\ .
\]
Then, from \eqref{eq:ea10},  \eqref{eq:ea9} and the above, we have
\[
\sum_{1\leq n\leq \frac{t_{\phi}}{8\pi \alpha}}|\tilde{\rho}_{\phi}(n)|^{2} \leq \frac{100}{8\pi \alpha}t_{\phi}
\]
which proves the proposition.
\end{proof}

%Corollary 3.8
\begin{cor}\label{shortsumcorl} Let $\frac{a}{q}\in \mathbb{Q}$ with $(a,q)=1$ and $q\geq 1$. There exist small positive constants $D_{q}$ and $\ve_{q}$ such that if $X=\omega t_{\phi}$, with $0<\omega <\infty$ fixed, then
\[
 D_{q}X \leq \sum_{1\leq |n| \leq X}|\tilde{\rho}_{\phi}(n)|^{2}\cos^{2}\big{(}2\pi \frac{an}{q}\big{)} \leq 100X
\]
and
\[
\sum_{\ve_{q}X \leq |n| \leq X}|\tilde{\rho}_{\phi}(n)|^{2}\cos^{2}\big{(}2\pi \frac{an}{q}\big{)} \geq \frac{D_{q}}{2}X .
\]
(Here $\ve_{q} = \frac{1}{200}D_{q}$ suffices).
\end{cor}
\proof \ We will use the following: there is a positive number $C_{q}$ such that 
\begin{itemize}
\item[(i)] if $4 \nmid q$, then $\cos^{2}\big{(}2\pi \frac{an}{q}\big{)} \geq C_{q}$ for all $n\in \mathbb{Z}$;
\item[(ii)] if $4 \mid q$, the same is true for those $n$ not of the form $\frac{q}{4}m$ with $m$ odd. Otherwise, the cosine vanishes. Note that this implies that if $8\mid q$, then $\cos^{2}\big{(}2\pi \frac{an}{q}\big{)} \geq C_{q}$ for all odd numbers $n$.
\end{itemize}
Thus the lower bounds in the corollary follow from Prop. \ref{shortsum} with $D_{q} = \frac{1}{100}C_{q}$ when $4 \nmid q$. 

For the remaining cases we consider first the case when $8\mid q$. Write $q=2^{J}q_{0}$ with $J\geq 3$ and $q_{0}$ odd. Put
\[
S(X) = \sum_{1\leq n \leq X}|\tilde{\rho}_{\phi}(n)|^{2} \quad \text{and} \quad T(X) = \sum_{1\leq n \leq X}|\tilde{\rho}_{\phi}(n)|^{2}\cos^{2}\big{(}2\pi \frac{an}{q}\big{)}\ .
\]
%pg18%
Then,
\begin{equation}\label{eq:d80}
T(X) \geq C_{q}\sum_{\substack{n\leq X \\ n \neq \frac{q}{4}m, m \ \text{odd}}} |\tilde{\rho}_{\phi}(n)|^{2} := C_{q}(S(X) - V(X)),
\end{equation}
say, where 
\[
V(X) = \lambda_{\phi}^{2}(2^{J-2})\sideset{}{^*}\sum_{\substack{l\leq \frac{X}{2^{J-2}}\\ q_{0}\mid l}}|\tilde{\rho}_{\phi}(l)|^{2},
\]
where the $\sideset{}{^*}\sum$ denotes a sum over odd numbers. 

If $\lambda_{\phi}^{2}(2^{J-2}) \leq \half$, then $V(X) \leq \half S(\frac{X}{2^{J-2}}) \leq \half S(X)$ so that $T(X) \geq \half C_{q}S(X)$, and we are done. In the other case, we first observe that
\begin{align*}
T\Big{(}\frac{X}{2^{J-2}}\Big{)} &\geq \sideset{}{^*}\sum_{\substack{l\leq \frac{X}{2^{J-2}}\\ q_{0}\mid l}}|\tilde{\rho}_{\phi}(l)|^{2}\cos^{2}(2\pi \frac{a}{q}l)\cr
&\geq C_{q}\lambda_{\phi}^{-2}(2^{J-2})V(X)\ .
\end{align*}
Substituting into \eqref{eq:d80} gives
\[
T(X) \geq \frac{C_{q}}{1 +  \lambda_{\phi}^{2}(2^{J-2})}S(X).
\]
Since $\lambda_{\phi}^{2}(2^{J-2})$ is bounded above independently of $\phi$, the conclusion follows.

Last, we consider the case when $4 ||q$ so that we write $q=4q_{0}$. First we have
\begin{equation}\label{eq:d81}
T(X) \geq C_{q} \sum_{\substack{n\leq X \\ 2\mid n}} |\tilde{\rho}_{\phi}(l)|^{2} := C_{q}W(X),
\end{equation}
say. 

Suppose $\lambda_{\phi}^{2}(2)\geq \half$. Then it follows that $W(X) \geq \half (S(\half X) - W(\half X))$, so that $W(X) \geq \frac{1}{3}S(\half X)$. Substituting this into \eqref{eq:d81} gives the result. Next, if $\lambda_{\phi}^{2}(2)\leq \half$, then the Hecke relations imply $\lambda_{\phi}^{2}(4)\geq \frac{1}{4}$.Then we rewrite \eqref{eq:d81} as
\begin{align*}
T(X) &\geq C_{q}\sum_{\substack{n\leq X\\4||n}}|\tilde{\rho}_{\phi}(n)|^{2} \geq C_{q}\lambda_{\phi}^{2}(4)\sideset{}{^*}\sum_{n\leq \frac{X}{4}} |\tilde{\rho}_{\phi}(n)|^{2}\cr
&\geq \frac{1}{4}C_{q}\big{(}S(\frac{1}{4}X) - \sum_{\substack{n\leq \frac{1}{4}X\\n \ \text{even}}} |\tilde{\rho}_{\phi}(n)|^{2}\big{)}\cr
&=\frac{1}{4}C_{q}S(\frac{1}{4}X) - \frac{1}{4}C_{q}W(\frac{1}{4}X).
\end{align*}
From \eqref{eq:d81}, we also have $C_{q}W(X) \leq T(X)$. Combining with the above we conclude that $T(X) \geq \frac{1}{5}C_{q}S(\frac{1}{4}X)$ completing the proof. The second conclusion follows from the first.

%%%%%%%%%%%%%%%%%%%%%%%%%%%%%%%%%%%%%%%%%%%%%%%%%%%%%%%%%%%%%%%%%%%%%%%%%%%%%%%%%%%%%%%%%%%%%%%%%%%%%%%%%%%%%%%%%%%%%%%%%%%%%%%%%%%%%%%%%%%%%%%%%%%%%%%%%%%%%%%%%%%%%%%%%%%%%%%%%%%%%%%%%%%%%%%%%%%%%%%%%%%%%%%%%%%%%%%%%%%%%%%%%%%%%%%%%%%%%%%%%%%%%%%%%%%%%%%%%%%%%%%%%%%%%%%%%%%%%%%%%%%%%%%%%%%%%%%%%%%%%%%%%%%%%%%%%%%%%%%%%%%%%%%%%%%%%%%%%%%%%%%%%%%%%%%%%%%%%%%%%%%%%%%%%%%%%%%%%%%%%%%%%%%%%%%%%%%%%%%%%%%%%%%%%%%%%%%%%%%%%%%%
\section{Nodal domains near the cusp.}%S3
In this section, we investigate the structure of the nodal domains as we move away from the cusp. As is expected, the nodal domains are quite predictable very near the cusp but soon lose this feature as we move into the fundamental domain. Let $R_{\phi}$ denote the number of inert nodal domains of $\phi$ in $\mathbb{X}$, and let $Z_{\phi}$ denote the nodal line. More generally, let $\alpha$ denote an analytic arc segment in $\mathbb{X}$ and let $N^{\alpha}(\phi)$ be the number of nodal domains of $\phi$ whose boundary meets $\alpha$, so that $R_{\phi} = N^{\delta}(\phi)$. 

%%%% pg19 %%%%%%%%%%%%%%%%%%%%%%%%%%%%%%%%%%%%%%%%%%%%%%%%%%%%%%%%%%%%%%%%%%%%%%%%%%%%%%%%%%%%%%%%%%%%%%%%%%%%%%%%%%%%%%%%%%%%%%%%%%%%%%%%%%%%%%%%%%%%%%%%%%%%%%%%%%%%%%%%%%%%%%%%%%%%%%%%%%%%%%%%%%%%%%%%%%%%%%%%%%%%%%%
\subsection{A lower bound.}\ Our first result gives a lower bound for $R_{\phi}$ (and so for $N(\phi)$) by counting the number of sign changes of $\phi$ on the boundary $\delta$, since it follows from Theorem \ref{TheoremNodal} that the number of nodal domains must exceed half the number of such sign changes. 

%%%%%%%%%%%%%%%%%%%%%%%%%%%%%%%%%%%%%%%%%%%%%%%%%%%%%%%%%%%%%%%%%%%%%%%%%%%%%%%%%%%%%%%%%%%%%%%%%%%
%%%%%%%%%%%%%%%%%%%%%%%%%%%%%%%%%%%%%%%% Theorem 3.1%%%%%%%%%%%%%%%%%%%%%%%%%%%%%%%%%%%%%%%%%%%%%%%%

\begin{thm}\label{Theorem-3.1} If $\phi$ is an even Hecke-Maass cusp form, then
\[
N(\phi) \geq R_{\phi} \gg t_{\phi}.
\]
\end{thm}

This follows from the following Proposition with $x_{0}=0$;
%%%%%%%%%%%%%%%%%%%%%%%%%%%%%%%%%%%%%%%%%%%%%%%%%%%%%%%%%%%%%%%%%%%%%%%%%%%%%%%%%%%%%%%%%%%%%%%%%%%%%%%
%%%%%%%%%%%%%%%%%%%%%%%%%%%%%%%%%%%%%%% Prop 3.2%%%%%%%%%%%%%%%%%%%%%%%%%%%%%%%%%%%%%%%%%%%%%%%%%%%%%
\begin{prop}\label{Prop-3.2}  For any fixed $-\frac{1}{2} <  x_{0} \leq \frac{1}{2}$, there are at least $ct_{\phi}$ sign changes  $z=x_{0}+iy$ of $\phi(z)$ in the segment $\frac{1}{100}t_{\phi} < y < t_{\phi}$, with a positive constant $c$.
\end{prop}

\noindent{\bf Proof.} For large $y$, we expect that only the terms $n = \pm 1$ would contribute significantly in the sum \eqref{eq:b7}. Since $\phi$ is even, the zeros of $\phi$ are the same as those of
\begin{equation}\label{eq:c1}
\check{\Phi}(z) := e^{\frac{\pi}{2}t_{\phi}}{\rm K}_{it_{\phi}}(2\pi y)\cos(2\pi x_{0}) + \sum_{n\geq 2}\lambda_{\phi}(n)e^{\frac{\pi}{2}t_{\phi}}{\rm K}_{it_{\phi}}(2\pi ny)\cos(2\pi nx_{0}).
\end{equation}
Using Lemma \ref{LemmaOne}, we see that the sum above is exponentially small and that the first term is highly oscillatory due to the behavior of the Bessel function. More precisely, let $\ve >0$ be a sufficiently small fixed number and choose $y$ so that $(\frac{1}{2} + \ve)t_{\phi} < 2\pi y < (1 - \ve)t_{\phi}$. Then, we have
\[
  \check{\Phi}(z) = \frac{\sqrt{2\pi}}{(t_{\phi}^{2} - (2\pi y)^{2})^{\frac{1}{4}}}\sin \left(\frac{\pi}{4} +t_{\phi} {\rm H}(\frac{2\pi y}{t_{\phi}})\right)\cos(2\pi x_{0}) \left(1 + O(\frac{1}{t_{\phi}})\right) + O(e^{-\ve t_{\phi}}).
\]
For any positive integer $m$, let $y_{m}$ be the unique solution to ${\rm H}(\frac{2\pi y}{t_{\phi}}) = \frac{\pi m}{t_{\phi}}$ so that if $x_{0} \neq \pm \frac{1}{4}$ (or not close to it), we have 
\[
\frac{(t_{\phi}^{2} - (2\pi y_{m})^{2})^{\frac{1}{4}}}{\sqrt{\pi}} \sec(2\pi x_{0})\check{\Phi}(x_{0}+iy_{m}) = (-1)^{m} + O(\frac{1}{t_{\phi}}).
\]
Due to the restriction imposed on $y$, it follows that one may take $1 \leq m \ll t_{\phi}$, so that we get at least $ct_{\phi}$ sign changes of $\phi$.

It remains to consider the lines $x_{0} = \pm \frac{1}{4}$ (and those $x_{0}$ close by). The term $n=2$ is now the main term unless $|\lambda_{\phi}(2)|$ is very small. So first suppose that $|\lambda_{\phi}(2)| > \frac{1}{\log t_{\phi}}$, so that
\[
\check{\Phi}(z) =  -\lambda_{\phi}(2)e^{\frac{\pi}{2}t_{\phi}}{\rm K}_{it_{\phi}}(4\pi y) + \sum_{n\geq 2}\lambda_{\phi}(2n)e^{\frac{\pi}{2}t_{\phi}}{\rm K}_{it_{\phi}}(4\pi ny)(-1)^{n} + O(e^{-\ve t_{\phi}}).
\]
The analysis is now the same as that given above for the case $x_{0}\neq \frac{1}{4}$, with $y$ replaced with $2y$, so that we assume $(\frac{1}{2} + \ve)t_{\phi} < 4\pi y < (1 - \ve)t_{\phi}$ and the same conclusion follows.

Next, if $|\lambda_{\phi}(2)| \leq \frac{1}{\log t_{\phi}}$, the Hecke relations imply that $|\lambda_{\phi}(4)| \geq \frac{1}{2}$, so that this time the term involving $\lambda_{\phi}(2)$ can be absorbed into the error term so that we obtain
\begin{multline}
 \quad \check{\Phi}(z) =  \lambda_{\phi}(4)e^{\frac{\pi}{2}t_{\phi}}{\rm K}_{it_{\phi}}(8\pi y)(1 + O(\frac{1}{\log t_{\phi}})) 
+ \cr
\sum_{n\geq 3}\lambda_{\phi}(2n)e^{\frac{\pi}{2}t_{\phi}}{\rm K}_{it_{\phi}}(4\pi ny)(-1)^{n} + O(e^{-\ve t_{\phi}})\ .\quad\quad
\end{multline}
The conclusion then follows if we restrict $y$ so that $(\frac{1}{3} + \ve)t_{\phi} < 4\pi y < (\frac{1}{2} - \ve)t_{\phi}$.
%%%%%%%%%%%%%%%%%%%%%%%%%%%%%%%%%%%%%%%%%%%%%%%%%%%%%%%%%%%%%%%%%%%%%%%%%%%%%%%%%%%%%%%%%%%%%%%%%%%%%%%%%%%%%%%%%%%%%%%%%%%%%%%%%%%%%%%%%%%%%%%%%%%%%%%%%%%%%%%%%%%%%%%%%%%%%%%% pg20 %%%%%%%%%%%%%%%%%%%%%%%%%%%%%%%%%%%
\subsection{Analysis away from the cusp.}\ To describe the nodal domains further, we will need some results regarding the distribution of zeros of the functions ${\rm K}_{ir}(u)$. It is known that (see \cite{Du90}, \cite{Ol74}) ${\rm K}_{ir}(u)$ has infinitely many positive simple zeros in the interval $0<u<r$, and has no zeros for $r\leq u < \infty$. We denote the zeros by $k_{r,j}$ with $ j=1,2, ...$ such that $r> k_{r,1}>k_{r,2} > ... > 0$. For $r$ fixed and $j \rightarrow \infty$, we have $k_{r,j} \rightarrow 0$. On the other hand for any fixed $j$, with $r$ growing, one has 
\begin{equation}\label{eq:c2}
k_{r,j} = r + a_{j}\left(\frac{r}{2}\right)^{\frac{1}{3}} + \frac{3}{20}a_{j}^{2}\left(\frac{r}{2}\right)^{-\frac{1}{3}} + O\left(r^{-\frac{2}{3}}\right),
\end{equation}
where $a_{j}$ denotes the (negative) zeros of the Airy function ${\rm Ai}(x)$, with $0>a_{1}$$>a_{2}> ...$ . It is known that \cite{Ol74}
\begin{equation}\label{eq:c3}
a_{j} = - \left(\frac{3\pi}{8}(4j-1)\right)^{\frac{2}{3}} + O\left(j^{-\frac{4}{3}}\right),
\end{equation}
as $j\rightarrow \infty$. The first few zeros of the Airy function are approximately $a_{1}= -2.33811$, $a_{2}=-4.08795$, $a_{3}= -5.52056$ etc. .

We recall the definition of $y_{l}$ from \eqref{eq:b8}, so that if $y > y_{1}$, ${\rm K}_{it_{\phi}}(2\pi y)$ is nonzero (the first zero occurs in the transitional region). Then by Lemma \ref{LemmaOne}, Corollary \ref{CorlOne} and \eqref{eq:c1}, we have 
\begin{equation}\label{eq:c4}
\check{\Phi}(z) = e^{\frac{\pi}{2}t_{\phi}}{\rm K}_{it_{\phi}}(2\pi y)\left(\cos(2\pi x) + O(e^{-\ve t_{\phi}})\right),
\end{equation}
for some small $\ve >0$. It thus follows that the nodal line $Z_{\phi}$ must lie in an exponentially narrow neighborhood of the lines $x = \pm \frac{1}{4}$. If, as $y$ varies in this range away from $y_{1}$,  $Z_{\phi}$ does not intersect these vertical lines, then there are only two possible configurations for the nodal domains here, illustrated by Figure 9. The question of which of the two occurs is completely determined by the sign of the second Fourier coefficient $\lambda_{\phi}(2)$, assuming it is not zero (because the Bessel functions are of constant sign). Thus, if $\lambda_{\phi}(2)>0$, we have Figure 9(a) and Figure 9(b) if $\lambda_{\phi}(2)<0$ (the black dots in the figures are the points $\frac{1}{2\pi}k_{t_{\phi},j}$ with $j = 1$ and $2$). Under these conditions the nodal domains here would be inert. It is conceivable that $Z_{\phi}$ may intersect the line $x = \pm \frac{1}{4}$ but that would occur only if $\lambda_{\phi}(2)$ is non-zero but exponentially small as $t_{\phi}$ grows, which is probably unlikely (at any event, it is easily shown that there cannot be more than one such intersection in this region).

\begin{figure}[ht]
  \begin{center}
    \includegraphics[width=4.5in]{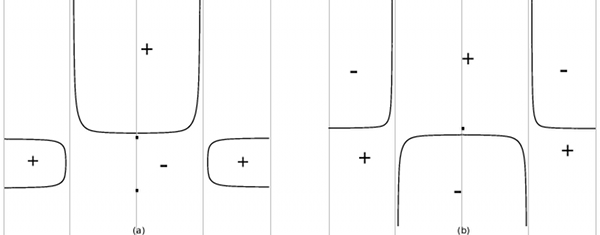}
  \end{center}
  \caption{}
\end{figure}

\begin{figure}[ht]
  \begin{center}
    \includegraphics[width=4.5in]{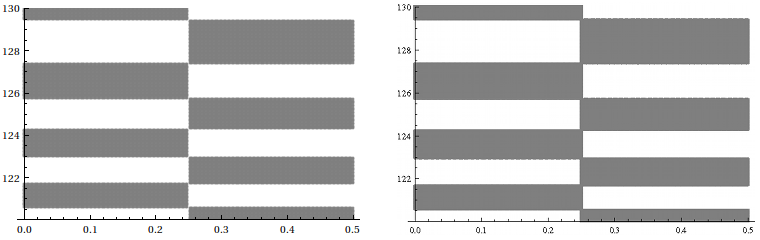}
  \end{center}
  \caption{}
\end{figure}

The analysis above can be extended to the region $y_{2} + O(t^{\frac{1}{3}}) < y \leq y_{1}$ except for neighborhoods of those $y$'s close to the points $\frac{1}{2\pi}k_{t_{\phi},j}$ (because in this range, the terms in the sum in \eqref{eq:c1} have exponential decay). In other words, $Z_{\phi}$ is exponentially close to the vertical lines $x = \pm \frac{1}{4}$ or the appropriate horizontal lines $y = \frac{1}{2\pi}k_{t_{\phi},j}$.  Now suppose there is a split nodal domain within this region. Then necessarily this domain must be contained in a band of exponentially small thickness that weaves about the aforementioned lines. Since the length of $Z_{\phi}$ is $O(t_{\phi})$ (see \cite{DF88}) the area $A$ of such a band is exponentially small. However, by the Faber-Krahn Theorem (see Prop. \ref{Theorem-3.4} below) it is necessarily the case that $A \gg t_{\phi}^{-2}$ (here we assumed that $Z_{\phi}$ is a smooth curve) so that all nodal domains for this range of $y$ are inert. 
%pg21%

We illustrate this analysis, using data from \cite{Th05},  with $t_{\phi}=830.42904848...$ and $t_{\phi}= 830.43027421...$  in Fig. 10 (showing only half of the fundamental domain), where $\lambda_{\phi}(2) >0$ and $<0$ respectively. Here, $y_{1} \sim 132.167$, $y_{2} \sim 66.08$, $\frac{1}{2\pi}k_{t_{\phi},1} \sim 129.40$ and $\frac{1}{2\pi}k_{t_{\phi},2} \sim 127.36$. The shaded regions represent the points where $\phi(x+iy)$ is positive, so that the boundaries represent $Z_{\phi}$. The topmost regions extend to infinity without any change and the first horizontal cut occurs near $y=129$. 

The situation changes rather dramatically as we move into the region $y_{3}<y<y_{2}$ since now 
\begin{equation}\label{eq:c5}
\check{\Phi}(z) = \left(e^{\frac{\pi}{2}t_{\phi}}{\rm K}_{it_{\phi}}(2\pi y)\cos(2\pi x)+ \lambda_{\phi}(2)e^{\frac{\pi}{2}t_{\phi}}{\rm K}_{it_{\phi}}(4\pi y)\cos(4\pi x)\right)\left(1 + O(e^{-\ve t_{\phi}})\right),
\end{equation}
so that the zeros of the main term above form a curve $Z^{*}_{\phi}$ close to $Z_{\phi}$.  To determine the points on $Z^{*}_{\phi}$, it is not hard to verify that for some values of  $y$ there is a uniquely determined $x$ and for the rest, there are two choices of $x$ not close to each other in general (we assume here, for simplicity that $\lambda_{\phi}(2) \neq 0$) . Moreover, there are no $y$'s which give rise to points in the interior for which $Z^{*}_{\phi}$ has a horizontal tangent except for those $y$'s that satisfy the simultaneous condition that ${\rm K}_{it_{\phi}}(2\pi y)$ and ${\rm K}_{it_{\phi}}(4\pi y)$ are both small. Since split nodal domains must have points (not on the boundary $\delta$) with horizontal tangents, we expect that there are few (if any) such for this range of $y$. In particular, we expect that most nodal domains here are inert. Figure 11(a) illustrates the domains for $t_{\phi}=830.42904848...$ with 11(b) a close-up view, while Figure 12 does the same with $t_{\phi}= 830.43027421...$. Here $y_{2}\sim 66.08$ and $y_{3} \sim 44.05$ and we observe the influence of the vertical lines $x=\frac{1}{8}$ and $x=\frac{3}{8}$. The close-up view also shows the lack of horizontal tangents except on the boundary.

\begin{figure}[ht]
  \begin{center}
    \includegraphics[width=4.9in]{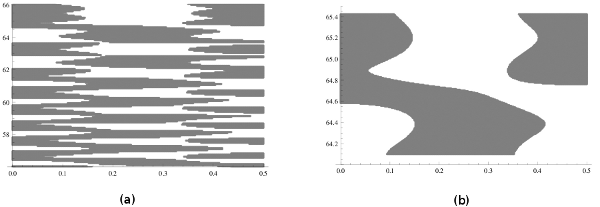}
  \end{center}
  \caption{}
\end{figure}

\begin{figure}[ht]
  \begin{center}
    \includegraphics[width=4.9in]{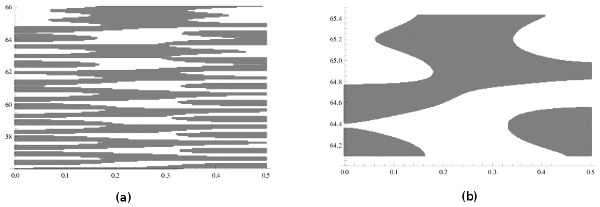}
  \end{center}
  \caption{}
\end{figure}

While we are unable to be more precise when we consider extended ranges of $y$, the following Theorem shows that there is some form of regularity (locally) in the neighborhoods of many special values of $y$:

%%%%%%%%%%%%%%%%%%%%%%%%%%%%%%%%%%%%%%%%%%%% Theorem 3.2 %%%%%%%%%%%%%%%%%%%%%%%%%%%%%%%%%%%%%%%%%
\begin{thm}\label{Theorem-3.2}  Let $\ve > 0$ be sufficiently small and $\phi$ an even Maass cusp form. Let $l$ be an integer satisfying $ 1 \ll l \ll t_{\phi}^{\frac{1}{6} - \ve}$ and assume that $|\lambda_{\phi}(l)| \geq \ve$. Put $y_{l}$ as in \eqref{eq:b8} so that $t_{\phi}^{\frac{5}{6} + \ve} \ll y_{l} \ll \ve t_{\phi}$. Then in the bounded region in the strip $-\frac{1}{2}< x\leq \frac{1}{2}$ with  $|y - y_{l}| \leq \ve y_{l} t_{\phi}^{-\frac{2}{3}}$, the nodal line $Z_{\phi}$ is  located in narrow neighborhoods (going to zero with $t$) of the lines $x = \frac{m}{4l} $ with $|m| \leq 2l$ odd integers.
%pg22%
\end{thm}
\begin{proof} We use the second part of Proposition \ref{PropTwo}. By the Cauchy-Schwarz inequality, the sum is bounded by
\begin{equation}\label{eq:c6}
\left(\sum_{|n|<l}|\lambda_{\phi}(n)|^{2}\right)^{\frac{1}{2}} \left(\sum_{1}^{l-1} \frac{1}{(t_{\phi}^{2}-(2\pi ny)^{2})^{\frac{1}{2}}}\right)^{\frac{1}{2}} \ll (\frac{l}{y})^{\frac{1}{2}(1+ 100\ve)},
\end{equation}
by Lemma \ref{LemmaTwo} and Lemma \ref{LemmaZero}. The conditions on $y$ ensures that $y \sim y_{l}$ so that
\begin{multline*}
\Phi(x+iy)=  \pi \left(\frac{16}{t_{\phi}}\right)^{\frac{1}{3}}\lambda_{\phi}(l){\rm Ai}\left((2\pi l(y-y_{l}))(\pi ly)^{-\frac{1}{3}}\right)\cos(2\pi lx) \cr
+ O(t_{\phi}^{-\frac{1}{3} -\ve})\ ,\quad\quad\quad
\end{multline*}
provided $l \ll t_{\phi}^{\frac{1}{6} - \ve}$. Moreover, the Airy function is asymptotically close to ${\rm Ai}(0)$ so that the assumption on the size of $\lambda_{\phi}(l)$ ensures that the sign changes of $\phi$ are determined by those of $\cos(2\pi lx)$, from which the Theorem follows.
\end{proof}

%%%%%%%%%%%%%%%%%%%%%%%%%%%%%%%%%%%%%%%%%%%%%%%%%%%%%%%%%%%%%%%%%%%%%%%%%%%%%%%%%%%%%%%%%%%%%%%%%%%%%%%%%%%%%%%%%%%%%%%%%%%%%%%%%%%%%%%%%%%%%%%%%%%%%%%%%%%%%%%%%%%%%%%%%%%%%%%%%%%%%%%%%%%%%%%%%%%%%%%%%%%%%%%%%%%
\subsection{An upper bound.}\ We end this section by giving an extension of the Courant nodal domain theorem for regions close to the cusp. 
\begin{thm}\label{Theorem-3.3} Suppose $Z_{\phi}$ is a smooth curve. Then, for any $\ve_{0} > 0$ sufficiently small, there are at most $O_{\ve_{0}}(t_{\phi})$ nodal domains of $\phi$ in the region $y > \ve_{0} t_{\phi}$.
\end{thm}
%pg23%
This Theorem, taken together with Theorem \ref{Theorem-3.1} suggest the possibility that almost all the nodal domains near the cusp are inert. For $l>l'\geq 0$, put
\begin{equation}\label{eq:c7}
\mathbb{D}_{l',l} = \big{\{} -\frac{1}{2}\leq x\leq \frac{1}{2}, \quad y_{l}\leq y \leq y_{l'}\Big{\}},
\end{equation}
where we put $y_{0}=\infty$. The Theorem is a consequence of the following

\begin{prop}\label{Theorem-3.4} Suppose $Z_{\phi}$ is a smooth curve.
\begin{itemize}
\item[(1)] Let $1\leq l'\leq l \ll t_{\phi}^{\frac{1}{6}-\ve}$ be integers such that $|\lambda_{\phi}(l)|$ and  $|\lambda_{\phi}(l')|$ both exceed $\ve $  for some $\ve>0$ sufficiently small. Then, the number of nodal domains of $\phi$ in $\mathbb{D}_{l',l}$ is $O_{\ve}((l-l')t_{\phi})$.
\item[(2)] Suppose $l \ll t_{\phi}^{\frac{1}{6}-\ve}$ and  $|\lambda_{\phi}(l)| \geq \ve$. Then the number of nodal domains of $\phi$ in the region $y>y_{l}$ is $O_{\ve}(\frac{t_{\phi}^{2}}{y_{l}})$.
\end{itemize}
\end{prop}

\begin{proof}

 Let $S_{1}, S_{2}, ... S_{N}$ be the nodal domains of $\phi$ wholly contained in $\mathbb{D}_{l',l}$ (so that they do not intersect the horizontal boundaries). Since $\phi$ is an eigenfunction of $\Delta$ with eigenvalue $\lambda$, on each $S_{i}$, $\lambda$ is the least non-zero eigenvalue of the associated Dirichlet boundary value problem. We appeal to the Faber-Krahn Theorem for bounded domains in hyperbolic space (see \cite{Ch84}, Chap. 4.2) which states that if the boundary of $S_{i}$ is smooth (which we will assume), there is a constant $C >0$ such that
\[
{\rm Area}(S_{i})\lambda > C.
\]
Summing over all the $S_{i}$'s gives us
\[
N \leq \frac{\lambda}{C}\sum_{i}{\rm Area}(S_{i}) \leq \frac{\lambda}{C}{\rm Area}(\mathbb{D}_{l',l})\ll (l-l')t_{\phi},
\]
since ${\rm Area}(\mathbb{D}_{l',l})=\frac{2\pi (l-l')}{t_{\phi}}$.

The remaining nodal domains necessarily intersect the horizontal boundaries $y_{l}$ or $y_{l'}$ so that the number of such nodal domains is bounded by the number of zeros of $\phi$ on these closed horocycles. Applying the formula \eqref{eq:c6} with $y=y_{l}$ and with $1\leq l \ll t_{\phi}^{\frac{1}{6}-\delta}$, we have
\[
\Phi(z_{l}) = \frac{C'}{l^{\frac{1}{3}}}\cos(2\pi lx) \lambda_{\phi}(l) + O(l^{-\frac{1}{3} - \delta}),
\]
for some constant $C'>0$ and for any small $\delta>0$. Now suppose that $|\lambda_{\phi}(l)|>\ve$. It follows that if $\alpha$ is a zero of $\Phi(x +iy_{l})$, then $\alpha$ is near $\frac{n}{4l}$ for some odd integer $n$. This zero is necessarily simple since Prop. \ref{PropTwo} can be extended to $\frac{\partial}{\partial x} \Phi(x+iy)$ so that the main term is not small at $\alpha$. We now consider the complexification of $\Phi(x+iy_{l})$, with  $x$ a complex variable. The estimates for the Fourier coefficients from \eqref{eq:b2}, \eqref{eq:b3}, \eqref{eq:b4} and the Bessel function show that $\phi(x+iy_{l})$ is a holomorphic function of $x\in \mathbb{C}$ for $|x-\alpha|\leq \frac{y_{l}}{4t} =\frac{1}{8\pi l}$, for any real number $\alpha$. Applying Jensen's Lemma, in each ball of radius $\frac{1}{16\pi l}$ centered at $\alpha$ there are $O(\log t_{\phi})$ zeros so that $\Phi(x+iy_{l})$ has at most $O(l\log t_{\phi})$ real zeros for $-\frac{1}{2}\leq x \leq \frac{1}{2}$. Part one of the proposition then follows. The second part follows using $l' = 1$ since $\lambda_{\phi}(1)=1$. 

To complete the proof of Theorem \ref{Theorem-3.3}, we have to choose $l$ in Prop. \ref{Theorem-3.4} such that our condition $|\lambda_{\phi}(l)| \geq \ve_{0}$ is satisfied. By the Hecke relations, for any prime number $p$, this condition is satisfied with $l = p$ or $p^{2}$. Choosing $\ve_{0} = \frac{1}{l}$ and $p$ sufficiently large but independent of $t$ gives us our conclusion. 

\end{proof} 
%%%%%%%%%%%%%%%%%%%%%%%%%%%%%%%%%%%%%%%%%%%%%%%%%%%%%%%%%%%%%%%%%%%%%%%%%%%%%%%%%%%%%%%%%%%%%%%%%%%%%%%%%%%%%%%%% pg24 %%%%%%%%%%%%%%%%%%%%%%%%%%%%%%%%%%%%%%%%%%%%%%%%%%%%%%%%%%%%%

\section{Restriction to closed horocycles}%S2

For $0<Y<\infty$, denote by $\mathcal{C}_{Y}$ the closed horocycle on $\mathbb{X}$ given by the points $z = x+iY$, with $-\frac{1}{2} < x\leq \frac{1}{2}$. We are interested in the restriction of $\phi$ to $\mathcal{C}_{Y}$. We first give a sharp bound for the $L^{2}$-restriction in

%%%%%%%%%%%%%%%%%%%%%%%%%%%%%%%%%%%%%%%%%%%%%%%% Theorem 1 %%%%%%%%%%%%%%%%%%%%%%%%%%%%%%%%%%%%%%%%%%%%
%%%%%%%%%%%%%%%%%%%%%%%%%%%%%%%%%%%%%%%%%%%%%%%%%%%%%%%%%%%%%%%%%%%%%%%%%%%%%%%%%%%%%%%%%%%%%%%%%%%%%%
\begin{thm}\label{TheoremOne} \
\begin{itemize}
\item[(1)] For any fixed $Y>0$ and any $\ve > 0$
\[
t_{\phi}^{-\ve} \ll \lVert\phi\big{|}_{_{\mathcal{C}_{Y}}}\rVert_{_{2}}^{2}=\int_{-\frac{1}{2}}^{\frac{1}{2}}|\phi(x+iY)|^{2} \ \d x \ll t_{\phi}^{\ve},
\]
where the implied constants depend on $\ve$ and $Y$.
\item[(2)] For all $c < Y < \frac{t_{\phi}}{2\pi}$ with any fixed $c>0$ we have
\[
\lVert\phi\big{|}_{_{\mathcal{C}_{Y}}}\rVert_{_{2}} \ll_{_{\ve}} t_{\phi}^{\ve}.
\]
\end{itemize}
\end{thm}

Using this, we then prove

%%%%%%%%%%%%%%%%%%%%%%%%%%%%%%%%%%%%%Theorem 2%%%%%%%%%%%%%%%%%%%%%%%%%%%%%%%%%%%%%%%%%%%%%%%%%%%%%%%%
%%%%%%%%%%%%%%%%%%%%%%%%%%%%%%%%%%%%%%%%%%%%%%%%%%%%%%%%%%%%%%%%%%%%%%%%%%%%%%%%%%%%%%%%%%%%%%%%%%%%%%%
\begin{thm}\label{TheoremTwo} For $\ve > 0$ and fixed $ Y > 0$
\[
t_{\phi}^{\frac{1}{12}-\ve} \ll_{\ve} \#\{z\in \mathcal{C}_{Y}: \phi(z)=0\} \ll t_{\phi}.
\]
\end{thm}

\begin{rem} It is easy to see that if \;$ Y > At_{\phi} $\; for some sufficiently large constant $A$, then there are two zeros of $\phi(z)$ on the horocycle $\mathcal{C}_{Y}$. This follows from the exponential decay of the Bessel function so that the term $n=1$ dominates $\Phi(z)$. Moreover, since there are no real zeros of ${\rm K}_{it_{\phi}}(u)$ with $u \geq t_{\phi}$, the cosine for even Maass forms (or sine for odd forms) provides the sign change as $x$ varies.
\end{rem}

\begin{rem} Jung \cite{Ju11} has recently established a lower bound of $e^{-At_{\phi}}$ for $ \lVert\phi|_{_{\mathcal{C}_{Y}}}\rVert_{_{2}}^{2}$ in Theorem \ref{TheoremOne} without having to assume that $\phi$ is a Hecke eigenform (such a bound follows trivially from the argument given in Section 2.2 for Hecke eigenforms). This, together with results of Toth and Zelditch \cite{TZ09} allows him to establish the upper bound in Theorem \ref{TheoremTwo} of $O(t_{\phi})$ for such $\phi$'s and, in fact, for any hyperbolic surface with finite area (without any arithmetic assumptions).
\end{rem}

\begin{rem} In Appendix A, we obtain the analog of the above theorems as $Y$ moves towards the cusp. The lower bound above depends on standard QUE and is responsible for our assumption that $Y$ be bounded. It is not  sufficient for our purposes as $Y$ grows with $t$ for which we  appeal to a conjectured quantitative version of QUE (see also the introduction for precise statements). 

\end{rem} 

While the theorems above hold for fixed $Y$, we show in Theorem \ref{TheoremThree} below that as we approach the cusp with $Y \gg t_{\phi}^{\frac{5}{6} +\ve}$, there are some horocycles with many zeros that are evenly spaced out on the horocycle (see also Theorem \ref{Theorem-3.2}).

%%%%%%%%%%%%%%%%%%%%%%%%%%%%%%%%%%%%%%%%% Theorem 3 %%%%%%%%%%%%%%%%%%%%%%%%%%%%%%%%%%%%%%%%%%%%%%%%%%%
%%%%%%%%%%%%%%%%%%%%%%%%%%%%%%%%%%%%%%%%%%%%%%%%%%%%%%%%%%%%%%%%%%%%%%%%%%%%%%%%%%%%%%%%%%%%%%%%%%%%%%%
\begin{thm}\label{TheoremThree} Let $\ve > 0$ be sufficiently small and $\phi$ an even Maass cusp form.
There are $\gg t_{\phi}^{\frac{1}{12} -\ve}$ integers $l$ satisfying $ 1 \ll l \ll t_{\phi}^{\frac{1}{6} - \ve}$ such that there are  $2l$ zeros of  $\phi$ on the horocycle $\mathcal{C}_{Y_{l}}$, with $Y_{l} = \frac{t_{\phi}}{2\pi l}$. Thus, there are a growing number of horocycles $\mathcal{C}_{Y}$ such that each have $\frac{1}{\pi}\frac{t_{\phi}}{Y}$ zeros of $\phi$.

\end{thm}
\begin{rem}  The asymptotic behavior of $\#\{z\in \mathcal{C}_{Y}: \phi(z)=0\}$ in Theorem \ref{TheoremTwo} should  follow that predicted by a random wave model. In Appendix B we show that such a model predicts that for $Y\ll t_{\phi}^{\frac{1}{3} - \ve}$
\[
\#\{z\in \mathcal{C}_{Y}: \phi(z)=0\} \sim \frac{1}{\pi}\frac{t_{\phi}}{Y},
\]
which is consistent with the result in the above Theorem.
\end{rem}

%%%% pg26 %%%%%%%%%%%%%%%%%%%%%%%%%%%%%%%%%%%%%%%%%%%%%%%%%%%%%%%%%%%%%%%%%%%%%%%%%%%%%%%%%%%%%%%%%%%%%%%%%%%%%
\subsection*{}

{\noindent \bf Proof of Theorem \ref{TheoremOne}.}  For $Y$ and $\Delta$ as in Prop. \ref{PropOne}, we have
\begin{multline*}
\quad\int_{-\half}^{\half} |\phi(x+iY)|^{2} \ \d x  \ll |\rho_{\phi}(1)|^{2}e^{-\pi t_{\phi}}\sum_{|n|\leq \frac{t_{\phi}-\Delta}{2\pi Y}} |\lambda_{\phi}(n)|^{2} \frac{Y}{\sqrt{t_{\phi}^{2} - (2\pi nY)^{2}}}  \cr
 + O\big{(}t_{\phi}^{\theta -\frac{4}{3}+\ve}(\Delta + Y)\big{)}\ ,\quad\quad
\end{multline*}
with the error term uniform in $Y$. It is easily seen that the error term is $O(t_{\phi}^{\ve})$ so that applying the Cauchy-Schwarz inequality and using \eqref{eq:b2}, we have
\[
\int_{-\half}^{\half} |\phi(x+iY)|^{2} \ \d x  \ll t_{\phi}^{2\ve}Y\left(\sum_{|n|\leq \frac{t_{\phi}}{2\pi Y}}|\lambda_{\phi}(n)|^{4}\right)^{\frac{1}{2}}\left(\sum_{|n|\leq \frac{t_{\phi}-\Delta}{2\pi Y}}\frac{1}{t_{\phi}^{2} - (2\pi nY)^{2}}\right)^{\frac{1}{2}}.
\]
It follows from Lemma \ref{LemmaTwo} that we obtain the desired upper-bound (for $Y$ not fixed), since the right hand side above is
\[
\ll t_{\phi}^{2\ve}Y(\frac{t_{\phi}}{Y})^{\frac{1+\ve}{2}}(t_{\phi}Y)^{-\frac{1}{2}}(\log t_{\phi})^{\frac{1}{2}} \ll t_{\phi}^{4\ve}.
\]
For the lower bound, using Lemma \ref{LemmaOne} and discarding some terms from the resulting sum by positivity, we have for any $Y>0$ (still not fixed)
\begin{equation}\label{eq:d2}
 \int_{-\half}^{\half} |\phi(x+iY)|^{2} \ \d x \gg Y|\rho_{\phi}(1)|^{2}e^{-\pi t_{\phi}}\frac{1}{t_{\phi}}\sum_{|n|\leq \frac{t_{\phi}}{4\pi Y}} |\lambda_{\phi}(n)|^{2} \sin ^{2} \left(\frac{\pi}{4} +t_{\phi} {\rm H}(\frac{2\pi |n|Y}{t_{\phi}})\right),
\end{equation}
with $H$ as given in \eqref{eq:b6}.
Let $0 < \eta < 1$ be small and define the set
\[
\mathcal{S}' = \mathcal{S}_{\eta}'= \Big{\{} |n|\leq \frac{t_{\phi}}{4\pi Y}:\quad \Big{|}\sin \left(\frac{\pi}{4} +t_{\phi} {\rm H}(\frac{2\pi |n|Y}{t_{\phi}})\right)\Big{|} \leq \eta\Big{\}}. 
\]
Then, $\mathcal{S}'$ is contained in the set
\begin{equation}\label{eq:d3}
\mathcal{S}=\Big{\{} |n|\leq \frac{t_{\phi}}{4\pi Y}:\Big{|}\frac{1}{4} +\frac{t_{\phi}}{\pi} {\rm H}(\frac{2\pi |n|Y}{t_{\phi}})-m\Big{|} \leq 10\eta, \; \text{for some} \quad m\in\mathbb{Z}\Big{\}}.
\end{equation}
The pair of integers $(n,m)$ as above are lattice points in $\mathbb{Z}^{2}$ which are close to the real analytic curve
\begin{equation}\label{eq:d4}
y=\frac{1}{4} +\frac{t_{\phi}}{\pi} {\rm H}(\frac{2\pi xY}{t_{\phi}}),
\end{equation}
with $x$ in the appropriate range. Since this curve has no line segments, it is known since Van der Corput (see Huxley \cite{Hu03}) that the number of such lattice points satisfies
\begin{equation}\label{eq:d5}
|\mathcal{S}| \ll \eta\frac{t_{\phi}}{Y} + O(t_{\phi}^{\frac{2}{3}}Y^{-\frac{1}{3}}),
\end{equation}
provided $Y \ll \eta^{2}\sqrt{t_{\phi}}$.
We choose $\eta = t_{\phi}^{-\delta}$ with $\delta > 0$ small and fixed. The contribution to the sum in \eqref{eq:d2} from those $n$ in $\mathcal{S}$ is 
\begin{equation}\label{eq:d6}
\ll \eta^{2}\left(\sum_{n\in \mathcal{S}} 1\right)^{\frac{1}{2}}\left(\sum_{n\leq \frac{t_{\phi}}{Y}}|\lambda_{\phi}(n)|^{4}\right)^{\frac{1}{2}} \ll_{\ve} \eta^{2}(\eta \frac{t_{\phi}}{Y})^{\frac{1}{2}}(\frac{t_{\phi}}{Y})^{\frac{1+\ve}{2}} \ll \eta^{\frac{5}{2}}(\frac{t_{\phi}}{Y})^{1+\frac{\ve}{2}}
\end{equation}
using \eqref{eq:d5} and Lemma \ref{LemmaTwo}. On the other hand for $n$ not in $\mathcal{S}$, the contribution is 
\[
\geq \eta^{2}\left( \sum_{n\leq \frac{t_{\phi}}{4\pi Y}}|\lambda_{\phi}(n)|^{2} - \sum_{n\in \mathcal{S}}|\lambda_{\phi}(n)|^{2}\right)
\]
so that on combining with the previous analysis and \eqref{eq:d2}, we get for $0<c<Y\ll t_{\phi}^{\frac{1}{2}-\ve}$
\begin{equation}\label{eq:d7}
\int_{-\half}^{\half} |\phi(x+iY)|^{2} \ \d x \geq \frac{\eta^{2}Y}{t_{\phi}}\sum_{n\leq \frac{t_{\phi}}{4\pi Y}}|\tilde{\rho}_{\phi}(n)|^{2}  +O_{\ve}(\eta^{\frac{5}{2}}t_{\phi}^{3\ve}),
\end{equation}
which gives us our lower bound on applying Prop. \ref{shortsum}, for any fixed $Y$.

%%%% pg26 %%%%%%%%%%%%%%%%%%%%%%%%%%%%%%%%%%%%%%%%%%%%%%%%%%%%%%%%%%%%%%%%%%%%%%%%%%%%%
\subsection*{}
{\noindent \bf Proof of Theorem \ref{TheoremTwo}.} 

The upper bound is a consequence of the complexification argument of Toth-Zelditch (\cite{TZ09}, Theorem 6) where the lower bound obtained in Theorem \ref{TheoremOne} above (or more generally the weaker lower bound of \cite{Ju11}) is used to satisfy the notion of a ``good curve" in \cite{TZ09}.  It remains to prove the lower bound in the statement of our Theorem.

Let $M$ be the number of sign-changes of $\phi(z)$ on $\mathcal{C}_{Y}$, say at the points $\{ \alpha_{j}\}$ with $1\leq j\leq M$ (we put $\alpha_{0}=-\half$ and $\alpha_{M+1}=\half$). Then,
\begin{equation}\label{eq:d8}
\int_{\mathcal{C}_{Y}}|\phi(z)| \ \d x = \sum_{i=0}^{M}\Big{|}\int_{\alpha_{j}}^{\alpha_{j+1}} \phi(x+iY) \ \d x \Big{|}.
\end{equation}
Using the $L^{\infty}$ estimate $|\phi(z)| \ll t_{\phi}^{\frac{5}{12}+\ve}$ in \cite{IS95} and Theorem \ref{TheoremOne}, we have
\[
t_{\phi}^{-\ve} \ll \int_{\mathcal{C}_{Y}}|\phi(z)|^{2} \ \d x \ll t_{\phi}^{\frac{5}{12}+\ve} \int_{\mathcal{C}_{Y}}|\phi(z)| \ \d x.
\]
On the other hand, by Prop. \ref{PropOne}, the second integral  in \eqref{eq:d8} is bounded by 
\[
t_{\phi}^{\ve}\sum_{0<|n|\leq \frac{t_{\phi}-\Delta}{2\pi Y}} \frac{|\lambda_{\phi}(n)|}{|n|} \frac{\sqrt{Y}}{(t_{\phi}^{2}-(2\pi nY)^{2})^{\frac{1}{4}}},
\]
which is then bounded by $t_{\phi}^{-\frac{1}{2}+\ve }$ after using the Cauchy-Schwarz inequality, Lemma \ref{LemmaZero} and Lemma \ref{LemmaTwo}. Then from \eqref{eq:d8} we conclude 
\[
t_{\phi}^{-\ve } \ll t_{\phi}^{\frac{5}{12}+\ve}Mt_{\phi}^{-\frac{1}{2}+\ve}
\]
from which our lower bound for $M$ follows. If one were to use the conjectural $L^{\infty}$ estimate $|\phi(z)| \ll t_{\phi}^{\ve}$, valid for $z$ in a compact part of $\mathbb{H}$, then one would obtain $M \gg t_{\phi}^{\frac{1}{2}-\ve}$.

%%%%%%%%%%%%%%%%%%%%%%%%%%%%%%%%%%%%%%%%%%%%%%%%%%%%%%%%%%%%%%%%%%%%%%%%%%%%%%%%%%
\subsection*{}

{\noindent \bf Proof of Theorem \ref{TheoremThree}.} The proof is the same as in Theorem \ref{Theorem-3.2}  so that
\begin{equation}\label{eq:d9}
 \Phi(z_{l}) =  \pi \left(\frac{16}{t_{\phi}}\right)^{\frac{1}{3}}\lambda_{\phi}(l){\rm Ai}(0)\cos(2\pi lx) +O(\frac{l}{\sqrt{t_{\phi}}}) + O(t_{\phi}^{\theta - \frac{2}{3}}).
 \end{equation}
Thus, if $l = o(t_{\phi}^{\frac{1}{6}})$, the error terms are negligible provided $\lambda_{\phi}(l)$ is not too small. We use the Hecke relations
\begin{equation}\label{eq:d10}
\lambda_{\phi}(p)^{2} = \lambda_{\phi}(p^2) +1
\end{equation}
valid for all prime numbers $p \geq 2$.  It follows that either $|\lambda_{\phi}(p)| \geq \frac{1}{2}$ or if not, then $|\lambda_{f}(p^2)| \geq \frac{1}{2}$. We then choose $l$ accordingly, so that now we can find $\gg t_{\phi}^{\frac{1}{12}-\ve}$ values of $l$ in the interval $1\leq l \ll t_{\phi}^{\frac{1}{6} -\ve}$ with the main term dominating in \eqref{eq:d9}. For each such fixed $l$, we then see that there are $2l$ sign changes of $\phi(x+iy_{l})$ for $-\frac{1}{2} < x\leq \frac{1}{2}$.

%%%%%% pg27 %%%%%%%%%%%%%%%%%%%%%%%%%%%%%%%%% e = Section 4 %%%%%%%%%%%%%%%%%%%%%%%%%%%%%%%%%%%%%%%%%%%%%%%%%%%%%%%%%%%%%%%%%%%%%%%%%%%%%%%%%%%%%%%%%%%%%%%%%%%%%%%%%%%%%%%%%%%%%%%%%%%%%%%%%%%%%%%%%%%%%%%%%%%%%
\section{Restriction to geodesics and compact segments of $\delta$.}

In this section, which is the heart of the paper, we prove Theorem \ref{TheoremI2} and Theorem \ref{TheoremI5}.  However, in Section 6.1, we first give bounds for the $L^{2}$-restriction of $\phi$ to any complete geodesic joining two cusps. The rest of the section is then devoted to proving the main result which gives a sharp lower bound for this $L^{2}$-restriction to a sufficiently long but fixed segment $\beta$ of $\delta_{1}$ (and which carries over without change to $\delta_{2}$). This gives a localization of Theorem \ref{Theorem5.1} below; the proof involves choosing a similar auxiliary function $F$ but the analysis is more intricate. In \cite{GRS} we develop the analogue of Theorem \ref{Theorem5.1} but for a non-split closed geodesic about which reflection defines a global symmetry of a compact hyperbolic surface (so that for this result  no arithmetic tools are used).
%%%%%%%%%%%%%%%%%%%%%%%%%%%%%%%%%%%%%%%%%%%%%%%%%%%%%%%%%%%%%%%%%%%%%%%%%%%%%%%%%%%%%%%%%%%%%%%%
%%%%%%%%%%%%%%%%%%%%%%%%%%%%%%%%%%%%% Section %%%%%%%%%%%%%%%%%%%%%%%%%%%%%%%%%%%%%%%%%%%%%%%%%%%%%
\subsection{Restriction to geodesics}

\begin{thm}\label{Theorem5.1} Let $\mathcal{C}$ denote any geodesic joining two cusps. Then
\[
1 \ll_{_{\mathcal{C}}} \int_{\mathcal{C}}\phi^{2}(z)\d^{\times} z \ll_{_{\mathcal{C},\ve}} t_{\phi}^{\ve}.
\]
\end{thm}
\begin{rem}
By assuming a suitable version of the quantitative QUE (see Appendix A), one can get an asymptotic formula for the integral in the Theorem  of the type $\sim \frac{12}{\pi}\log t_{\phi}$.
\end{rem}

We provide a proof of this Theorem when $\mathcal{C}$ is $\delta_{1}$ (see the remarks at the end for the general case). We will denote $t$ by $t_{\phi}$ to distinguish it from the imaginary part of the complex variable $s = \sigma + it$.

Recall that
 \begin{equation}\label{eq:e20}
 \int_{0}^{\infty} \phi(iy)y^{s-\half}\d^{\times}y = 2\rho_{\phi}(1)L(s,\phi)\Gamma\big{(}\frac{s+it_{\phi}}{2}\big{)}\Gamma\big{(}\frac{s-it_{\phi}}{2}\big{)},
 \end{equation}
 where for $\Re s > 1$ 
 \[
 L(s,\phi) = \sum_{1}^{\infty}\frac{\lambda_{\phi}(n)}{n^{s}}.
 \]
 By Parseval's formula, it then follows that
 \begin{equation}\label{eq:e21}
 \int_{0}^{\infty} |\phi(iy)|^{2}\d^{\times}y \asymp |\rho_{\phi}(1)|^{2}\int_{0}^{\infty}|L(\half + it,\phi)\Gamma\big{(}\frac{\half + i(t+t_{\phi})}{2}\big{)}\Gamma\big{(}\frac{\half+i(t-t_{\phi})}{2}\big{)}|^{2} \d t.
 \end{equation}
\\ \\
\noindent(i) \underline{The upper bound}. \\
 Using Stirling's formula, the product of gamma functions above is bounded by
 \[
 (1+|t+t_{\phi}|)^{-\frac{1}{4}}(1+|t-t_{\phi}|)^{-\frac{1}{4}}\min(e^{-\frac{\pi}{2} t},e^{-\frac{\pi}{2} t_{\phi}}).
 \]
We may then neglect the contribution from say, $t > 2t_{\phi}$ in \eqref{eq:e21} for it gives us an estimate of exponential decay for that part of the integral, after using \eqref{eq:b2} and the fact that the $L$-function has at most polynomial growth in $t$ . For the remaining, we then have
\begin{equation}\label{eq:e22}
 \int_{0}^{\infty} |\phi(iy)|^{2}\d^{\times}y \ll t_{\phi}^{-\half +\ve} \int_{0}^{2t_{\phi}}|L(\half + it,\phi)|^{2}(1+|t-t_{\phi}|)^{-\frac{1}{2}} \d t.
 \end{equation}
 If we were now to assume the Lindelof Hypothesis, that is $|L(\half + it,\phi)| \ll t_{\phi}^{\ve}$, our required upper bound for Theorem \ref{TheoremI2} would follow easily. To obtain an unconditional result, we replace $L(\half + it,\phi)$ with Dirichlet polynomials using the approximate functional equation. To this end, we use Theorem 2.5 from \cite{Ha02} which implies that for $g$ a smooth real function satisfying $g(x) + g(\frac{1}{x})=1$, we have
%pg28%
 \begin{equation}\label{eq:e23}
 |L(\half + it,\phi)| \ll \mid\sum_{n=1}^{\infty}\frac{\lambda_{\phi}(n)}{n^{\half +it}}g(\frac{n}{\sqrt{C_{\phi}}})\mid + \ (1+|t-t_{\phi}|)^{-1}C_{\phi}^{\frac{1}{4} + \ve},
 \end{equation}
 where $C_{\phi}$ is the analytic conductor of $L(s,\phi)$ given by 
 \[
 C_{\phi} = C_{\phi}(t) =
\frac{1}{4\pi^{2}}\sqrt{\big{(}\frac{1}{4} + (t+t_{\phi})^2\big{)}\big{(}\frac{1}{4} + (t-t_{\phi})^2}\big{)}.
 \]
 The contribution from the error term above to \eqref{eq:e20} is seen to be $O(t^{\ve})$. We choose $g(x)$ such that it is zero for $x\geq 2$, is decreasing smoothly for $1 \leq x < 2$ and is non-negative otherwise, so that $n \leq 2\sqrt{C_{\phi}}$ in the sum above. We break the integral in \eqref{eq:e22} into the two ranges $0\leq t\leq t_{\phi}$ and $t_{\phi}\leq t \leq 2t_{\phi}$, so that after suitable changes of variables, we finally have to estimate
\[
t_{\phi}^{-\half +\ve} \int_{1}^{t_{\phi}}\bigl\lvert\sum_{n=1}^{\infty}\frac{\lambda_{\phi}^{*}(n)}{n^{\half +it}}g(\frac{n}{\sqrt{D_{\phi}(t)}})\bigr\rvert^{2}t^{-\half}\  \d t,
\]
where $\lvert\lambda_{\phi}^{*}(n)\rvert = \lvert\lambda_{\phi}(n)\rvert$ and $D_{\phi}(t) \asymp t_{\phi}t$. Expanding the sum in the integral above, we see that the diagonal contribution is $O(t_{\phi}^{\ve})$. For the non-diagonal terms, we break the integral into $O(\log t_{\phi})$ dyadic subintervals of the type $[A,2A]$ so that we have to estimate
\[
t_{\phi}^{-\half +\ve}A^{-\half}\sideset{}{^*}\sum_{m\neq n}^{\infty}\frac{|\lambda_{\phi}(m)\lambda_{\phi}(n)|}{\sqrt{mn}} \big{|}\int_{A}^{2A} (\frac{m}{n})^{it}g(\frac{m}{\sqrt{D_{\phi}(t)}})g(\frac{n}{\sqrt{D_{\phi}(t)}})\ \d t \big{|},
\]
where $\sideset{}{^*}\sum$ indicates that the variables are restricted to satisfy $m,n \ll \sqrt{t_{\phi}A}$. Since the product of the $g$-functions in the integral are monotonic in the variable $t$, integration by parts shows that the integral is bounded by $\min(A,|\log \frac{m}{n}|^{-1})$ so that after a suitable change of variables, we are left to estimate
\[
t_{\phi}^{-\half +\ve}A^{-\half}\sideset{}{^*}\sum_{m,h \geq 1} \frac{|\lambda_{\phi}(m)|^{2}}{m}\min(A,\frac{m}{h}),
\]
so that we obtain the required bound of $O(t^{\ve})$ after applying \eqref{eq:b2} and the restriction imposed above on $m$ and $n$ (and so also on $h$).
\\ \\
\noindent (ii) \underline{The lower bound}.\\
The method here is similar to the one above but more intricate. From \eqref{eq:e21}, applying Stirling's formula and discarding a portion of the integral, we may rewrite \eqref{eq:e21} as implying
\begin{equation}\label{eq:e41}
\int_{0}^{\infty} |\phi(iy)|^{2}\ \d^{\times}y \gg \frac{1}{t_{\phi}}\int_{0}^{\infty}|\tilde{L}(\half + it,\phi)|^{2}V\big{(}\frac{t}{t_{\phi}}\big{)} \d t\ ,
\end{equation}
where $V(x)$ is a fixed, non-negative $C^{\infty}$-function supported in the interval $(\frac{1}{3},\frac{1}{4})$, with $\int V = 1$, and where we define
\[
\tilde{L}(s,\phi) = \sum_{1}^{\infty}\frac{\tilde{\rho}_{\phi}(n)}{n^{s}}.
 \]
As in \eqref{eq:e21}, we may now write the approximate functional equation in the form 
\[
\tilde{L}(\half +it,\phi) = \sum_{1}^{\infty}\frac{\tilde{\rho}_{\phi}(n)}{n^{\half +it}}g(\frac{n}{\sqrt{C_{\phi}}}) + e^{i\theta(t)}\sum_{1}^{\infty}\frac{\tilde{\rho}_{\phi}(n)}{n^{\half -it}}g(\frac{n}{\sqrt{C_{\phi}}}) + O(t_{\phi}^{-\half + \ve}),
\]
since the conductor has size $t_{\phi}^{2}$. The phase function $\theta(t)$ satisfies
\[
\theta'(t) \sim -\log\big{(}\frac{t_{\phi}+t}{2}\big{)} -\log\big{(}\frac{t_{\phi}-t}{2}\big{)}
\]
%pg29%
To estimate the lower bound for the second integral in \eqref{eq:e41}, we use Cauchy-Schwarz as follows: fix $\ve_{0}>0$ (which will be chosen small) and set
\[
F(t) = \sum_{\frac{\ve_{0}t_{\phi}}{1000}\leq \nu \leq \ve_{0}t_{\phi}}\frac{\tilde{\rho}_{\phi}(\nu)}{\nu^{\half +it}}.
\]
Consider
\begin{equation}\label{eq:e45}
\big{|}\int_{0}^{\infty}\overline{F(t)}\;\tilde{L}(\half + it,\phi)V\big{(}\frac{t}{t_{\phi}}\big{)} \ \d t\big{|}^{2} \ .
\end{equation}
This is 
\begin{align}\label{eq:e46}
&\leq \big{(}\int_{0}^{\infty}|F(t)|^{2}V\big{(}\frac{t}{t_{\phi}}\big{)}\ \d t\big{)} \big{(}\int_{0}^{\infty}|\tilde{L}(\half + it,\phi)|^{2}V\big{(}\frac{t}{t_{\phi}}\big{)}\ \d t\big{)}\cr
&\ll t_{\phi}\int_{0}^{\infty}|\tilde{L}(\half + it,\phi)|^{2}V\big{(}\frac{t}{t_{\phi}}\big{)}\ \d t\ ,
\end{align}
using standard mean-value theorems for Dirichlet polynomials and Theorem \ref{shortsum}. To evaluate the integral in \eqref{eq:e45}, we expand the sums and use the approximate functional equation (discarding the error term) to get
\begin{multline}\label{eq:e47}
\sum_{n}\sum_{\nu} \frac{\tilde{\rho}_{\phi}(n)\tilde{\rho}_{\phi}(\nu)}{\sqrt{n\nu}}\int_{0}^{\infty}\big{(}\frac{\nu}{n}\big{)}^{it}g(\frac{n}{\sqrt{C_{\phi}}})V\big{(}\frac{t}{t_{\phi}}\big{)}\ \d t + \cr
 \sum_{n}\sum_{\nu} \frac{\tilde{\rho}_{\phi}(n)\tilde{\rho}_{\phi}(\nu)}{\sqrt{n\nu}}\int_{0}^{\infty} e^{i\theta(t) + it\log(n\nu)}g(\frac{n}{\sqrt{C_{\phi}}})V\big{(}\frac{t}{t_{\phi}}\big{)}\ \d t\ .
\end{multline}
The diagonal term $n=\nu$ above gives, using Theorem \ref{shortsum}, a lower bound 
\[
\gg t_{\phi}\sum_{\frac{\ve_{0}t_{\phi}}{1000}\leq \nu \leq \ve_{0}t_{\phi}}\frac{|\tilde{\rho}_{\phi}(\nu)|^2}{\nu} \gg t_{\phi},
\]
with an absolute constant. 

For the integral in the first sum in \eqref{eq:e47}, for the non-diagonal terms, repeated integration by parts gives an estimate of the type $(t_{\phi}|\log{\frac{n}{\nu}}|)^{-l}$, for any integer $l\geq 1$. Substituting this and estimating the resulting sums using Prop. \ref{shortsum} gives an estimate of the type $O(\ve_{0}^{l}t_{\phi})$. For the rest of the sums, we observe that the oscillatory integrals have stationary points only when $n\nu \sim t_{\phi}^{2}$; however, if $\ve_{0}$ is chosen small enough, these are avoided. Moreover, the derivative of the phase is monotonic, so that the integral is bounded by $(\log t_{\phi})^{-1}$. By Cauchy-Schwarz, and Prop. \ref{shortsum}, the resulting sum in \eqref{eq:e47} is $o(t_{\phi})$. Thus the diagonal term dominates and combining the estimates from \eqref{eq:e46} and \eqref{eq:e45} gives us the lower bound in the Theorem.
\begin{rem} For the case of arbitrary geodesics with two cusps as endpoints, it suffices to consider the vertical half-lines with real part $\frac{a}{q}$, for  $q \geq 1$ and with $a$ and $q$ coprime. Following the reasoning above, $L(s,\phi)$  is replaced with $L(s,\frac{a}{q},\phi)$, which is the $L$-function twisted with $\cos(2\pi \frac{an}{q})$ and similarly for $F$, but with a modification for the range of summation. This modified L-function has a approximate functional equation (which may be derived by the methods in \cite{Ha02}) of the type
\begin{multline}
\quad \tilde{L}(\half +it,\frac{a}{q},\phi) = \sum_{1}^{\infty}\frac{\tilde{\rho}_{\phi}(n)}{n^{\half +it}}\cos(2\pi \frac{an}{q})g(\frac{n}{q\sqrt{C_{\phi}}}) \ +\cr
 e^{i(\theta(t) + 2t\log q)}\sum_{1}^{\infty}\frac{\tilde{\rho}_{\phi}(n)}{n^{\half -it}}\cos(2\pi \frac{\overline{a}n}{q})g(\frac{n}{q\sqrt{C_{\phi}}}) + O_{q}(t_{\phi}^{-\half + \ve}),
\end{multline}
where $\overline{a}$ is the multiplicative inverse of $a$ modulo $q$. The upper bound follows immediately. For the lower bound, the incooperation of terms depending on $q$ present no new difficulties as $q$ is fixed, so that the conclusions above still hold, with the exception of the nature of the diagonal term which is now of the form
\[
\gg t_{\phi}\sum_{\ve_{q}X \leq \nu \leq X} \frac{|\tilde{\rho}_{\phi}(\nu)|^2}{\nu}\cos^{2}(2\pi \frac{a\nu}{q}) \gg  \ t_{\phi},
\]
by  Corollary \ref{shortsumcorl}  with $X=\ve_{0}t_{\phi}$. 
\end{rem}

%% pg30 %%%%%%%%%%%%%%%%%%%%%%%%%%%%%%%%%%%%%%%%%%%%%%%%%%%%%%%%%%%%%%%%%%%%%%%%%%%%%%%%%%%%%%%%%%%%%%%%%%%
%%%%%%%%%%%%%%%%%%%%%%%%%%%%%%%%%%%% subsection %%%%%%%%%%%%%%%%%%%%%%%%%%%%%%%%%%%%%%%%%%%%%%%%%%%%%%

 \subsection{Restriction to compact segments.\\} \

We are now ready to prove the main result, namely the lower bound in Theorem \ref{TheoremI2} for $\beta$ a sufficiently long (but fixed) compact segment of $\delta$. The analysis is delicate and it would be more direct and simpler in terms of showing the off-diagonal terms are smaller than the diagonal, if we had a more flexible version of QUE as in \eqref{aa2} (indeed, one would then be able to take $\beta$ to be an arbitrary fixed segment on $\delta$). Since this last is not available to us, we make do with standard QUE and a cruder, more involved analysis in which we are able to show that the diagonal dominates the off-diagonal (but with both the same order of magnitude in $t_{\phi}$), if $\beta$ is long enough. 

We begin by setting up the test function to be used. Let $a$ be a  positive number independent of $t$ to be chosen ($a$ will be suitably large) and let $0<b<a$ (again to be chosen later). Let $k(y) \in C_{o}^{\infty}(\mathbb{R})$ be an even function satisfying the following properties for $y \geq 0$ :

\begin{itemize}
\item[(1)] $k(y) = 1$ for $a\leq y \leq 2a$,
\item[(2)] $k(y)=0$ for $y\in [0,a-b]\cup[2a+b,\infty)$,
\item[(3)] $0\leq k(y) \leq 1$ for $y \in (a-b,a)\cup (2a,2a+b)$ and
\item[(4)] $|k^{(l)}(y)| \leq C_{l}b^{-l}$ for some constant $C_{l}$ with any $l\geq 0$.
\end{itemize}

The support of $k$ determines the corresponding segment on $\delta_{1}$ which we denote by $\beta$. The size of $b$ is determined by the size of $a$ and any constant that depends on either will generically be denoted by $C_{a}$. Similarly,  we will let $C$ denote a generic positive constant independent of $a$ and $t$ but which will change. The parameter $l$ will be chosen suitably large later on but will be independent of $a$ and $t$. At the end of the proof, we will choose $a$ suitably and then let $t_{\phi}$ become arbitrarily large so that the dependence of $C_{a}$ on $a$ will be unimportant. In places where the value of $a$ plays an important role we will be precise and indicate this dependence. Finally, the symbol $\ll$ will mean an implied constant independent of $a$ and $t_{\phi}$.

For any real function $F(y)$, we have
\begin{equation}\label{eq:e1}
\quad I:= \int_{\beta} \phi(z)^{2}k(\Im z)\ \d^{\times} z \geq 2\int F(y)\phi(iy)k(y)\ \d^{\times} y \ - \int F(y)^{2}k(y)\ \d^{\times} y \ .
\end{equation}
By \eqref{eq:b1}, we write
%%%%%%%%%%%%%%%%%%%%%%
\begin{equation}\label{eq:e2}
\phi(iy) = 2\sqrt{y}\sum_{n\geq 1}\tilde{\rho}_{\phi}(n)e^{\frac{\pi}{2}t_{\phi}}K_{it_{\phi}}(2\pi ny).
\end{equation} 
%%%%%%%%%%%%%%%%%%%%%%
By \eqref{eq:b2},\eqref{eq:b3} and \eqref{eq:b4} we have $\tilde{\rho}_{\phi}(n) \ll t_{\phi}^{\theta +\ve}$.
%pg31%
Put $M=\frac{t_{\phi}}{\eta}$ with $\eta = \eta(a) \gg a$ to be chosen later and define 

\begin{equation}\label{eq:e3}
F(y) = 2\sqrt{y}\sum_{m\geq M}\tilde{\rho}_{\phi}(m)e^{\frac{\pi}{2}t_{\phi}}K_{it_{\phi}}(2\pi my).
\end{equation}
Then, we write in \eqref{eq:e1}, $I \geq 2I_{1} - I_{2}$. 

We provide the analysis for $j = 1$ and $2$ together, so that we write
\begin{equation}\label{eq:e4}
I_{j}=4\sum_{m\geq M}\sum_{n\geq 1}\tilde{\rho}_{\phi}(m)\tilde{\rho}_{\phi}(n)\alpha_{j}(n)G(m,n) ,
\end{equation}
where we define
\[
G(m,n)=e^{\pi t_{\phi}}\int K_{it_{\phi}}(2\pi my)K_{it_{\phi}}(2\pi ny)k(y)\ \d y \ ,
\]
and with $\alpha_{1}(n) = 1$ while $\alpha_{2}(n)$ does the same except it vanishes for $n < M$.
We decompose the sum above into three pieces: $I_{j}^{D}$ denoting the part of the sum with $m=n$ which is the diagonal, $I_{j}^{S}$ denoting the sum with small spacing $1\leq |m-n| < R$ and $I_{j}^{L}$ for the large spacings.

%%%%%%%%%%%%%%%%%%%%%%%%%%%%%%%%%%%%%%%%%%%%%%%%%%%%%%%%%%%%%%%%%%%%%%%%
%%%%%%%%%%%%%%%%%%%%%%%%%%%%%%%%%%%%%%%%% subSubsection %%%%%%%%%%%%%%%%%%%%%%%%%%%%%%%%%%
\subsubsection{The large spacings for $I_{1}$ and $I_{2}$\ .\\}  

We first deal with $I_{j}^{L}$ and prove the following 
\begin{prop}\label{Prop-e10} For any integer $l \geq 1$
\begin{multline*}\label{eq:e50}
\mathop{\sum \sum}_{\substack{m\geq M, n\geq 1 \\ |m-n| \geq R}}|\tilde{\rho}_{\phi}(m)\tilde{\rho}_{\phi}(n)G(m,n)|  \ll_{l} \frac{a}{b^{l}}R^{-l+\frac{3}{2}}\sqrt{\eta}\Big{(}\frac{1}{t_{\phi}}\sum_{n\ll \frac{t_{\phi}}{a}} |\tilde{\rho}_{\phi}(n)|^{2}\Big{)} + \cr
 \frac{a}{b}\ R^{-\half}\Big{(} \frac{1}{t_{\phi}^{\frac{3}{2}}}\sum_{n\ll \frac{t_{\phi}}{a}}\tilde{\rho}_{\phi}^{2}(n)\sqrt{n} \Big{)} + \frac{a}{b^3}R^{-\half} \Big{(}\sqrt{t_{\phi}}\sideset{}{^*}\sum_{m\geq M}\frac{|\tilde{\rho}_{\phi}(m)|^{2}}{m^{\frac{3}{2}}}\Big{)}  + o(1),\quad
\end{multline*}
as $t_{\phi}\rightarrow \infty$.
\end{prop}

The proof of this proposition relies heavily on the asymptotics of $G(m,n)$ which is determined with the following

\begin{lem}\label{Lemma-e1} \

\begin{itemize}
\item[(1)] For $u, v>0$, $r\geq 0$ and $c$ real (see \cite{GR00})
\[
\int_{0}^{\infty}K_{ir}(uy)K_{ir}(vy)\cos{(cy)} \ \d y = \frac{\pi ^2}{4\sqrt{uv}}\sec{(i\pi r)}P_{-\half +ir}\Big{(}\frac{u^2 + v^2 + c^2}{2uv}\Big{)},
\]
\item[(2)] if $w>1$ and $r \rightarrow \infty$, then (\cite{Du91})
\[
P_{-\half +ir}(w) = \big{(}\frac{\log ^{2}(w+\sqrt{w^{2}-1})}{w^{2}-1}\big{)}^{\frac{1}{4}}J_{0}(r\log(w+\sqrt{w^{2}-1}))\big{(}1+O(\frac{1}{r})\big{)},
\]
\item[(3)] if $x \rightarrow \infty$, then (\cite{GR00})
\[
J_{0}(x) = \sqrt{\frac{2}{\pi x}} \cos(x-\frac{\pi}{4})\big{(}1+O(\frac{1}{x})\big{)}.
\]
\end{itemize}

\end{lem}

Let $\hat{k}(x) = \int_{0}^{\infty}k(y)\cos (2\pi xy)\ \d y$ be the Fourier transform. Then, applying integration by parts multiple times shows that $|\hat{k}(x)| \leq C_{l}b(xb)^{-l}$ as $x\rightarrow \infty$ for any integer $l\geq 1$ so that $|\hat{k}(x)| \leq C_{l}a\min(1,(xb)^{-l})$ for all $x$. Moreover, if $k_{1}(x) = \int_{0}^{x}\hat{k}(r)\ \d r$, we have first that $|k_{1}(x)| \leq Cax$ if $x<\frac{1}{b}$ and since $k_{1}(\infty) = Ck(0) = 0$ we have also $|k_{1}(x)| = |\int_{x}^{\infty}\hat{k}(r)\ \d r| \leq C_{l}(bx)^{-l} $, with $l$  chosen suitably large. 
%pg32%
We first note that in \eqref{eq:e4} the contribution to $I_{j}^{L}$ from either $m$ or $n$ exceeding $\frac{1}{\pi a}t_{\phi}$ is exponentially small due to the decay of the Bessel functions (since $y \in [a,2a]$). Thus we may assume both $m, n \leq \frac{1}{\pi a}t_{\phi}$ in what follows (denoted by an asterisk in the summation) so that using Lemma \ref{Lemma-e1}
\begin{multline}\label{eq:e5}
I_{j}^{L}=C\underset{|m-n|\geq R}{\sideset{}{^*}\sum_{m\geq M}\sideset{}{^*}\sum_{n\geq 1}}\frac{\tilde{\rho}_{\phi}(m)\tilde{\rho}_{\phi}(n)\alpha_{j}(n)}{\sqrt{mn}}\int_{0}^{\infty} \hat{k}(x)P_{-\half +it_{\phi}}\Big{(}1+\frac{(m-n)^2 +x^2}{2mn}\Big{)}\ \d x \cr
 +O(e^{-ct_{\phi}})\ ,\qquad
\end{multline}
for some constant $c>0$. 

We put $h=|m-n|$ so that $R \leq h \ll \frac{t_{\phi}}{a}$ and in Lemma \ref{Lemma-e1} we put 
\[
w= 1+\frac{h^2 + x^2}{2mn} > 1 .
\]
 Then if $w-1 < \ve$ is small, $\log(w+\sqrt{w^{2}-1}) = \sqrt{2}\sqrt{w-1} +O(w-1) \gg \sqrt{w-1}\gg \frac{h}{\sqrt{mn}}$, while if $w-1 \geq \ve$, then $\log(w+\sqrt{w^{2}-1}) \gg_{_{\ve}} 1$. Thus since $\sqrt{mn} \ll \frac{t_{\phi}}{a}$ and $h \geq 1$, in either case we have $t\log(w+\sqrt{w^{2}-1}) \gg a$ is sufficiently large  so that we may use part (3) of Lemma \ref{Lemma-e1} in part (2)  to get
\begin{multline}\label{eq:e6}
\quad P_{-\half +it_{\phi}}(w) =C\frac{1}{\sqrt{t_{\phi}}}\Big{(}\frac{1}{w^{2}-1}\Big{)}^{\frac{1}{4}}\cos{\big{(}t_{\phi}\log(w+\sqrt{w^{2}-1})- \frac{\pi}{4}\big{)}} \cr
 \times \Big{(}1 + O\big{(}\frac{1}{t_{\phi}}\max(1,\frac{1}{\sqrt{w-1}})\big{)}\Big{)}\ .\qquad
\end{multline}
We first break the integral in \eqref{eq:e5} with $x\in [0,h]$ and $x\in (h,\infty)$.   

For $x>h$, using \eqref{eq:e6} we have
\[
P_{-\half +it_{\phi}}(w) \ll \frac{1}{\sqrt{t_{\phi}}}\big{(}\frac{x^2}{mn}\big{)}^{-\frac{1}{4}}\big{(}1+\frac{x^2}{mn}\big{)}^{-\frac{1}{4}} \ .
\]
We consider two cases:

(i) if $x \geq \sqrt{mn}$, then $P_{-\half +it_{\phi}}(w) \ll \frac{1}{x}\sqrt{\frac{mn}{t_{\phi}}}$, so that the contribution to the integral in \eqref{eq:e5} is
\begin{align}\label{eq:e7}
\int_{\sqrt{mn}}^{\infty} \hat{k}(x)P_{-\half +it_{\phi}}(w) \ \d x &\ll a\sqrt{\frac{mn}{t_{\phi}}}\int_{\sqrt{mn}}^{\infty} (xb)^{-l}\frac{1}{x}\ \d x ,\cr
&\ll a\sqrt{\frac{mn}{t_{\phi}}}(b\sqrt{mn})^{-l}\ .
\end{align}
The contribution to $I_{j}^{L}$ is then
\[
\ll \frac{a}{\sqrt{t_{\phi}}}b^{-l}\sum_{m}\sum_{n}\frac{|\tilde{\rho}_{\phi}(m)\tilde{\rho}_{\phi}(n)|}{(mn)^{\frac{l-1}{2}}} \ll_{a} t_{\phi}^{-\half + \ve},
\]
for $l$ chosen large. This contribution will be negligible and will be ignored. Next the case

(ii) if $h < x \leq \sqrt{mn}$. Now $P_{-\half +it_{\phi}}(w) \ll \frac{1}{\sqrt{t_{\phi}}}(\frac{x^2}{mn})^{-\frac{1}{4}}$ so that
\[
\int_{h}^{\sqrt{mn}} \hat{k}(x)P_{-\half +it_{\phi}}(w) \ \d x \ll \frac{a}{\sqrt{t_{\phi}}}(mn)^{\frac{1}{4}}\int_{h}^{\sqrt{mn}} (xb)^{-l}x^{-\half}\ \d x \ll \frac{a}{\sqrt{t_{\phi}}}(mn)^{\frac{1}{4}}b^{-l}h^{-l+\half}.
\]
Now, the contribution to $I_{j}^{L}$ is 
\begin{equation*}
\ll \frac{a}{b^{l}\sqrt{t_{\phi}}}\ \underset{|h|\geq R}{\sideset{}{^*}\sum_{m\geq M}\sideset{}{^*}\sum_{n\geq 1}}\frac{\tilde{\rho}_{\phi}(m)\tilde{\rho}_{\phi}(n)\alpha_{j}(n)}{(mn)^{\frac{1}{4}}}\frac{1}{h^{l-\half}}\ .
\end{equation*}
%pg33%
We break the sum into two pieces: first with $n \leq \half m$. Then, $h \geq \half m \gg_{a} t_{\phi}$ so that the contribution from the sum is negligible if $l$ is chosen suitably large. For the rest, $n > \half m \gg M$, so that the sum is 
\begin{align}\label{eq:ea1}
&\ll \frac{a}{b^{l}\sqrt{t_{\phi}}}\ \sideset{}{^*}\sum_{h\geq R}\frac{1}{h^{l-\half}}\sideset{}{^*}\sum_{n> \half M}\frac{|\tilde{\rho}_{\phi}(n)|^{2}}{\sqrt{n}}\ ,\cr
&\ll \frac{a}{b^{l}}R^{-l+\frac{3}{2}}\Big{(}\frac{1}{\sqrt{t_{\phi}M}}\sum_{n\ll \frac{t_{\phi}}{a}} |\tilde{\rho}_{\phi}(n)|^{2}\Big{)}\ .
\end{align}
For the remaining range of $x$ with $0\leq x \leq h$ we see that  $w-1 \asymp \frac{h^2}{mn}$. We first consider the contribution from the error-terms in \eqref{eq:e6}. 

If $\frac{h^2}{mn} \geq 1$ then $w-1 \gg 1$ so that $P_{-\half +it_{\phi}}(w) \ll t_{\phi}^{-\frac{3}{2}}\frac{\sqrt{mn}}{h}$. Then
\begin{align*}
\int_{0}^{h} \hat{k}(x)P_{-\half +it_{\phi}}(w) \ \d x &\ll a\ t_{\phi}^{-\frac{3}{2}}\frac{\sqrt{mn}}{h}\int_{0}^{h}\min\big{(}1,(bx)^{-l}\big{)}\ \d x \ , \cr
&\ll \frac{a}{b}\ t_{\phi}^{-\frac{3}{2}}\frac{\sqrt{mn}}{h}\ .
\end{align*}
The contribution to $I_{j}^{L}$ in \eqref{eq:e5} is
\[
\ll \frac{a}{b}t_{\phi}^{-\frac{3}{2}}\sideset{}{^*}\sum_{h\geq R}\frac{1}{h}\sideset{}{^*}\sum_{n}\tilde{\rho}_{\phi}^{2}(n) \ll_{a} t_{\phi}^{-\half + \ve},
\]
which is negligible.

If on the other hand $\frac{h^2}{mn} \leq 1$, the error term in \eqref{eq:e6} gives us $P_{-\half +it_{\phi}}(w) \ll (ht_{\phi})^{-\frac{3}{2}}(mn)^{\frac{3}{4}}$. Then, as in the case above,
\[
\int_{0}^{h} \hat{k}(x)P_{-\half +it_{\phi}}(w) \ \d x \ll \frac{a}{b}\ (ht_{\phi})^{-\frac{3}{2}}(mn)^{\frac{3}{4}} \ .
\]
The contribution to $I_{j}^{L}$ is
\begin{align}\label{eq:ea2}
&\ll \frac{a}{b}\ t_{\phi}^{-\frac{3}{2}}\sum_{h\geq R}h^{-\frac{3}{2}}\sum_{n\ll \frac{t_{\phi}}{a}}\tilde{\rho}_{\phi}^{2}(n)\sqrt{n}\ ,\cr
&\ll \frac{a}{b}\ R^{-\half}\Big{(} t_{\phi}^{-\frac{3}{2}}\sum_{n\ll \frac{t_{\phi}}{a}}\tilde{\rho}_{\phi}^{2}(n)\sqrt{n} \Big{)}\ .
\end{align}

We now consider the contribution from the main term in \eqref{eq:e6} to $I_{j}^{L}$. Denoting it by $Q(w)$ and using integration by parts, we estimate
\begin{equation}\label{eq:e8}
 \int_{0}^{h} \hat{k}(x)Q(w)\ \d x = k_{1}(h)Q(h) - \int_{0}^{h} k_{1}(x)\frac{\partial}{\partial x}Q(w)\ \d x.
\end{equation}
Since $k_{1}(h) \ll (bh)^{-l}$ and $Q(w) \ll (ht_{\phi})^{-\half}(mn)^{\frac{1}{4}}$, the contribution from these terms to $I_{j}^{L}$ is seen to be the same as in \eqref{eq:ea1}. Next, it is easily seen that
\begin{align}
\frac{\partial}{\partial x}Q(w) &\leq C\frac{1}{\sqrt{t_{\phi}}}\frac{1}{(w^2 -1)^{\frac{5}{4}}}\big{(}w + t_{\phi}\sqrt{w^2 -1}\big{)}\frac{x}{mn} \cr
&\ll\sqrt{t_{\phi}}\frac{1}{(w-1)^{\frac{3}{4}}}\frac{x}{mn} \ll \frac{\sqrt{t_{\phi}}x}{(mn)^{\frac{1}{4}}h^{\frac{3}{2}}}.
\end{align}
%pg34%
So the contribution to \eqref{eq:e8} is (following the remarks made after Lemma \ref{Lemma-e1})
\begin{align}
&\ll \frac{\sqrt{t_{\phi}}}{(mn)^{\frac{1}{4}}h^{\frac{3}{2}}}\int_{0}^{h}x|k_{1}(x)|\ \d x \cr
&\ll \frac{\sqrt{t_{\phi}}}{(mn)^{\frac{1}{4}}h^{\frac{3}{2}}}\int_{0}^{h}x\min{(ax,(xb)^{-l})}\ \d x \cr
&\ll \frac{a}{b^3}\frac{\sqrt{t_{\phi}}}{(mn)^{\frac{1}{4}}h^{\frac{3}{2}}}.
\end{align}
Thus, the contribution to \eqref{eq:e5} is then
\begin{equation}\label{eq:e9}
\ll \sqrt{t_{\phi}}\frac{a}{b^3}\sideset{}{^*}\sum_{m\geq M}\sideset{}{^*}\sum_{\substack{n\\h\geq R}}\frac{|\tilde{\rho}_{\phi}(m)\tilde{\rho}_{\phi}(n)|}{(mn)^{\frac{3}{4}}h^{\frac{3}{2}}}\ .
\end{equation}
We now estimate the sum in \eqref{eq:e9} by considering three cases: $n>m$, $n<\half m$ and $\half m \leq n \leq m$.

In the first case, the main term in \eqref{eq:e9} is
\begin{align}\label{eq:ea3}
&\ll \sqrt{t_{\phi}}\frac{a}{b^3}\sideset{}{^*}\sum_{h>R}\frac{1}{h^{\frac{3}{2}}}\ \sideset{}{^*}\sum_{m\geq M}\frac{|\tilde{\rho}_{\phi}(m)\tilde{\rho}_{\phi}(m+h)|}{\big{(}m(m+h)\big{)}^{\frac{3}{4}}},\cr
&\ll \frac{a}{b^3}R^{-\half} \Big{(}\sqrt{t_{\phi}}\sideset{}{^*}\sum_{m\geq M}\frac{|\tilde{\rho}_{\phi}(m)|^{2}}{m^{\frac{3}{2}}}\Big{)}\ .
\end{align}
For the second case, we see that $|m-n|\geq \half m$ so that in \eqref{eq:e9} we find the main term is, by Lemma \ref{LemmaTwo}
\[
\leq C_{a}\sqrt{t_{\phi}}\sideset{}{^*}\sum_{\substack{m\geq M\\n\geq 1}}\frac{|\tilde{\rho}_{\phi}(m)\tilde{\rho}_{\phi}(n)|}{m^{\frac{9}{4}}n^{\frac{3}{4}}} \ll_{a,\ve} \sqrt{t_{\phi}}M^{-\frac{5}{4}}t^{\frac{1}{4}+\ve} \ll_{a,\ve} t_{\phi}^{-\half +\ve},
\]
which is negligible. Finally, for the last case, we have $ \half m\leq n \leq m-R$ so that both $m$ and $n$ exceed $CM$ so that we may reverse their roles and get the estimate obtained in the first case, that is \eqref{eq:ea3}. 

Collecting the estimates from \eqref{eq:e5}, \eqref{eq:ea1}, \eqref{eq:ea2} and \eqref{eq:ea3}, we see that
\begin{multline}\label{eq:e10}
|I_{j}^{L}| \ll_{l} \frac{a}{b^{l}}R^{-l+\frac{3}{2}}\sqrt{\eta}\Big{(}\frac{1}{t_{\phi}}\sum_{n\ll \frac{t_{\phi}}{a}} |\tilde{\rho}_{\phi}(n)|^{2}\Big{)} + \frac{a}{b}\ R^{-\half}\Big{(} t_{\phi}^{-\frac{3}{2}}\sum_{n\ll \frac{t_{\phi}}{a}}\tilde{\rho}_{\phi}^{2}(n)\sqrt{n} \Big{)} + \cr
 \frac{a}{b^3}R^{-\half} \Big{(}\sqrt{t_{\phi}}\sideset{}{^*}\sum_{m\geq M}\frac{|\tilde{\rho}_{\phi}(m)|^{2}}{m^{\frac{3}{2}}}\Big{)}  + o(1),\quad
\end{multline}
as $t\rightarrow \infty$, for $j=1$ and $2$.
%%%%%%%%%%%%%%%%%%%%%%%%%%%%%%%%subSubsection%%%%%%%%%%%%%%%%%%%%%%%%%%%%%%%%%%%%%%%%%%%%%%%%%%%%%%%
\vspace{12pt}
\subsubsection{Small spacings for $I_{1}$ and $I_{2}$.\\} 
We first remark that in our application at hand, we will not need to deal with small spacings since we are able to choose $R=1$. This is due to the fact that we choose our segment $\beta$ to be sufficiently long. However, we include the analysis here to indicate that QUE can be used to control such sums. It suffices to deal here with $I_{1}^{S}$. This is because when $R$ is bounded, $I_{1}^{S} - I_{2}^{S}$ is a sum where $m$ and $n$ both run through a bounded  interval (depending on $R$) so that the contribution from them is negligible as $t_{\phi} \rightarrow \infty$ and if $\frac{a}{\eta} \rightarrow 0$ as $a$ grows (we use here the bounds \eqref{eq:b4} for the Fourier coefficients and the estimates for the Bessel function away from the $transitional \; range$ in Corollary \ref{CorlOne}). 
%pg35%
\begin{prop}\label{Prop-e11}  For bounded $R$,
\[
I_{1}^{S}= \mathop{\sum \sum}_{\substack{m\geq M, n\geq 1 \\ 1\leq |m-n| < R}}\tilde{\rho}_{\phi}(m)\tilde{\rho}_{\phi}(n)G(m,n)  \ll Ra\Big{(}\frac{1}{t_{\phi}}\sum_{m\leq M} \tilde{\rho}_{\phi}^{2}(m)\Big{)} + o(1)\ .
\]
\end{prop}

For the proof, we use the notation from the subsection above and rewrite 
\begin{multline}\label{eq:e11a}
\quad I_{1}^{S}= \sum_{\substack{m, m+h \neq 0\\1\leq |h|\leq R}}\rho_{\phi}(m)\rho_{\phi}(m+ h)\int K_{it}(2\pi |m|y)K_{it_{\phi}}(2\pi |m+ h|y)k(y)\ \d y \cr
 - 2\sum_{h=1}^{R}\sum_{m=1}^{M}\tilde{\rho}_{\phi}(m)\tilde{\rho}_{\phi}(m\pm h)\int e^{\pi t_{\phi}}K_{it_{\phi}}(2\pi my)K_{it_{\phi}}(2\pi (m\pm h)y)k(y)\ \d y \ .\quad
\end{multline}
We denote the above as $I_{1}^{S}=I_{1}^{S_{1}}-I_{1}^{S_{2}}$. We will show that $I_{1}^{S_{1}}$ is small using  QUE  directly, while $I_{1}^{S_{2}}$ will be bounded in a manner similar to $I_{1}^{L}$ in the previous subsection.
 
We will first consider $I_{2}^{S_{1}}$. We assume $\eta > 100a$ so that $2\pi my < \half t_{\phi}$, and since $h$ is relatively small, the integral in \eqref{eq:e10} is $O(at_{\phi}^{-1})$ by Lemma \ref{LemmaOne}. Thus
\begin{equation}\label{eq:e19}
 I_{2}^{S_{1}} \ll Ra\Big{(}\frac{1}{t_{\phi}}\sum_{m\leq M} \tilde{\rho}_{\phi}^{2}(m)\Big{)}\ . 
\end{equation}
\begin{rem} It may appear that the estimate above for the integral is rather crude. A careful analysis shows that there is not sufficient oscillation in the integral to extract futher savings since $mh \sim t_{\phi}$.
\end{rem}

We next consider $I_{1}^{S_{1}}$. Let $g(y)\in C_{0}^{\infty}(\mathbb{R}^{+})$ with the support contained in the interval $(A^{-1},\infty)$ with $A\geq 1$. For any $h\in\mathbb{Z}$ and $z\in\mathbb{H}$ we have the incomplete Poincare series
\[
P_{h,g}(z)=\sum_{\gamma\in\Gamma_{\infty}\backslash\Gamma}g\big{(}y(\gamma z)\big{)}e\big{(}hx(\gamma z)\big{)}.
\]
The sum is finite and $P_{h,g}(z)$ is $\Gamma$-invariant. Moving $z$ into the standard fundamental domain we see that $|P_{h,g}(z)| \leq A\parallel g\parallel_{\infty}$ uniformly in $z$. Since $\int_{\mathbb{X}}P_{h,g}(z)\ \d {\rm A}(z) =0$ for $h\neq 0$, it follows from QUE that there is a function $\ve(t_{\phi}) \rightarrow 0$ such that 
\begin{equation}\label{eq:e14}
\int_{\mathbb{X}}\phi(z)^{2}P_{h,g}(z)\ \d {\rm A}(z) \ll_{h,g} \ve(t_{\phi})\ .
\end{equation}
Unfolding the integral above and using the fact that $\phi$ is even, we see that the integral in \eqref{eq:e14} is equal to (up to to a multiplicative constant) the sum over $m$ in $I_{1}^{S_{1}}$ with $g(y) = yk(y)$, so that 
\begin{equation}\label{eq:e15}
|I_{1}^{S_{1}}|
 \leq C_{_{a,R}} \ \ve(t_{\phi})\ ,
\end{equation}
which is negligible if $R$ is bounded and $t \rightarrow \infty$.
Hence from \eqref{eq:e11a}, \eqref{eq:e15} and  \eqref{eq:e19}, we have 
\begin{equation}\label{eq:ea4}
I_{1}^{S} \ll Ra\Big{(}\frac{1}{t_{\phi}}\sum_{m\leq M} \tilde{\rho}_{\phi}^{2}(m)\Big{)} + o(1)\ .
\end{equation}

\begin{rem} It is worth pointing out that a stronger version of Prop. \ref{Prop-e11} holds with $R \rightarrow \infty$ with $t$ if we assume a quantitative form of QUE. To obtain this we have to estimate the Sobolev norm of $P_{h,g}$ uniformly in $h$. Since
\[
\Delta\big{(}g(y)e(hx)\big{)} = Lg(y)e(hx)
\]
where
\[
Lg(y)=y^{2}g''(y) + 4\pi^{2}h^{2}g(y),
\]
it follows that for $m\geq 0$, $\Delta^{m}P_{h,g}(z)= P_{h,L^{m}g}(z)$. Then for $g$ fixed and $h\neq 0$
\begin{equation}\label{eq:e12}
\parallel \Delta^{m}P_{h,g}\parallel_{\infty} \ll_{g,m} |h|^{2m}.
\end{equation}
Hence, for $k$ fixed, the Sobolev norm satisfies (for $h\neq 0$)
\begin{equation}\label{eq:e13}
\parallel P_{h,g}\parallel_{W^{k}}^{2} \ll_{g,k} |h|^{k}.
\end{equation}

Then, QQUE (see Appendix A1) implies that there are constants $k$ and $\nu >0$ such that equation \eqref{eq:e14}  takes the form
\[
\int_{\mathbb{X}}\phi(z)^{2}P_{h,g}(z)\ \d {\rm A}(z) \ll_{g} |h|^{k}t_{\phi}^{-\nu +\ve}.
\]
so that
\[
I_{1}^{S} \ll Ra\Big{(}\frac{1}{t_{\phi}}\sum_{m\leq M} \tilde{\rho}_{\phi}^{2}(m)\Big{)} + O_{a}(R^{(k+1)}t_{\phi}^{-\nu +\ve})\ .
\]
The error term here is now $o(1)$ for the larger range $R \ll t_{\phi}^{\frac{\nu}{k+1} -\ve}$.
\end{rem}

%%%%%%%%%%%%%%%%%%%%%%%%%%%%%%%%%%%%%%%%%%%%%%%%%%%%%%%%%%%%%%%%%%%%%%%%%%%%%%%%%%%%%%%%%%%%%%%%%%%%
%%%%%%%%%%%%%%%%%%%%%%%%%%%%%%%%%%%%%%% subsubsection %%%%%%%%%%%%%%%%%%%%%%%%%%%%%%%%%%%%%%%%%%%%
\subsubsection{Completion of the proof.\\}
Noting that $I_{1}^{D} = I_{2}^{D}$, from \eqref{eq:e4}, we rewrite the diagonal term as
\begin{equation}\label{eq:ea5}
I_{1}^{D} =2\sum_{m\neq 0}\rho_{\phi}(m)^{2}\int K_{it_{\phi}}(2\pi |m|y)^{2}k(y)\ \d y + O\Big{(}\frac{a}{t_{\phi}}\sum_{m\leq M} \tilde{\rho}_{\phi}^{2}(m)\Big{)},
\end{equation}
where the error term comes from $m\leq M$. The main term here is 
\[
2\int_{\mathbb{X}}\phi(z)^{2}P_{0,g}(z)\ \d {\rm A}(z),
\]
where $g(y)=yk(y)$. By  QUE this is
\[
\sim 2\int_{\mathbb{X}}P_{0,g}(z)\ \d {\rm A}(z) \sim 2\int k(y)\ \d^{\times} y \geq 2\log 2.
\]
Collecting all the estimates from Prop. \ref{Prop-e10},  Prop. \ref{Prop-e11} and the remarks made  above, we conclude that $|I| \geq I_{1}^{D} + I_{1}^{S} +2I_{1}^{L} - I_{1}^{L} + o(1)$ so that there are positive absolute constants $C_{j}$ such that
\begin{multline}\label{eq:ea40}
|I| \geq 2\log 2 - C_{1}\frac{a}{b^{l}}R^{-l+\frac{3}{2}}\sqrt{\eta}\Big{(}\frac{1}{t_{\phi}}\sum_{n\ll \frac{t_{\phi}}{a}} \tilde{\rho}_{\phi}^{2}(n)\Big{)} -
 C_{2}\frac{a}{b}\ R^{-\half}\Big{(} t_{\phi}^{-\frac{3}{2}}\sum_{n\ll \frac{t_{\phi}}{a}}\tilde{\rho}_{\phi}^{2}(n)\sqrt{n} \Big{)} - \cr
 C_{3}\frac{a}{b^3}R^{-\half} \Big{(}\sqrt{t_{\phi}}\sideset{}{^*}\sum_{m\geq M}\frac{\tilde{\rho}_{\phi}^{2}(m)|}{m^{\frac{3}{2}}}\Big{)} - C_{4}Ra\Big{(}\frac{1}{t_{\phi}}\sum_{m\leq M} \tilde{\rho}_{\phi}^{2}(m)\Big{)}  + o(1)\ .\qquad 
\end{multline}
To estimate these sums, we appeal to Prop. \ref{shortsum}, which  by partial summation gives us
\begin{align*}
|I| &\geq 2\log 2 - C_{1}\frac{1}{b^{l}}R^{-l+\frac{3}{2}}\sqrt{\eta} -
 C_{2}\frac{a}{b}\ R^{-\half}a^{-\frac{3}{2}} - C_{3}\frac{a}{b^3}R^{-\half}\sqrt{\eta}  - C_{4}Ra\eta^{-1}  + o(1)\ ,\cr
&\geq \log 2\ ,
\end{align*}
on choosing $b=\sqrt{\eta}$, $R=1$, $l=100$ and $\eta = a^{\frac{101}{100}}$, say and letting $a$ be sufficient large.  This completes the proof of the lower bound.

%%%%%%%%%%%%%%%%%%%%%%%%%%%%%%%%%%%%%%%%%%%%%%%%%%%%%%%%%%%%%%%%%%%%%%%%%%%%%%%%%%%%%%%%%%%%%%%%
%%%%%%%%%%%%%%%%%%%%%%%%%%%%%%%%%%%%% Section %%%%%%%%%%%%%%%%%%%%%%%%%%%%%%%%%%%%%%%%%%%%%%%%%%%%%
\subsection{Proof of Theorem  \ref{TheoremI5}.\\} The method here is exactly the same as that used in the proof of Theorem \ref{TheoremTwo}, where one dealt with  closed horocycles. For this we obtain an upper bound for the analogue of the integral appearing in \eqref{eq:d8} so that we have to prove 
\begin{equation}\label{eq:e24}
\int_{\alpha_{j}}^{\alpha_{j+1}}\phi(iy)\ \d^{\times} y \ll t_{\phi}^{-\half + \ve}.
\end{equation}
Using the discussion in Sec. 6.1, it is easy to see that Parseval's theorem applied to the product of $\phi(iy)$ and the characteristic function for the subinterval $[\alpha_{j},\alpha_{j+1}]$ shows that the integral in \eqref{eq:e24} is 
\begin{equation}\label{600}
2\rho_{\phi}(1)\int_{-\infty}^{\infty}L(\half + it,\phi)\Gamma\big{(}\frac{\half + i(t+t_{\phi})}{2}\big{)}\Gamma\big{(}\frac{\half+i(t-t_{\phi})}{2}\big{)}\big{(}\frac{\alpha_{j+1}^{it} - \alpha_{j}^{it}}{t}\big{)}\ \d t.
\end{equation}
In exactly the same way as was done in Sec. 6.1 leading to \eqref{eq:e22}, it then follows that the integral in \eqref{eq:e24} is 
\[
\ll t^{-\frac{1}{4}+\ve}\int_{0}^{2t_{\phi}}|L(\half + it,\phi)|(1+|t-t_{\phi}|)^{-\frac{1}{4}}\min(1,\frac{1}{t})\ \d t + \ e^{-ct_{\phi}},
\]
for some constant $c>0$. Then the Lindelof Hypothesis for $L(\half + it,\phi)$ gives us the estimate in \eqref{eq:e24}.

%%%%%%%%%%%%%%%%%%%%%%%%%%%%%%%%%%%% Appendices %%%%%%%%%%%%%%%%%%%%%%%%%%%%%%%%%%%%%%%%%%%%%%%%%%%%

\appendix
\section*{Appendix}

\section{Formulations of QUE and applications to short closed horocycles}

The usual formulation of QUE on the configuration space $\mathbb{X}$ was given in \eqref{i101}. In our analysis, we have used the Fourier expansion of $\phi$ repeatedly. We examine here the formulation of QUE in terms of the Fourier coefficients. For holomorphic cusp forms of even weight $k$ ($L^{2}$-normalized with respect to the Petersson inner-product) it is shown in \cite{LS03} that QUE is equivalent to:

For $m\in \mathbb{Z}$ and $h\in C^{\infty}_{0}(0,\infty)$ fixed 
\begin{equation}\label{aa1}
\frac{\Gamma(k)}{k}\sum_{n=1}^{\infty}\rho_{f}(n)\overline{\rho_{f}(n+m)}h\big{(}\frac{4\pi n}{k}\big{)} \rightarrow \frac{3}{\pi}\delta_{m,0}\int_{0}^{\infty}h(y)\ \d y \ ;
\end{equation}
as $k\rightarrow \infty,  $ with $\rho_{f}(n) = a_{f}(n)(4\pi n)^{\half(1-k)}$ where $a_{f}$ are the fourier coefficients of $f$.

\noindent Similarly, we expect that for Maass forms we have:

For $m\in \mathbb{Z}$ and $h\in C^{\infty}_{0}(0,\infty)$ fixed
\begin{equation}\label{aa2}
 \frac{e^{-\pi t_{\phi}}}{t_{\phi}}\sum_{n\neq 0}\rho_{\phi}(n)\overline{\rho_{\phi}(n+m)}h\big{(}\frac{\pi |n|}{t_{\phi}}\big{)} \rightarrow \frac{12}{\pi^3}\delta_{m,0}\int_{0}^{\infty}h(y)\ \d y \ ;
\end{equation}
as $t_{\phi}\rightarrow \infty \ $.

While \eqref{eq:ea9} and \eqref{eq:e14} show that \eqref{aa2} follows from QUE for certain $h$'s, there is a basic difficulty in realizing all $h$'s of compact support in this way. The problem comes from the nonlocal transform $W_{g}$ of $g$ in \eqref{ea8b} (and it is more complicated when $m\neq 0$), which is a feature of Maass forms and is not present in the holomorphic case. As a consequence we have that \eqref{aa2} implies QUE (which is straightforward) but we don't know the converse implication. The flexible form \eqref{aa2}, in as much as it allows one to truncate the $n$-sum into a compact dyadic piece would allow for a much more direct treatment treatment of Section 6.2 . In fact, it allows one to establish Theorem \ref{TheoremI2} for $\beta$ a fixed but arbitrary small segment in $\delta$. This in turn would yield Theorem \ref{TheoremI6} for such a $\beta$.

Both forms of QUE above can be quantified with rates in obvious ways, and again the quantitative version of \eqref{aa2} will imply the quantitative QUE which we call QQUE:
%Conjecture 1.2
\begin{Conj}\label{conjecture2}{\bf QQUE}\\
There is a $\nu > 0$ and a $k < \infty$ such that for any $f\in C^{\infty}(\mathbb{X})$ with $\int_{\mathbb{X}}f\ \d A =0$ we have
\[
\int_{\mathbb{X}}|\phi(z)|^{2}f(z)\ \d A \ll t_{\phi}^{-\nu} \parallel f\parallel_{W^{k}} ,
\]
where $W^{k}$ is the Sobolev $k$-norm on $\mathbb{X}$.	
\end{Conj}

\begin{rem}
While QUE is known thanks to Lindenstrauss \cite{Lin06} and Soundararajan \cite{So10}, it is not known with any rate. There is little doubt (in our minds) about the truth of Conjecture \ref{conjecture2} as it follows from a subconvex estimate for the triple product $L$-function $L(s,\phi \times \phi \times \psi )$, where $\psi$ is another Maass form (fixed ie independent of $\phi$) or a unitary Eisenstein series.\footnote{If $\phi$ is itself a unitary Eisenstein series or a dihedral cusp form, this subconvexity is known \cite{Sa01}, \cite{LS95}.} This implication is a consequence of Watson's explicit formula \cite{Wa02}. The subconvex estimate is in turn a weak form of the sharp estimate for critical values of such $L$-functions, known as the Lindelof Hypothesis. It is well known that the latter is a consequence of the Riemann Hypothesis for $L(s,\phi \times \phi \times \psi )$ so that Conjecture \ref{conjecture2} has a very firm basis.
\end{rem}

For our application, we need a more explicit version of QQUE stated as follows: let 
\begin{equation}\label{eq:app-a1}
 \mathcal{F}_{Y} = \{ z \in \mathbb{X}: \Im (z) \geq Y\}.
 \end{equation}

\begin{Conj}\label{QQUE2}
There is a number $\nu > 0$ such that for $Y>0$, as $t_{\phi} \rightarrow \infty$
\[
\int_{ \mathcal{F}_{Y}} |\phi(z)|^{2} \frac{\ \d x\ \d y}{y^{2}} = \frac{3}{\pi}Vol(\mathcal{F}_{Y}) + O(Y^{-\half}t_{\phi}^{-\nu +\ve}).
\]
Consequently,
\[
\int_{0}^{1}\int_{Y}^{2Y}|\phi(x+iy)|^{2}\frac{\ \d x\ \d y}{y^{2}} \gg \frac{1}{Y}
\]
for $0 < Y < t_{\phi}^{2\nu - \ve}$.
\end{Conj}

\begin{rem} The Lindelof Hypothesis for triple product $L$-functions would allow the exponent $\nu = \half$, in which case  the second lower bound would be valid for almost the full range of $Y$. This conjecture is a more precise version of the "escape of mass" Proposition 2 of \cite{So10}.
\end{rem}

We now apply these quantitative versions to the analysis of the zeros of $\phi$ for short closed horocycles nearer the cusp. In the proof of Theorem \ref{TheoremOne} the upper bound was valid for closed horocycles $\mathcal{C}_{Y}$ with $0<Y\ll t_{\phi}^{1 -\ve}$, while the lower bound, which uses QUE was then restricted to bounded values of $Y$. This restriction to  bounded $Y$ carries over to Theorem \ref{TheoremTwo}. To extend these theorems to closed horocycles located higher up in the fundamental domain,  we need the analog of Prop. \ref{shortsum} for shorter sums. For such sums, we appeal Conjecture \ref{QQUE2}. 

\begin{prop}\label{LemmaAThree} Assume Conjecture \ref{QQUE2}. Then  for any $ 0<c < Y \ll t_{\phi}^{2\nu -\ve}$
\[
\sum_{|n|\leq  \frac{t_{\phi}}{Y}}|\tilde{\rho}_{\phi}(n)|^{2} \gg_{_{\ve}} \frac{t_{\phi}}{Y}.
\]
\end{prop}
\proof The proof is identical to that in Prop. \ref{shortsum} with $\alpha=\frac{Y}{2\pi}$ and replacing \eqref{eq:d40} with the Conjecture.

Using Prop. \ref{LemmaAThree} we have the following:  let $T_{\phi} = \min(t_{\phi}^{2\nu},t_{\phi}^{\half})$. Then
%%%%%%%%%%%%%%%%%%%%%%%%%%%%%%%%%%%%%%%%%%%%%%%% Theorem A1 %%%%%%%%%%%%%%%%%%%%%%%%%%%%%%%%%%%%%%%%%%%%
%%%%%%%%%%%%%%%%%%%%%%%%%%%%%%%%%%%%%%%%%%%%%%%%%%%%%%%%%%%%%%%%%%%%%%%%%%%%%%%%%%%%%%%%%%%%%%%%%%%%%%
\begin{thm}\label{TheoremAOne} Assume Conjecture \ref{QQUE2}. For any $\ve > 0$  and $0<Y\ll T_{\phi}^{1-\ve}$
\[
 \lVert\phi|_{_{\mathcal{C}_{Y}}}\rVert_{_{2}}^{2}=\int_{0}^{1}|\phi(x+iY)|^{2} \ \d x \gg_{\ve} t_{\phi}^{-\ve}.
\]
\end{thm}

We then have

%%%%%%%%%%%%%%%%%%%%%%%%%%%%%%%%%%%%%Theorem A2 %%%%%%%%%%%%%%%%%%%%%%%%%%%%%%%%%%%%%%%%%%%%%%%%%%%%%%%%
%%%%%%%%%%%%%%%%%%%%%%%%%%%%%%%%%%%%%%%%%%%%%%%%%%%%%%%%%%%%%%%%%%%%%%%%%%%%%%%%%%%%%%%%%%%%%%%%%%%%%%%
\begin{thm}\label{TheoremATwo} Assume Conjecture \ref{QQUE2}. For $\ve > 0$ and $ 0<c < Y \ll T_{\phi}^{1-\ve}$
\[
t_{\phi}^{\frac{1}{12}-\ve} \ll_{\ve} \#\{z\in \mathcal{C}_{Y}: \phi(z)=0\} \ll \frac{t_{\phi}}{Y}\log t_{\phi}.
\]
\end{thm}
\begin{rem} Theorem \ref{TheoremThree} suggests that the upper bound here may be close to best possible.
\end{rem}

The proof of Theorem \ref{TheoremAOne} follows directly from the proof of Theorem \ref{TheoremOne} since \eqref{eq:d7} implies our result on using the analog of Prop. \ref{shortsum} given in Prop. \ref{LemmaAThree}.\\
%%%%%%%%%%%%%%%%%%%%%%%%%%%%%%%%%%%%%%%%%%% Lemma A3 %%%%%%%%%%%%%%%%%%%%%%%%%%%%%%%%%%%%%%%%%%%%%%%
%%%%%%%%%%%%%%%%%%%%%%%%%%%%%%%%%%%%%%%%%%%%%%%%%%%%%%%%%%%%%%%%%%%%%%%%%%%%%%%%%%%%%%%%%%%%%%%%%%%

%%%%%%%%%%%%%%%%%%%%%%%%%%%%%%%%%%%%%%%%%%%%%%%%%%%%%%%%%%%%%%%%%%%%%%%%%%%%%%%%%%%%%%%%%%%%%%%%%%%%%%%%%
{\noindent \bf Proof of Theorem \ref{TheoremATwo}.}\  It is clear that the above conditional results provide a proof of the lower bound using the analysis following \eqref{eq:d8}. For the upper bound, we give a direct proof (suggested by the complexification idea in \cite{TZ09}) as follows: for any $Y>0$, define for $\zeta = x+i\gamma$
\[
\phi_{Y}(\zeta) = \sum_{n \neq 0} \rho_{\phi}(n) Y^{\frac{1}{2}}K_{it_{\phi}}(2\pi |n|Y)e(n\zeta).
\]
The estimates for the Fourier coefficients from \eqref{eq:b2}, \eqref{eq:b3}, \eqref{eq:b4} and the Bessel function show that $\phi_{Y}(\zeta)$ is a holomorphic function in the horizontal strip given by $|\Im(\zeta)|\leq \frac{Y}{4t_{\phi}}$. Moreover, it has at most a polynomial rate of growth in $t_{\phi}$ in such a strip. The lower bound in Theorem \ref{TheoremAOne} shows that $\phi_{Y}(\zeta)$ is not identically zero. Then, by Jensen's lemma, it follows that $\phi_{Y}(\zeta)$ has at most $O(\log t_{\phi})$ zeros in a disc of radius at most $\frac{Y}{8t_{\phi}}$ centered at a point $x$ real. Hence, $\phi_{Y}(\zeta)$ has at most $O(\frac{t_{\phi}}{Y}\log t_{\phi})$ real zeros in the interval $0\leq x \leq 1$ which then implies the upper bound for the zeros in the Theorem.

%%%%%%%%%%%%%%%%%%%%%%%%%%%%%%%%%%%%%%%%%%%%%%%%%%%%%%%%%%%%%%%%%%%%%%%%%%%%%%%%%%%%%%%%%%%%%%%%%%%%%%%%%%%%
%%%%%%%%%%%%%%%%%%%%%%%%%%  Section %%%%%%%%%%%%%%%%%%%%%%%%%%%%%%%%%%%%%%%%%%%%%%%%%%%%%%%%%%%%%%%%%%%%%%%
\section{Random wave model on closed horocycles} In \cite{GS12} we used the random wave model in Edelman-Kostlan \cite{EK95} to obtain predictions on the distribution of zeros of the real part of holomorphic cusp forms on vertical lines. We apply the same methods here to even cusp forms but on horizontal line segments. From (7) we write
\begin{equation}\label{eq:app-b1}
\tilde{\Phi}(z) =  \sum_{n\geq 1}\lambda_{\phi}(n)e^{\frac{\pi}{2}t_{\phi}}K_{it_{\phi}}(2\pi ny)\cos(2\pi nx).
\end{equation}
For any fixed $Y$, which may vary with $t_{\phi}$, we consider the vectors
\[
\mathbf{v}=\mathbf{v}(x)= \sum_{n\geq 1}e^{\frac{\pi}{2}t_{\phi}}K_{it_{\phi}}(2\pi nY)\cos(2\pi nx)\mathbf{e}_{n}
\]
where $\mathbf{e}_{n} = (...,0,0,1,0,0,...)$ denotes the unit vector with the $1$ in the nth- \\coordinate. The probability density function for the real zeros of the the associated random wave function is given by
\begin{equation}\label{eq:app-b2}
 \mathcal{P}(\mathbf{v}) = \mathcal{P}(\mathbf{v},x) = \frac{1}{\pi} \sqrt{\frac{<\mathbf{v},\mathbf{v}><\mathbf{v'},\mathbf{v'}> - <\mathbf{v},\mathbf{v'}>^{2}}{<\mathbf{v},\mathbf{v}>^{2}}}
\end{equation}
where $\mathbf{v'}$ is the derivative with respect to $x$ and with the standard inner-product. This is the expected number of real zeros of the associated wave function per unit length at the point $x$. Unlike the situation in \cite{GS12}, we do not expect the term $<\mathbf{v},\mathbf{v'}>$ to contribute to the main term, in which case, we should have 
\[
\mathcal{P}(\mathbf{v}) \sim \frac{1}{\pi} \sqrt{\frac{<\mathbf{v'},\mathbf{v'}> }{<\mathbf{v},\mathbf{v}>}}.
\]
Define, for $r>0$ 
\begin{equation}\label{eq:app-b3}
S(x)= S(x,r) = \sum_{n\geq 1}e^{\pi r}K_{ir}^{2}(2\pi nY)\cos(2\pi nx),
\end{equation}
so that with $r=t_{\phi}$ we have
\begin{flalign*}
<\mathbf{v},\mathbf{v}> &= \frac{1}{2}(S(0) + S(2x)),\\
<\mathbf{v},\mathbf{v'}> &= -\frac{1}{2}S'(2x) 
\end{flalign*}
\text{and} 
\[
\quad<\mathbf{v'},\mathbf{v'}> = \frac{1}{2}(S''(0) - S''(2x)).
\]
For $x$ not close to a half-integer, we would not expect the terms above involving $x$ to contribute due to cancellation. Using Lemma \ref{LemmaOne}, we write the sum $S(x)= S_{1}(x)+ S_{2}(x)+ S_{3}(x)$ for $n$'s restricted by the ranges as indicated there. Using Corollary \ref{CorlOne}, we see that $S_{2}(x)$ is negligible. Similarly, $S_{3}(x)$ is bounded by $(1 + \Delta Y^{-1})t_{\phi}^{-\frac{2}{3}} = o(Y^{-1})$ if $Y = o(t_{\phi}^{\frac{2}{3}})$. In what follows, we will restrict $Y$ to satisfy $Y = o(t^{\frac{1}{3}})$, so that $\frac{\Delta}{Y} \rightarrow \infty$. Put $T= T_{\phi} = \frac{t_{\phi}}{2\pi Y}$ and $E= \frac{\Delta}{2\pi Y}$ so that
\begin{align}\label{eq:app-b4}
 YS_{1}(x) &\sim \sum_{n=1}^{T-E} \frac{1}{\sqrt{T^{2} - n^{2}}} \sin ^{2} \left(\frac{\pi}{4} +t_{\phi} H(\frac{n}{T})\right)\cos(2\pi nx) \cr
&\sim \frac{1}{2} \sum_{n=1}^{T-E} \frac{1}{\sqrt{T^{2} - n^{2}}} \cos(2\pi nx) + O\left(|\sum_{n=1}^{T-E} \frac{1}{\sqrt{T^{2} - n^{2}}} \sin(f(n))|\right),
\end{align}

\noindent where we denote by $f(n)$ any function of the type $2\pi nx \pm 2t_{\phi}H(\frac{n}{T}) + C$, with $C$ a constant. We show that the second sum in \eqref{eq:app-b2} above is $o(1)$ for all $x$, while the same is true for the first sum except for those $x$ near integers. For $x=0$, the first sum  is  asymptotic to $\frac{\pi}{2}$ and it is easy to show it is $o(1)$ if $\lVert x\rVert \gg Yt_{\phi}^{-\frac{2}{3}}$. To deal with the second sum, we use the following lemma from \cite{Vi76} (pg. 306):

%%%%%%%%%%%%%%%%%%%%%%%%%%%%%%%%%%%%%%%%%%%%%%%%%%%%%%%%%%%%%%%%%%%%%%%%%%%%%%%%%%%%%%%%%%%%%%%%%%%%%%%%
%%%%%%%%%%%%%%%%%%%%%%%%%%%%%%%%%%%% Lemma 4 %%%%%%%%%%%%%%%%%%%%%%%%%%%%%%%%%%%%%%%%%%%%%%%%%%%%%%%%%%
\begin{lem}\label{LemmaFour} For $q < x \leq r$, suppose $f''(x)$ is continuous and satisfies $ A^{-1} \ll |f''(x)| \ll A^{-1}$ for $A$ suitably large. Futher, suppose $g(x) \ll G$ is  monotonic. Then
\[
\sum_{q<n\leq r} g(n)e(f(n)) \ll G\left(\frac{r-q}{\sqrt{A}} + \sqrt{A} + \log (\rm{max}(r-q,A))\right).
\]
\end{lem}

For those $n \leq T^{1-\delta}$, with $\delta > 0$ sufficiently small, the sum is trivially $o(1)$. For $T^{1-\delta} \leq n \leq T - E$, we subdivide the sum into subsums with $T- 2^{j+1}E < n \leq  T- 2^{j}E$ with $0\leq j \leq J-1$. For $n$ in such a subsum
\begin{equation}\label{eq:app-b5}
f''(n) \asymp  \frac{Tt_{\phi}}{n^{2}\sqrt{T^{2} - n^{2}}} \asymp \frac{\sqrt{T}t_{\phi}}{(T-2^{j}E)^{2}\sqrt{2^{j}E}} := \frac{1}{A}\ .
\end{equation}
Moreover, we may take $G = (2^{j}ET)^{-\frac{1}{2}}$ so that the lemma gives an estimate for each subsum of
\[
\ll (2^{j}ET)^{-\frac{1}{2}}\left( \frac{2^{j}E T^{\frac{1}{4}}t_{\phi}^{\frac{1}{2}}}{(T-2^{j}E)2^{\frac{j}{4}}E^{\frac{1}{4}}} + \frac{(T-2^{j}E)2^{\frac{j}{4}}E^{\frac{1}{4}}}{T^{\frac{1}{4}}t_{\phi}^{\frac{1}{2}}} + \log t_{\phi}\right).
\]
Simplifying and summing over $j$ then gives us the estimate of $o(1)$ if $Y = O(t_{\phi}^{\frac{5}{9}})$, as is the case. We thus have
\[
S(x) = \frac{\pi}{4Y}\left(1 + o(1)\right)
\]
except for those $x$ with $0 < \lVert x\rVert \ll Yt_{\phi}^{-\frac{2}{3}}$. In exactly the same way and with the same restriction, 
\[
S''(x) = \frac{\pi}{4Y}(2\pi T)^{2}\left(1 + o(1)\right),
\]
while $S'(x) = o(\frac{T}{Y})$. This then gives us 
\[
\mathcal{P}(\mathbf{v},x)  \sim \frac{1}{\pi} \frac{t_{\phi}}{Y}
\]
uniformly for all $x$ except those satisfying $0 < \lVert 2x\rVert \ll Yt_{\phi}^{-\frac{2}{3}}$, with $Y = o(t_{\phi}^{\frac{1}{3}})$ (the last condition may probably be relaxed). Integrating over all $-\half < x \leq \half$, using the estimate $\mathcal{P}(\mathbf{v},x) \ll t_{\phi}$ for the exceptional $x$'s, we obtain the expected number $\frac{1}{\pi} \frac{t_{\phi}}{Y}$ of zeros of $\phi(z)$ on such a horocycle (compare with Theorem \ref{TheoremThree} and Theorem \ref{TheoremATwo}).

%%%%%%%%%%%%%%%%%%%%%%%%%%%%%%%%%%%%%% section %%%%%%%%%%%%%%%%%%%%%%%%%%%%%%%%%%%%%%%%%%%%%%%%%%%%%%%%%%
\begin{section}{Nodal domains of odd eigenfunctions}
 
In this Appendix, we sketch the proof of the remark made in the Introduction regarding the nodal domains of an odd Maass form. Let $\phi$ be an odd, real Hecke-Maass eigenform, so that $\phi$  satisfies $\phi(-\overline{z}) = -\phi(z)$. Hence $\phi$ vanishes identically on $\delta$, which is therefore part of the nodal line of $\phi$, and so we consider the nodal domains within $\mathcal{F}^{+}$ in Figure 3. To prove the analogous lower bound for the number of nodal domains for $\phi$ which meet $\beta \subset \delta$ (these being the `inert' nodal domains) we seek sign-changes of the  normal derivative $\partial_{n}\phi$ of $\phi$ along $\delta$ (and consequently those of  $\phi_{x} = \frac{\partial{\phi}}{\partial{x}}$ along  $\delta_{1}$ and $\delta_{2}$). Each point in $\beta$ at which there is such a sign-change must be the end point of an interior nodal line and a topological argument similar to the one in Section 2 shows that the number of associated nodal domains is at least half of the number of such sign-changes. As before in order to produce sign changes in $\partial_{n}\phi$ we compare its  $L^\infty$-norm, its $L^2$-norm and also its meanvalue all restricted to $\beta$.

First, we give an extension of the subconvexity $L^\infty$ result of \cite{IS95}. This is easily proved using Bernstein's inequality for polynomials on the unit circle. In fact one has 

\begin{lem}\label{TheoremC1}
 Let $z$ belong in a compact set $D$ in $\mathbb{H}$. Then for any $k\geq 0$ and any Hecke-Maass cusp form $\phi$
\begin{align*}
\sup_{\substack z\in D}\Big{|}\frac{\partial^k}{\partial x^k} \phi(z)\Big{|} &\ll_{_{D,k}} t_{\phi}^{k} \sup_{\substack z\in D}|\phi(z)|,\cr
&\ll_{_{D,k, \ve}} t_{\phi}^{k+ \frac{5}{12}+\ve}.
\end{align*}
\end{lem}
\proof Write
\[
 \phi(z) = \sum_{n\in \mathbb{Z}}a_{n}(y)e(nx),
\]
and for any $N\geq 1$, put 
\[
 \phi(z;N) = \sum_{|n|\leq N}a_{n}(y)e(nx).
\]
For fixed $y$, we apply Bernstein's inequality to get
\[
 \sup_{\substack |x|\leq \half}|(\frac{\partial}{\partial x})^{k} \phi(z;N)| \ll_{_{D,k}} N^{k} \sup_{\substack |x|\leq\half}|\phi(z;N)|.
\]
Then choose $N=100\frac{t_{\phi}}{y}$ so that $\phi(z,N)$ and all its derivatives are well approximated by $\phi$ and its derivatives respectively, with an error that decays exponentially in $t$ so that the  lemma follows.

Next, we need to prove
\begin{equation}\label{eq:A3.1}
 \int_{\alpha_j}^{\alpha_{j+1}} \phi_{x}(iy)\ \d^{\times}y \ll_{\ve} t_{\phi}^{\half + \ve},
\end{equation}
which is the analogue of  \eqref{eq:e24}. This is obtained in exactly the same way as the discussion there, where \eqref{600} now becomes
\[
 4\pi\rho_{\phi}(1)\int_{-\infty}^{\infty}L(\half + it,\phi)\Gamma\big{(}\frac{\frac{3}{2} + i(t+t_{\phi})}{2}\big{)}\Gamma\big{(}\frac{\frac{3}{2}+i(t-t_{\phi})}{2}\big{)}\big{(}\frac{\alpha_{j+1}^{-1 -it} - \alpha_{j}^{-1 -it}}{1 + it}\big{)}\ \d t.
\]
Then, Stirling's formula together with the Lindelof Hypothesis gives the required result.

Finally, we need to prove, as in Section 6.2
\begin{equation}\label{eq:A3.2}
 \int_{\beta} \phi_{x}^{2}(s)\ \d s \gg_{\beta} t_{\phi}^{2}.
\end{equation}
The method is the same, with the non-diagonal terms offering no new problems. However, the diagonal term cannot be evaluated as in Section 6.2.3 because $\phi_{x}$ is not automorphic so that the diagonal is not an inner-product involving an incomplete Eisenstein series. Instead we obtain a lower-bound for the diagonal using positivity, asymptotic estimates for the Bessel functions and finally QUE by using Prop. \ref{shortsum}. The constant obtained this way is much smaller than that obtained for the even case, but does not affect the outcome if $a$ is chosen large enough.
\begin{rem} 
 To obtain the analogue of Section 6.2.2 we could have proceeded by using the gradient of $\phi$ instead of $\phi_x$ as follows
\[
 \int k(y)|\phi_{x}|^{2}\ \d y = \int k(y)\|{\bm\triangledown\phi}\|^{2}\ \d y = \int_{_\mathbb H} \sum_{h\in \mathbb{Z}}e(hx) \|y{\bm\triangledown\phi}\|^{2}\ \d \rm{A}.
\]
For small values of $h$  one may rewrite the last integral, after integrating by parts, as an inner product with incomplete Eisenstein and Poincare series against $\phi^{2}$, with a small error, so that QUE can be applied. For the large values of $h$, one has to expand the sums and analyse the resulting expression. This device will prove useful if \eqref{aa2} were to hold.
\end{rem}

Combining all these estimates finally lead to the same statements as those in Theorems \ref{TheoremI5} and \ref{TheoremI6}   but for $\phi$ an odd form..

\end{section}

%%%%%%%%%%%%%%%%%%%%%%%%%%%%%%%%%%%%%%%%%%%%%%%%%%%%%%%%%%%%%%%%%%%%%%%%%%%%%%%%%%%%%%%%%%%%%%%%%%%%%%%%%%
%%%%%%%%%%%%%%%%%%%%%%%%%%%%%%%%%%%%%%%%%%%%%%%%%%%%%%%%%%%%%%%%%%%%%%%%%%%%%%%%%%%%%%%%%%%%%%%%%%%%%%%%%
%%%%%%%%%%%%%%%%%%%%%%%%%%%%%%%%%%%%%%%%%%%%%%%%%%%%%%%%%%%%%%%%%%%%%%%%%%%%%%%%%%%%%%%%%%%%%%%%%%%%%%%%%%%

%\input appendix1.tex
%%%%%%%%%%%%%%%%%%%%%%%%%%%%%%%%%%%%% References %%%%%%%%%%%%%%%%%%%%%%%%%%%%%%%%%%%%%%%%%%%%%%%%%%%%%%%%

{\small
\vspace{40pt}
\noindent AMIT GHOSH, Department of Mathematics, Oklahoma State University, Stillwater, OK 74078, USA \hfill {\itshape E-mail address}: ghosh@math.okstate.edu 
\vspace{12pt}

\noindent ANDRE REZNIKOV, Department of Mathematics, Bar-Ilan University, Ramat Gan, 52900, ISRAEL \hfill {\itshape E-mail address}: reznikov@math.biu.ac.il
\vspace{12pt}

\noindent PETER SARNAK, School of Mathematics, Institute for Advanced Study; Department of Mathematics, Princeton University, Princeton, NJ 08540, USA \\ {\itshape E-mail address}: sarnak@math.ias.edu}

\end{document}